\documentclass[reqno,11pt,final]{amsart}

\usepackage[a4paper,left=25mm,right=25mm,top=30mm,bottom=30mm,marginpar=20mm]{geometry} 

\usepackage{amsmath}
\usepackage{amssymb}
\usepackage{amsthm}
\usepackage{amscd}
\usepackage[ansinew]{inputenc}
\usepackage[table]{xcolor}
\usepackage{color}
\usepackage[final]{graphicx}
\usepackage{tikz}
\usepackage{bbm}
\usepackage{appendix}

\usepackage{stmaryrd}
 \usetikzlibrary{fadings}

\usepackage{menukeys}

\numberwithin{figure}{section}

\usetikzlibrary{patterns} 
\usetikzlibrary{positioning}
\usetikzlibrary{calc}
\usetikzlibrary{arrows,shapes,backgrounds}
\usetikzlibrary{intersections} 

\renewcommand{\epsilon}{\varepsilon}

\numberwithin{equation}{section}

\newtheoremstyle{thmlemcorr}{10pt}{10pt}{\itshape}{}{\bfseries}{.}{10pt}{{\thmname{#1}\thmnumber{ #2}\thmnote{ (#3)}}}
\newtheoremstyle{thmlemcorr*}{10pt}{10pt}{\itshape}{}{\bfseries}{.}\newline{{\thmname{#1}\thmnumber{ #2}\thmnote{ (#3)}}}
\newtheoremstyle{defi}{10pt}{10pt}{\itshape}{}{\bfseries}{.}{10pt}{{\thmname{#1}\thmnumber{ #2}\thmnote{ (#3)}}}
\newtheoremstyle{remexample}{10pt}{10pt}{}{}{\bfseries}{.}{10pt}{{\thmname{#1}\thmnumber{ #2}\thmnote{ (#3)}}}
\newtheoremstyle{ass}{10pt}{10pt}{}{}{\bfseries}{.}{10pt}{{\thmname{#1}\thmnumber{ A#2}\thmnote{ (#3)}}}

\theoremstyle{thmlemcorr}
\newtheorem{theorem}{Theorem}
\numberwithin{theorem}{section}
\newtheorem{lemma}[theorem]{Lemma}
\newtheorem{corollary}[theorem]{Corollary}
\newtheorem{proposition}[theorem]{Proposition}

\theoremstyle{thmlemcorr*}
\newtheorem{theorem*}{Theorem}
\newtheorem{lemma*}[theorem]{Lemma}
\newtheorem{corollary*}[theorem]{Corollary}
\newtheorem{proposition*}[theorem]{Proposition}
\newtheorem{problem*}[theorem]{Problem}
\newtheorem{conjecture*}[theorem]{Conjecture}

\theoremstyle{defi}
\newtheorem{definition}[theorem]{Definition}

\theoremstyle{remexample}
\newtheorem{remark}[theorem]{Remark}
\newtheorem{example}[theorem]{Example}

\theoremstyle{ass}

\newcommand{\Acal}{\mathcal{A}}

\newcommand{\Mcal}{\mathcal{M}}

\newcommand{\Ibb}{\mathbb{I}}

\DeclareMathOperator{\diag}{diag}
\DeclareMathOperator{\charf}{\mathbbm{1}}

\DeclareMathOperator{\id}{id}

\DeclareMathOperator{\diam}{diam}

\DeclareMathOperator{\curl}{curl}
\DeclareMathOperator{\dist}{dist}

\newcommand{\ceilb}[1]{\big\lceil #1 \big\rceil}

\newcommand{\norm}[1]{\|#1\|}

\newcommand{\normb}[1]{\bigl\|#1\bigr\|}

\newcommand{\abs}[1]{|#1|}

 \newcommand{\double}[1]{{\llbracket{#1}\rrbracket}}

\newcommand{\dd}{\;\mathrm{d}}

\newcommand{\N}{\mathbb{N}}
\newcommand{\R}{\mathbb{R}}

\newcommand{\Z}{\mathbb{Z}}

\newcommand{\weakly}{\rightharpoonup}
\newcommand{\weaklystar}{\overset{*}\rightharpoonup}

\newcommand{\todown}{\downarrow}

\newcommand{\eps}{\epsilon}

\newcommand{\Ysoft}{Y_{\mathrm{soft}}}
\newcommand{\Ystiff}{Y_{\mathrm{stiff}}}

\newcommand{\Isoft}{I_{\mathrm{soft}}}
\newcommand{\Istiff}{I_{\mathrm{stiff}}}
\newcommand{\Wstiff}{W_{\mathrm{stiff}}}
\newcommand{\Wsoft}{W_{\mathrm{soft}}}

\DeclareMathOperator{\SO}{SO}

\DeclareMathOperator{\Whom}{W_\mathrm{hom}}

\def\Xint#1{\mathchoice 
{\XXint\displaystyle\textstyle{#1}}%
{\XXint\textstyle\scriptstyle{#1}}%
{\XXint\scriptstyle\scriptscriptstyle{#1}}%
{\XXint\scriptscriptstyle\scriptscriptstyle{#1}}%
\!\int} 
\def\XXint#1#2#3{{\setbox0=\hbox{$#1{#2#3}{\int}$} 
\vcenter{\hbox{$#2#3$}}\kern-.5\wd0}} 
\def\dashint{\,\Xint-}

\title[Asymptotic rigidity of layered structures and homogenization]{Asymptotic rigidity of layered structures and its\\ application in homogenization theory}

\author{Fabian Christowiak}
\address{Fakult\"at f\"ur Mathematik, Universit\"at Regensburg, 93040 Regensburg, Germany}
\email{Fabian.Christowiak@mathematik.uni-regensburg.de}

\author{Carolin Kreisbeck}
\address{Mathematisch Instituut, Universiteit Utrecht, Postbus 80010, 3508 TA Utrecht, The Netherlands}
\email{c.kreisbeck@uu.nl}

\begin{document}


\maketitle
\thispagestyle{empty}

 \begin{abstract} 
 
In the context of elasticity theory, rigidity theorems allow to derive global properties of a deformation from local ones. 
This paper presents a new asymptotic version of rigidity, applicable to elastic bodies with sufficiently stiff components arranged into fine parallel layers. We show that strict global constraints of anisotropic nature occur in the limit of vanishing layer thickness, and give a characterization of the class of effective deformations.  
The optimality of the scaling relation between layer thickness and stiffness is confirmed by suitable bending constructions.
Beyond its theoretical interest, this result constitutes a key ingredient for the homogenization of variational problems modeling high-contrast bilayered composite materials, where the common assumption of strict inclusion of one phase in the other is clearly not satisfied. We study a model inspired by hyperelasticity via $\Gamma$-convergence, for which we are able to give an explicit representation of the homogenized limit problem. It turns out to be of integral form with its density corresponding to a cell formula.

\vspace{8pt}                               
\vspace{8pt}

 
 
 \vspace{8pt}
 
 \noindent\textsc{Date:} \today.
 \end{abstract}

\section{Introduction}

Rigidity is a prevalent concept in different areas of mathematics. Generally speaking, it refers to powerful statements that
allow to draw far-reaching conclusions from seemingly little information, such as deducing global properties of a function from local ones.
A classical result along these lines is often referred to as Liouville's theorem on geometric rigidity, see e.g.~\cite{IwM98}. It says that every smooth local isometry of a domain corresponds to a rigid body motion. A generalization to the Sobolev setting is due to Reshetnyak~\cite{Res67}, precisely, if $u\in W^{1,p}(\Omega;\R^n)$ with $\Omega\subset \R^n$ a bounded Lipschitz domain and $1<p<\infty$ satisfies  
\begin{align}\label{exact_diffinclusion}
\nabla u \in SO(n) 
\end{align}
pointwise almost everywhere in $\Omega$, then $u$ is harmonic and 
\begin{align}\label{rigid_body_motion}
u(x)=Rx+b \ \ \text{ for $x\in\Omega$ with $R\in SO(n)$ and $b\in \R^n$.}
\end{align}
It is not hard to see that if connectedness of the domain fails, then global rigidity is no longer true, as different connected components can then be rotated and translated individually. 

Yet, for a domain that has several rigid components arranged into very fine parallel layers (see Fig.~\ref{fig:intro}), global geometric constraints of anisotropic nature occur in the limit of vanishing layer thickness. Since these restrictions become prominent only after a limit passage, we speak of asymptotic rigidity of layered structures.
A first rigorous result in this direction can be found in~\cite{ChK17} for the special case $n=2$ and $p=2$. There, it was proven that, under the assumption of local volume preservation and up to global rotations, only shear deformations aligned with the orientation of the layers can occur as effective deformations.

In this paper, we extend the result of~\cite{ChK17} to arbitrary dimensions $n\geq 2$ and general $1<p<\infty$, and more significantly, relax the assumption of rigid layers by requiring only sufficient stiffness (see Theorem~\ref{prop:rigidity}). Formally, this corresponds to replacing the exact differential inclusion~\eqref{exact_diffinclusion} by an approximate one, very much like the quantitative rigidity estimate by  Friesecke, James \& M\"uller~\cite[Theorem~3.1]{FJM02} 
 generalizes Reshetnyak's theorem. 
The paper~\cite{FJM02} has initiated increased interest in rigidity and its quantification over the last years, especially among  
analysts working on variational methods with applications in materials science. 
For instance, a quantitative version of piecewise rigidity for $SBV$-functions \cite{CGP07} was established in~\cite{FrS15}, and there is recent work on rigidity of conformal maps~\cite{FaZ05}, of non-gradient fields~\cite{MSZ14} or in the non-Euclidean setting~\cite{LeP11}.

To be more precise about our results, some notation on the geometry of bi-layered structures is needed. 
Throughout the manuscript, let $\Omega\subset \R^n$ with $n\geq 2$ be a bounded Lipschitz domain, $\lambda\in (0,1)$, and $Y=(0,1]^n$ the periodicity cell. We set 
\begin{align*}
\Ysoft= (0,1]^{n-1}\times (0,\lambda) \qquad \text{and} \qquad \Ystiff =  Y\setminus \Ysoft,
\end{align*}
cf.~Fig.~\ref{fig:intro}. Without further mentioning, $\Ysoft$ and $\Ystiff$ are identified with their $Y$-periodic extensions. To describe the thickness of two neighboring layers, we introduce a parameter $\eps>0$, which is supposed to be small and captures the length scale of the heterogeneities. The disjoint sets $\eps \Ystiff\cap\Omega$ and $\eps \Ysoft\cap\Omega$ partition the domain $\Omega$ into two phases of alternating layers. 
Notice that the parameter $\lambda$ stands for the relative thickness of the softer components.   

\begin{figure}[ht]
\begin{tikzpicture}
\begin{scope}[scale = 1.4]
\def\material{(0.5,0.3) to[out=320, in=180] (2,0) to (4,0) to[out=360,in=220] (5,0.3) to[out=40,in=320]  (5,1.2) to[out=140, in=310](3,1.8) to[out=130,in=220] (3,2.3) to[out=40,in=180] (3.6,2.5) to[out=0, in= 350] (3.6,3) to[out=170, in= 0] (2,3.1) to (1.5,3.1) to[out=180,in=70] (0.5,2.4) to[out=250, in=130]  cycle}

 	\foreach \z in {0,0.5,...,3}
        \draw[fill=gray!30] ($(-0.5,\z) + (0,0.25) - (0,0.25)$) rectangle ($(5.4,\z)+(0,0.5)- (0,0.25)$);
        \fill[fill = white,thick, scale = 1] \material (-0.6,-0.5) rectangle (5.6,4.6);
        \draw[ thick, black, fill =white, fill opacity=0, scale = 1] \material;

	\def \xs{1.2} 
	\def \ys{0.75}

        \draw[thick] ($(\xs,\ys) + (0,0)$) rectangle ($(\xs,\ys)+(0.5,0.5)$);

	\def \xl{-3}
	\def \yl {1.6}

        \draw[dashed,black]  ($(\xs,\ys) + (0,0)$) -- ($(\xl,\yl)+(0,0)$);
        \draw[dashed, black] ($(\xs,\ys) + (0.5,0)$) -- ($(\xl,\yl)+(2,0)$);
        \draw[dashed, black] ($(\xs,\ys) + (0.5,0.5)$) -- ($(\xl,\yl)+(2,2)$);
        \draw[dashed, black] ($(\xs,\ys) + (0,0.5)$) -- ($(\xl,\yl)+(0,2)$);

        \draw[white, fill = white]  ($(\xl,\yl)+(0,0)$) rectangle ($(\xl,\yl)+(2,2)$);
        \draw[rectangle, black, fill= gray!30] ($(\xl,\yl)+(0,1)$) rectangle ($(\xl,\yl)+(2,2)$);
        \draw[black] (\xl,\yl) rectangle ($(\xl,\yl)+(2,2)$);
        
        \draw ($(\xl,\yl)+(-0.1,2.1)+(0.45,-0.35)$) node {$\Ystiff$};
	\draw ($(\xl,\yl)+(-0.1,1.1)+(0.45,-0.35)$) node {$\Ysoft$};
	\draw[<->] ($(\xl,\yl)+(0,0)+(-0.15,0)$) -- ($(\xl,\yl)+(0,1)+(-0.15,0)$) ;
	\draw ($(\xl,\yl)+(0,0.5)+(-0.35,0)$) node {$\lambda$};
	\draw ($(\xl, \yl)+(6.4,1.8)$) node {$\Omega\subset \R^n$};
		\draw[<->] ($(\xl,\yl)+(0,0)+(-0.15,0)$) -- ($(\xl,\yl)+(0,1)+(-0.15,0)$) ;	
	\draw ($(\xl, \yl) +(0.2,2.35)$) node {$Y=(0,1]^n$};
	\draw[->, very thick] ($(\xl,\yl)+(8.0,0.5)+(-0.15,0)$) -- ($(\xl,\yl)+(8.0,1.5)+(-0.15,0)$) ;	
	\draw ($(\xl, \yl) + (8.1,1)$) node {$e_n$};

	\def \xe{2.8}
	\def \ye{2.5}
	\draw[<->] ($(\xe,\ye) + (2.5,-2.0) + (0.1,0)$) -- ($(\xe,\ye)+(2.5,-1.5)+(0.1,0)$);
	\draw ($(\xe,\ye) + (2.5,-2.0) + (0.25,0.25)$) node {$\epsilon$};

\draw \material;
\end{scope}
\end{tikzpicture}
\caption{Illustration of bi-layered structure with stiff (gray) and softer (white) components and periodicity cell $Y$, subdivided into $\Ysoft$ and $\Ystiff$.}\label{fig:intro}
\end{figure}
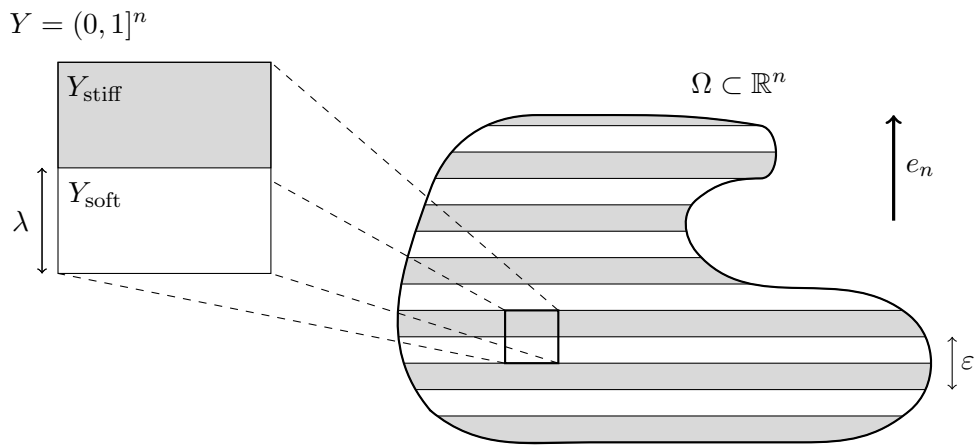

Under certain technical assumptions on the domain, which are specified in Definitions~\ref{def:connected} and~\ref{def:flat}, we obtain as our first main result a characterization for the asymptotic behavior of sequences of functions on $\Omega$ whose gradients are increasingly close to $SO(n)$ in $\eps \Ystiff\cap \Omega$ as $\eps\to 0$. 

\begin{theorem}\label{prop:rigidity} 
Let $\Omega\subset \R^n$ be a bounded, flat and cross-section connected Lipschitz domain and $1<p<\infty$.

i) Suppose that $(u_\eps)_\eps \subset W^{1,p}(\Omega;\R^n)$ is such that 
\begin{align}\label{approx_diffinclusion_intro}
 \int_{\eps\Ystiff\cap \Omega} \dist^p\big(\nabla u_\eps, SO(n)\big) \dd x \leq C\eps^\alpha 
 \end{align} for all $\eps >0$ with $\alpha\geq 0$ and a constant $C>0$. If $\alpha>p$ and $u_\eps\weakly u$ in $W^{1,p}(\Omega;\R^n)$ for some $u\in W^{1,p}(\Omega;\R^n)$, then 
\begin{align}\label{representation_nablau_intro}
u(x)= R(x)x + b(x), \quad x \in \Omega,
\end{align}
with $R\in W^{1,p}(\Omega;SO(n))$ and $b\in W^{1,p}(\Omega;\R^n)$ such that $\partial_i R =0$ and $\partial_i b =0$ for $i=1, \ldots, (n-1)$.

ii) If $u\in W^{1,p}(\Omega;\R^n)$ is of the form~\eqref{representation_nablau_intro}, then there exists a sequence $(u_\eps)_\eps \subset W^{1,p}(\Omega;\R^n)$ such that $u_\eps\weakly u$ in $W^{1,p}(\Omega;\R^n)$ and~$\nabla u_\eps\in SO(n)$ a.e.~in $\eps\Ystiff\cap \Omega$ for every $\eps>0$. 
\end{theorem}

One observes that~\eqref{representation_nablau_intro} resembles~\eqref{rigid_body_motion}, just that now $R$ will in general not be constant, but depends on the $x_n$-variable, and hence, varies in the direction orthogonal to the layers. This condition can be considered the result of a non-trivial interplay between the effects of rigidity and anisotropy. 

The proof of Theorem~\ref{prop:rigidity}\,$i)$ consists of three main steps, the layerwise approximation of each $u_\eps$ by rigid body motions, a compactness argument for the resulting one-dimensional auxiliary functions of piecewise constant rotations, and a limit representation argument. Regarding its overall structure, the reasoning organized similarly to~\cite[Proposition~2.1]{ChK17}. 
Technically, however, the transition from exact to the approximate differential inclusions requires two substantial changes, which make the arguments more involved than in~\cite{ChK17}.
Instead of Reshetnyak's theorem, we apply the quantitative rigidity estimate on each layer, and the Fr\'echet-Kolmogorov compactness result (see~Lemma~\ref{theo:FrechetKolmogorov}) is used as a refinement of Helly's selection principle. 

Proving the second part of Theorem~\ref{prop:rigidity} involves the explicit construction of an approximating sequence $(u_\eps)_\eps$ with the desired properties. To this end, we critically exploit the special structure of $u$ as in~\eqref{representation_nablau_intro}, which features a splitting of the $x_n$-variable from the remaining ones, so that $u$ has essentially the character of a one-dimensional function. 

\begin{remark} 
 a) The gradient of $u$ as in~\eqref{representation_nablau_intro} takes the form
\begin{align}\label{form_nablau}
\nabla u = R + (\partial_n R) x \otimes e_n + \partial_n b\otimes e_n,
\end{align}
which necessarily requires that $(\nabla u)e_i = Re_i$ for all $i=1, \ldots, n-1$. \\[-0.2cm]

b) We point out that the scaling regime $\alpha>p$, which quantifies the relation between thickness and stiffness of the layers, is optimal for Theorem~\ref{prop:rigidity}\,$i)$. 
As shown in Section~\ref{sec:overview&examples}, asymptotic rigidity of layered structures fails for $\alpha\leq p$.  
We provide explicit examples inspired by bending deformations, for which the limit maps $u$ are such that $\partial_1 u$ depends non-trivially on $x_1$ or $\partial_1 u$ is not normed to one. 

Note that the two extreme cases $\alpha=0$ and ``$\alpha=\infty$'' (formal for $\eps^\alpha=0$) in~\eqref{approx_diffinclusion} correspond the situations of the stiff layers being actually soft or fully rigid, respectively. \\[-0.2cm]

c) Theorem~\ref{prop:rigidity} can be extended in different directions. One generalization concerns a $(p,q)$-version Theorem~\ref{prop:rigidity}\,$i)$. Indeed, if the exponent $p$ in~\eqref{approx_diffinclusion} is replaced by $q\in (1, \infty)$ the statement remains valid provided that $\alpha>q$. In this more general setting, we can let $1\leq p<\infty$.  The only modification in the case $p=1$ is that $R$ and $b$ will be $BV$-functions. We refer to Remark~\ref{rem:rigidity}\,a) and Remark~\ref{rem:necessity}\,b) for more details. 
Moreover, as mentioned in Remark~\ref{rem:necessity}\,c), asymptotic rigidity in the sense of Theorem~\ref{prop:rigidity}\,$i)$ still holds if the relative thickness of the stiff layers depend on $\eps$, being much larger than $\eps^{\frac{\alpha}{p}-1}$.  
For a comment on reduced assumptions for the domain $\Omega$, see Remark~\ref{rem:domain} as well as~Theorem~\ref{theo:rigidity_general}. \\[-0.2cm]

d) If one requires additionally in Theorem~\ref{prop:rigidity} that the limit function $u$ is locally volume preserving, that is $u \in W^{1,r}(\Omega;\R^n)$ for $r \geq n$ with $\det \nabla u = 1$ a.e.~in $\Omega$, 
then $Re_n$ is constant, see Corollary~\ref{cor:asymptotic_rigidity}. 
In the two-dimensional setting with $n=2$, this implies that $R$ is constant, and one can think of $u$ as horizontal shear deformation up to global rotations, cf.~also~\cite[Proposition~2.1]{ChK17}.
\end{remark}

From the viewpoint of applications in materials science, Theorem~\ref{prop:rigidity} identifies characteristics of macroscopically attainable deformations of bi-layered high-contrast composite materials. This observation constitutes an important step towards a rigorous characterization of their effective behavior via homogenization. 
Indeed, we will discuss in the following how asymptotic rigidity of layered structures serves as the basis for solving a relevant class of homogenization problems in the context of hyperelasticity. 

In the 1970s, the Italian school around De Giorgi established the concept of $\Gamma$-convergence~\cite{Gio75, DeF75} (see also see~\cite{Bra05, Dal93} for a comprehensive introduction), bwhich has been used successfully among others in homogenization theory to bridge between microscopic and macroscopic scales.  
It is a natural notion for variational convergence, i.e.~ limit passages in parameter-dependent minimization problems.
The key property is that if a sequence of energy functionals $\Gamma$-converges to a limit functional, this automatically implies convergence of the corresponding (almost) minimizers. 

By now classcial homogenization results via $\Gamma$-convergence include the papers by Marcellini~\cite{Mar78} in the convex setting, as well as the first work in the non-convex case with standard $p$-growth by  M\"uller~\cite{Mul87} and Braides~\cite{Bra85}.
Within multiscale analysis, which comprises homogenization and relaxation theory, variational problems with non-convex pointwise or differential constraints are known to be technically challenging, cf.~\cite{BFL00, CDD06, CoD15, DuG16, KrK16}. 
Despite recent partial progress towards attacking the issue of localization, i.e.~proving that limit functionals preserve integral form, with different methods, e.g.~\cite{CoD15, DuG16, KRW15, Pra17},
there are still general open questions that cannot be worked out with existing tools.
In this article, we investigate homogenization problems subject to a special type of approximate differential inclusion constraint, which do not satisfy standard assumptions and therefore require a tailored approach.

Let $\alpha>0$ and $p\in (1, \infty)$. Consider for each $\eps>0$ the integral functional $E_\eps$ defined for $u\in W^{1,p}(\Omega;\R^n)$ by
\begin{align*}
E_\eps(u) =  \int_{\eps\Ystiff\cap\Omega} \frac{1}{\eps^\alpha} \dist^p(\nabla u, SO(n)) \dd{x} + \int_{\eps\Ysoft\cap \Omega} \Wsoft(\nabla u)\dd{x}
\end{align*}
with an integrand $\Wsoft:\R^{n\times n}\to \R$, which is in general not convex or quasiconvex.
These functionals model the elastic energy of a layered composite. The first term with diverging elastic constants, scaling like $\eps^{-\alpha}$, is the contribution of the stiff components and the second term is associated with the softer components.  

In the regime $\alpha>p$, we show that the $\Gamma$-limit of $(E_\eps)_\eps$ as $\eps\to 0$ with respect to strong convergence in $L^p(\Omega;\R^n)$, or equivalently weak convergence in $W^{1,p}(\Omega;\R^n)$, exists and determine a characterizing formula. The required technical assumptions on the geometry of $\Omega$ are those of Definitions~\ref{def:connected} and~\ref{def:flat} and the density $\Wsoft$ is supposed to satisfy $(H1)$-$(H3)$ below.
In fact, the $\Gamma$-limit has integral form, is subject to the constraints on the admissible macroscopic deformations induced by asymptotic rigidity (cf.~Theorem~\ref{prop:rigidity}), and can be expressed purely in terms of the energy density $\Wsoft$ and the relative thickness $\lambda$ of the softer layers.  Precisely,
\begin{align}\label{hom_formula}
E_{\rm hom}(u):= \Gamma\text{-}\lim_{\eps\to 0} E_\eps(u) = \int_\Omega \lambda\Wsoft^{\rm qc}\big(\tfrac{1}{\lambda} (\nabla u-(1-\lambda)R)\big)\dd{x}
\end{align}
for all $u$ of the form~\eqref{representation_nablau_intro}, and $E_{\rm hom}(u)=\infty$ otherwise. Here, $\Wsoft^{\rm qc}$ stands for the quasiconvex envelope of $\Wsoft$, for background information on generalized notions of convexity and relaxations, see e.g.~\cite{Dac08}.

Next, we collect a few remarks to put the above mentioned homogenization result - a detailed formulation of the full version is  given in Theorem~\ref{theo:hom} - in context with related work in the literature. 

\begin{remark}
a) General theorems on homogenization tend to be rather implicit in the sense that they involve (multi)cell formulas (e.g.~\cite{Bra85, Mul87}), which again require to solve infinite dimensional minimization problems. In contrast, the $\Gamma$-limit in ~\eqref{hom_formula} is clearly explicit with regards to the macroscopic effect of the heterogeneities. If the relaxation of the softer components, or in other words, the quasiconvexification of $\Wsoft$, is known, the representation of the homogenized energy density becomes even fully explicit. To illustrate the latter, we discuss the prototypical example of the Saint-Venant Kirchhoff stored energy function in Example~\ref{ex:SaintVenantKirchhoff}. \\[-0.2cm]

b) As we demonstrate in~Remark~\ref{rem:cell}, the density in~\eqref{hom_formula} coincides with a single-cell formula. This indicates that microstructures ranging over multiple cells (or layers) are not energetically favorable, in contrast with the general theory.
Indeed, by M\"uller's well-known counterexample~\cite{Mul87} (see also~\cite{BaG10}  for another example), multi-cell formulas are necessary to describe homogenized limits of general non-convex problems. The recent paper~\cite{NES18} refines this observation by showing that a single-cell formula is sufficient in a neighborhood of rotations, though.\\[-0.2cm]

c) Rigorous statements about variational homogenization of elastic high-contrast materials seem to be restricted
to the geometric assumption of strict inclusions, with either the stiff phases embedded in the softer \cite{BrG95, ChC12, DuG16}, or the other way around~\cite{CCN17}. To our knowledge, Theorem~\ref{theo:hom} provides the first characterization 
 in the setting of non-inclusion type heterogeneities. Their layered geometry is reflected macroscopically in the anisotropic structure of $E_{\rm hom}$. \\[-0.2cm]

d) Asymptotic rigidity of layered structures as a technical tool is not limited to the homogenization problem in Theorem~\ref{theo:hom}, but can be used in different contexts. It is for instance an important ingredient in the asymptotic analysis of models for layered materials in finite crystal plasticity, see~\cite{ChK17} and~\cite[Chapter~5, 6]{Chr18} for more details. Let us mention that the second reference makes a first step towards extending the results to the context of  stochastic homogenization, by assuming a random distribution of the layer thickness.  
\end{remark}

We conclude the introduction with a few words about the proof of Theorem~\ref{theo:hom}, focussing on the main ideas and technical challenges.
The construction of a recovery sequence for affine limit maps (Step~1) is based on laminates made of rotations and shear components (cf.~\cite[Section~4]{ChK17}), which we augment with suitable perturbations on the softer layers. 
The harder part is the case of general limits (Step~3). Recall that Theorem~\ref{prop:rigidity}\,$ii)$ provides an admissible approximating sequence for any possible limit map as in~\eqref{representation_nablau_intro}. However, these sequences fail to be energetically optimal in general. To remedy this problem, we localize by piecewise constant approximation of the limit functions, which can be done in a constraint preserving way due to the essentially one-dimensional character of the representation in~\eqref{representation_nablau_intro} (see also~\eqref{form_nablau}). Finally, we determine locally optimal microstructures as in the affine case and glue them onto the sequence from Theorem~\ref{prop:rigidity}\,$ii)$ in the softer parts. This construction is sufficient to recover the energy.

In essence, our reasoning for the liminf-inequality (Steps~2 and 4) comes down to using Theorem~\ref{prop:rigidity}\,$i)$ and to applying Jensen's inequality twice, first to obtain a lower bound energy estimate on each softer layer and then, in the optimization process over the entirety of layers. Besides, we employ the properties of Null-Lagrangians. 
The presented arguments rely strongly on the hypothesis that~$\Wsoft^{\rm qc}$ is polyconvex (referred to as $(H1)$), meaning that the quasiconvex envelope can be written as a convex function of the vector of minors, or in other words, that the quasiconvex envelope coincides with the polyconvex one. 
Notice that the same assumption can be found e.g.~in \cite{CoD15} in the context of relaxation problems with constraints on the determinant.  

Dropping $(H1)$ appears to be a non-trivial task. On a technical level, if the Jensen's inequalities mentioned above were to be replaced straight away by the related formulas defining quasiconvexity (see~\eqref{qc}), this would require careful cut-off arguments at the boundaries. In the stiff layers, though, cut-off conflicts with the rigidity constraints and difficulties may arise from non-local effects due to interaction between different layers. 
Hence, it remains an open question to understand whether removing $(H1)$ from the list of assumptions makes the $\Gamma$-limit $E_{\rm hom}$ in~\eqref{hom_formula} (if existent) smaller. Or in more intuitive terms, can the energy be further reduced by oscillations of the rotation matrices and long range effects over multiple layers? \\

\textbf{Structure of the article.} 
This paper is organized in five sections. In the subsequent Section~2, we discuss a range of explicit bending examples, which illustrate softer macroscopic behavior in the regimes $0 < \alpha \leq p$ and establish in particular the optimality of the condition $\alpha > p$ in Theorem~\ref{prop:rigidity}\,$i)$.  
Sections~3 and~4 contain the proofs of the asymptotic rigidity result formulated in Theorem~\ref{prop:rigidity}. 
In Section~3, we prove a generalization of the necessity part $i)$ 
as well as Corollary~\ref{cor:asymptotic_rigidity}, followed by a more detailed discussion on the geometric assumptions on the domain $\Omega$. Section~4 proceeds with the proof of the sufficiency statement $ii)$ of Theorem~\ref{prop:rigidity}.
In the final Section~5, we state our second main result on homogenization via $\Gamma$-convergence, that is Theorem~\ref{theo:hom}. For its proof, both parts of Theorem~\ref{prop:rigidity} are key. We conclude by relating the homogenization formula of~\eqref{hom_formula} to the cell formula as it occurs in models of composites with rigid layers. 
The appendix provides two technical auxiliary results in form of a specialized reverse Poincar\'e type inequality and a lemma on locally one-dimensional functions.\\

\textbf{Notation.}  
The standard unit vectors in $\R^n$ are denoted by $e_1, \dots, e_n$. For the Euclidean inner product between two vectors $a, b\in \R^n$ we write $a\cdot b$. Moreover,  let $a\otimes b=a b^T\in \R^{n\times n}$ for $a, b\in \R^n$, and set $a^\perp = (-a_n, a_2, \ldots, a_{n-1}, a_1)^T\in \R^n$ for $a\in \R^n$, which generalizes the usual notation for perpendicular vectors in two dimensions. The Frobenius norm of $A \in \R^{n\times n}$ is given by $|A|=\sqrt{AA^T}$. Our notation for block diagonal matrices is $A=\diag(A_1, A_2, \ldots, A_m)\in \R^{n\times n}$ with $A_i\in\R^{n_i\times n_i}$ and $\sum_{i=1}^m n_i=n$. 
In the following, we will often split up $a\in \R^n$ as $a=(a', a_n)$, where $a'=(a_1, \ldots, a_{n-1})$. For a matrix $A\in\R^{n\times n}$ a similar splitting into its columns is used, that is $A=(A'|Ae_n)$ with $A'\in \R^{n\times (n-1)}$.
For $t\in \R$, the expressions $\lfloor t \rfloor$ and $\lceil t\rceil$ stand for the largest integer smaller and smallest integer larger than $t$, respectively. 

By a domain $\Omega\subset \R^n$ we mean an open, connected subset of $\R^n$. An open cuboid is the Cartesian product $Q=(a_1, b_1)\times \ldots\times (a_n, b_n) =: \times_i (a_i, b_i)\subset \R^n$ with $a_i, b_i\in \R$ and $a_i<b_i$ for $i=1, \ldots, n$. Hence for us, cuboids will always be oriented along the coordinate axes. 
Furthermore, $\mathbbm{1}_E$ and $\chi_E$ are the indicator and characteristic function corresponding to a subset $E\subset \R^n$,
i.e., $\mathbbm{1}_E(x) =  1$ and $\chi_E(x)=0$ if $x\in E$, and $\mathbbm{1}_E(x)=0$ and $\chi_E(x)=\infty$ if $x\notin E$. 
For a measurable set $U$ and an integrable function $f:U\to\R^m$, let $\dashint_{U} f \dd{x} := \frac{1}{|U|} \int_U f\dd x$. 

We use the common notation for Lebesgue and Sobolev spaces, as well as for function spaces of continuously differentiable functions. By $L^p_0(\Omega;\R^m)$, we denote the space of functions in $L^p(\Omega;\R^m)$ with the property that their mean value vanishes. Periodic boundary condition are indicated by a lower case $\#$, for example in $W^{1,p}_{\#}(Y;\R^m)$. 

The distributional derivative of a function $f\in L^1_{\rm loc}(\Omega;\R^m)$ is denoted by $Df$, for partial derivatives in $e_i$-direction we write $\partial_i u$. Moreover, $Df=(D'f|\partial_n f)$ with $D'f=(\partial_1 f| \ldots|\partial_{n-1} f)$. If $f:\Omega\to \R^m$ is classically or weakly differentiable, we denote the (weak) gradient of $f$ by $\nabla f$. Here again, one has the splitting $\nabla f=(\nabla'f|\partial_n f)$ with $\nabla' f=(\partial_1 f|\ldots|\partial_{n-1}f)$. 
In case $f:J\to \R^m$ is a one-dimensional function with $J\subset \R$ an open interval, we simply write $f'$ for the derivative of $f$. 

Convergence of a sequence $(u_\eps)_\eps$ as $\eps\to 0$ means that $(u_{\eps_j})_j$ converges as $j\to \infty$ for any subsequence $\eps_j\todown 0$. 
Note finally the use of generic constants, mostly denoted by $c$ or $C$, which may vary from line to line without change in notation. 


\section{Optimality of the scaling regimes}\label{sec:overview&examples}

While for $\alpha =0$ in~\eqref{approx_diffinclusion} the class of effective deformations with finite energy comprises arbitrary Sobolev maps with vanishing mean value, the material response in the case ``$\alpha=\infty$'' is rather rigid. This raises the natural question up to which value of $\alpha$ softer material response can be encountered. 
In this section, we discuss four examples of macroscopically attainable deformations. They show that Theorem \ref{prop:rigidity}  and Corollary~\ref{cor:asymptotic_rigidity} fail for small elastic constants in the regime $\alpha\leq p$, and illustrate the effect of (local) volume preservation. For simplicity, we assume throughout this section that $\Omega\subset \R^n$ is the unit cube, i.e.~$\Omega=(0,1)^n$.

The idea behind the first two constructions for $\alpha=p$ is to bend the individual stiffer layers, first uniformly in Example~\ref{ex1}, and then in a locally volume-preserving way inspired by the bending of a stack of paper in Example~\ref{ex2}. Example~\ref{ex3} is based on a wrinkling construction for the individual layers, and shows that compression in layer direction is possible for $\alpha\in (0,p)$. Finally, we look into the effect of the local volume condition for $\alpha>p$ in Example~\ref{ex4}. 

The calculations behind these examples share a common structure and are all based on the following auxiliary result. We deliberately keep its formulation slightly more general than actually needed in the following. This facilitates the construction of an even larger variety of explicit deformations and yields immediate insight into their asymptotic properties.  As regards notation, for $\eps>0$ and $t\in \R$ we let $[t]_\eps = \eps \ceilb{\frac{t}{\eps}}-\eps + \tfrac{1+ \lambda}{2}\eps$ and write 
\begin{align*}
\double{x_\eps} =  (x_1,\ldots, x_{n-1}, [x_n]_\eps)
\quad\text{for $x\in \R^n$,}
\end{align*}
where in consideration of the layered structure $\eps\Ystiff$, $\double{x_\eps}$ corresponds to the orthogonal projection of $x$ onto the midsection of the closest stiff layer lying above.  
\begin{lemma} \label{lem: general form of gradient} 
Let $\overline{Q} = [0,1]^{n-1}\times [0, 2]$ and $1<p<\infty$. For $\eps\in (0,1)$, let $f_\eps \in C^2(\overline{Q}; \R^n)$ be such that $\abs{\partial_1 f_\eps}=1$, $\partial_1 f_\eps\in {\rm span}\{e_1, e_n\}$ and $\partial_i f_\eps = e_i$ for $i=2, \dots, n-1$, and define a Lipschitz function $u_\eps: \Omega\to \R^n$ by 
\begin{align}\label{defueps}
u_\eps(x) = f_\eps (\double{x}_\eps)+ (x_n - [x_n]_\eps) \partial_1 f_\eps^\perp (\double{x}_\eps) \quad \text{for $x\in \eps \Ystiff \cap \Omega,$}
\end{align}
and by linear interpolation in $e_n$-direction in $\eps \Ysoft\cap\Omega$. 

Then for any $\eps\in (0,1)$,
\begin{align}\label{example_energyestimate}
\int_{\eps \Ystiff\cap \Omega} \dist^p(\nabla u_\eps, SO(n))\dd{x} \leq  2^p \eps^p \norm{ \partial_{11}^2 f_\eps}_{L^\infty(\overline{Q};\R^n)}^p.
\end{align}
Moreover, if  $\lim_{\eps\to 0} \eps \norm{\nabla^2 f_\eps}_{L^\infty(\overline{Q};\R^{n\times n\times n})}=0$ and if there is $F\in L^p(\overline{Q};\R^{n\times n})$ such that either
\begin{itemize}
\item[$i)$] $\nabla f_\eps \to F$ in $L^p(\overline{Q};\R^{n\times n})$ as $\eps\to 0$, or 
\item[$ii)$] $\nabla f_\eps \weakly F$ in $L^p(\overline{Q};\R^{n\times n})$ as $\eps\to 0$  and $\partial_{n}(\nabla  f_\eps) = 0$ for all $\eps\in (0,1)$,
\end{itemize}
then 
\begin{align}\label{nablaueps}
\nabla u_\eps \weakly F \qquad \text{in $L^p(\Omega;\R^{n\times n})$.}
\end{align}
\end{lemma}

\begin{proof} 
We observe first that $\partial_i f_\eps = e_i$ for $i=2, \ldots, n-1$ and $\eps\in (0,1)$, hence, also $Fe_i=e_i$, both in case $i)$ and $ii)$. 
By definition, the functions $u_\eps$ are continuously differentiable on the connected components of $\eps \Ystiff\cap \Omega$ and $\eps \Ysoft\cap\Omega$. Then
\begin{align}\label{nablaeps_rig}
\nabla u_\eps(x) = \partial_1 f_\eps(\double{x}_\eps) \otimes e_1 + (x_n-[x_n]_\eps)\partial_{11}^2f^\perp_\eps (\double{x}_\eps) \otimes e_1 + \textstyle\sum_{i=2}^{n-1} e_i\otimes e_i + \partial_1 f^\perp_\eps(\double{x}_\eps) \otimes e_n
\end{align}
for $x\in \eps\Ystiff\cap \Omega$, and straight-forward calculation yields the gradients for $x\in \eps\Ysoft\cap\Omega$, 
 \begin{align*}
\nabla u_\eps(x) &= \big(\partial_1 f_\eps(\double{x}_\eps -\eps e_n) + \tfrac{1-\lambda}{2}\eps \partial_{11}^2 f^\perp_\eps(\double{x}_\eps -\eps e_n)\big) \otimes e_1 \\
&\hspace{1cm} + \tfrac{1}{\lambda\eps}\big(x_n - \ceilb{\tfrac{x_n}{\eps}}\eps + \eps\big)\big(\partial_1 f_\eps(\double{x}_\eps) - \partial_{1} f_\eps(\double{x}_\eps-\eps e_n) \big) \otimes e_1 \\
&\hspace{1cm}
-\tfrac{1-\lambda}{2\lambda} \big(x_n- \ceilb{\tfrac{x_n}{\eps}}\eps + \eps\big) \big(\partial_{11}^2 f^\perp_\eps(\double{x}_\eps) +\partial_{11}^2 f^\perp_\eps (\double{x}_\eps-\eps e_n) \big) \otimes e_1 \\
& \hspace{1cm} + \textstyle \sum_{i=2}^{n-1} e_i\otimes e_i + \tfrac{1}{\lambda\eps}\big(f_\eps(\double{x}_\eps) -f_\eps(\double{x}_\eps-\eps e_n) \big) \otimes e_n \\
&\hspace{1cm}
-\tfrac{1-\lambda}{2\lambda}\big(\partial_1 f_\eps^\perp (\double{x}_\eps) +\partial_1 f_\eps^\perp (\double{x}_\eps -\eps e_n) \big)\otimes e_n. 
 \end{align*}

In view of~\eqref{nablaeps_rig} and the observation that $\partial_1f_\eps(\double{x}_\eps)\otimes e_1 + \sum_{i=2}^{n-1} e_i\otimes e_i + \partial_1 f_\eps^\perp (\double{x}_\eps) \otimes e_n \in SO(n)$ for all $x\in \Omega$ due to $\abs{\partial_1 f_\eps} = \abs{\partial_1 f_\eps^\perp}=1$ and $\partial_1 f_\eps\in {\rm span}\{e_1, e_n\}$, the elastic energy contribution on the stiffer layers can be estimated by 
\begin{align*}
\int_{\eps\Ystiff \cap \Omega} \dist^p(\nabla u_\eps,SO(n)) \dd x & \leq  \int_{\eps\Ystiff\cap \Omega}\big| (x_n-[x_n]_\eps) \partial_{11}^2 f_\eps^\perp(\double{x}_\eps) \otimes e_1\big|^p \dd x \\
&\leq \|\partial_{11}^2 f_\eps\|^p_{L^\infty(\overline{Q};\R^n)}  \int_{\Omega}\big| x_n-[x_n]_\eps\big|^p\dd x. 
\end{align*}
This implies~\eqref{example_energyestimate}, since 
\begin{align}\label{eq55}
|x_n-[x_n]_\eps| = \eps \big| \tfrac{x_n}{\eps}  - \ceilb{\tfrac{x_n}{\eps}} +1 -\tfrac{1+\lambda}{2}\big|\leq 2\eps \quad \text{for all $x\in \Omega$. }
\end{align}

For the proof of~\eqref{nablaueps}, consider the auxiliary fields $V_\eps\in L^\infty(\Omega;\R^{n\times n})$ given by
\begin{align}\label{Veps}
\begin{split}
V_\eps& = \partial_1 f_\eps \otimes e_1 + \textstyle \sum_{i=2}^{n-1} e_i\otimes e_i+ \big(\partial_1 f_\eps^\perp\otimes e_n\big) \mathbbm{1}_{\eps\Ystiff\cap \Omega}  \\ &\hspace{3cm}+ \big( (\tfrac{1}{\lambda}\partial_n f_\eps - \tfrac{1-\lambda}{\lambda} \partial_1 f_\eps^\perp) \otimes e_n \big)\mathbbm{1}_{\eps\Ysoft\cap \Omega}.
\end{split}
\end{align}
Recall that the indicator function associated with a set $E\subset \R^n$ is denoted by $\mathbbm{1}_E$.   
We will show that
\begin{align}\label{convergenceVeps}
V_\eps-\nabla u_\eps \to 0 \qquad \text{in $L^\infty(\Omega;\R^{n\times n})$.} 
\end{align}
Indeed, along with the mean value theorem and~\eqref{eq55}, one obtains for $x$ in the interior of $\eps\Ystiff \cap \Omega$ that
\begin{align*}
\abs{\nabla u_\eps(x) -V_\eps(x)} & \leq \abs{\big(\partial_1f_\eps(\double{x}_\eps) -\partial_1 f_\eps(x)\big)\otimes e_1} + \abs{(x_n-[x_n]_\eps)\partial_{11}^2 f_\eps^\perp(\double{x}_\eps) \otimes e_n}  \\ & \qquad\qquad + \abs{\big(\partial_1 f_\eps^\perp(\double{x}_\eps) - \partial_1f_\eps^\perp (x)\big) \otimes e_n}\\
 & \leq |x_n-[x_n]_\eps| \big(\norm{\partial_{1n}^2 f_\eps}_{L^\infty(\overline{Q};\R^n)} + \norm{\partial_{11}^2 f_\eps}_{L^\infty(\overline{Q};\R^n)} + \norm{\partial_{1n}^2 f_\eps^\perp}_{L^\infty(\overline{Q};\R^n)}\big)\\ & \leq 6\eps \norm{\nabla^2 f_\eps}_{L^\infty(\overline{Q};\R^{n\times n\times n})},
\end{align*}
and similarly for $x\in \eps\Ysoft \cap \Omega$, 
\begin{align*}
\abs{\nabla u_\eps(x)-V_\eps(x)} & \leq \abs{\nabla u_\eps e_1(x)-V_\eps e_1(x)} + \abs{\nabla u_\eps e_n(x)-V_\eps e_n(x)} \\ & \leq  3\eps \bigl(\tfrac{1+\lambda}{\lambda}\bigr)\norm{\partial_{1n}^2 f_\eps}_{L^\infty(\overline{Q};\R^n)}  + \tfrac{2}{\lambda} \eps \norm{\partial_{11}^2 f_\eps^\perp}_{L^\infty(\overline{Q};\R^n)} \\
&\qquad\qquad + \tfrac{3}{\lambda} \eps \norm{\partial_{1n}^2 f_\eps}_{L^\infty(\overline{Q};\R^n)}  +  \tfrac{1}{\lambda} \big| \partial_n f_\eps(x_1, \xi) - \partial_n f_\eps(x)\big|\\
& \leq 6\eps \bigl(\tfrac{1+\lambda}{\lambda}\bigr) \bigl( \norm{\partial_{1n}^2 f_\eps}_{L^\infty(\overline{Q};\R^{n})} + \norm{\partial_{11}^2 f_\eps}_{L^\infty(\overline{Q};\R^n)} +  \norm{\partial_{nn}^2 f_\eps}_{L^\infty(\overline{Q};\R^{n})} \bigr)
\\ & \leq 18\eps \big(\tfrac{1+\lambda}{\lambda}\big) \norm{\nabla^2 f_\eps}_{L^\infty(\overline{Q};\R^{n\times n\times n})}
\end{align*}
with some $\xi\in ([x_n]_\eps-\eps, [x_n]_\eps)$. Accounting for $\lim_{\eps\to0}\eps \norm{\nabla^2 f_\eps}_{L^\infty(\overline{Q};\R^{n\times n\times n})}=0$ leads to~\eqref{convergenceVeps}.

In case $i)$, it follows from~\eqref{Veps} along with a weak-strong convergence argument that
\begin{align}\label{weakconvergence_Veps}
\begin{split}
V_\eps \weakly Fe_1\otimes e_1 &+ \textstyle \sum_{i=2}^{n-1} e_i\otimes e_i + (1-\lambda)(Fe_1)^\perp \otimes e_n \\ &+ Fe_n\otimes e_n - (1-\lambda)(Fe_1)^\perp \otimes e_n = F \quad \text{in $L^1(\Omega;\R^{n\times n})$,} 
\end{split}
\end{align} 
where we have used in particular that $\mathbbm{1}_{\eps \Ystiff \cap\Omega} \weaklystar (1-\lambda)$ and $\mathbbm{1}_{\eps \Ysoft \cap \Omega} \weaklystar \lambda$ in $L^\infty(\Omega)$. 

Combining ~\eqref{weakconvergence_Veps} and~\eqref{convergenceVeps} shows that $\nabla u_\eps\weakly F$ in $L^1(\Omega;\R^{n\times n})$. Since $(\nabla u_\eps)_\eps$ is uniformly bounded in $L^p(\Omega;\R^{n\times n})$ by~\eqref{example_energyestimate} and the requirement that $\eps\norm{\partial_{11}^2 f_\eps}_{L^\infty(\overline{Q}; \R^{n})} \to 0$, we finally infer~\eqref{nablaueps}, which finishes the proof under the assumption of $i)$. 

If assumption $ii)$ is satisfied, then $\partial_1 f_\eps^\perp$ depends only on $x_1$. Since $\mathbbm{1}_{\eps\Ystiff}$ on the other hand is constant in the $x_1$-variable, we observe a separation of variables in the product $(\partial_1 f_\eps^\perp) \mathbbm{1}_{\eps\Ystiff}$. In light of this observation, consider test functions $\varphi\in C^0(\overline{\Omega};\R^n)$ of the form $\varphi(x)=(\phi\otimes \psi)(x):=\phi(x_1)\psi(x_2, \ldots, x_n)$ for $x\in \overline{\Omega}$ with $\phi\in C^0([0,1];\R^n)$ and $\psi\in C^0([0,1]^{n-1})$. 
Then, due to Fubini's theorem and the lemma on weak convergence of rapidly oscillating periodic functions (see e.g.~\cite[Section~2.3]{CiD99}) it follows that 
\begin{align}\label{con98}
\int_{\Omega}(\partial_1 f_\eps^\perp \cdot \varphi)\mathbbm{1}_{\eps\Ystiff\cap \Omega}\dd{x} & = \Big(\int_{[0,1]}    \partial_1 f_\eps^\perp \cdot \phi \dd{x_1}\Big) \Bigl(\int_{[0,1]^{n-1}} \mathbbm{1}_{\eps\Ystiff\cap \Omega}\, \psi\dd{x_2}\ldots\dd{x_n}\Bigr) \nonumber \\ & \rightarrow \Big(\int_{[0,1]}(Fe_1)^\perp \phi(x_1)\dd{x_1} \Big)\Big(\int_{[0,1]^{n-1}} (1-\lambda) \psi\dd{x_2}\ldots \dd{x_n}\Big) \\ & =(1-\lambda) \int_{\Omega} (Fe_1)^\perp\cdot \varphi \dd{x}\qquad \text{as $\eps\to 0$.} \nonumber
\end{align}
We recall that as a corollary of the Stone-Weierstrass theorem (see e.g.~\cite[Theorem~7.32]{Rud76}) and the density of $C^0(\overline{\Omega};\R^n)$ in $L^q(\Omega;\R^n)$ with $1\leq q<\infty$, 
the span of functions $\phi\otimes \psi$ is dense in $L^q(\Omega;\R^{n})$.
Consequently, we infer from~\eqref{con98} that
 \begin{align*}
 (\partial_{1} f_\eps^\perp)\mathbbm{1}_{\eps\Ystiff\cap \Omega} \weakly (1-\lambda)(Fe_1)^\perp \quad \text{in $L^p(\Omega;\R^n)$.}
 \end{align*} 
Then the third term in~\eqref{Veps} converges weakly to $(1-\lambda) (Fe_1)^\perp\otimes e_n$ in $L^p(\Omega;\R^{n\times n})$. Arguing similarly for the other product terms in~\eqref{Veps} eventually yields $V_\eps\weakly F$ in $L^p(\Omega;\R^{n\times n})$. 
In conjunction with~\eqref{convergenceVeps} this proves~\eqref{nablaueps}, and thus the statement in case $ii)$.
\end{proof}

As announced at the beginning of the section, we will next discuss four specializations of Lemma~\ref{lem: general form of gradient}, using the same notations. 
These examples illustrate the optimality of the scaling regimes in Theorem~\ref{prop:rigidity} and Corollary~\ref{cor:asymptotic_rigidity}.

\begin{example}[Uniform bending of the individual stiffer layers]\label{ex1}
Let $g:[0,1]\to {\rm span}\{e_1, e_n\}\subset \R^n$ be a $C^2$-curve parametrized by arc length, i.e.,~$\abs{g'(t)}=1$ for all $t\in [0,1]$. 
We follow Lemma \ref{lem: general form of gradient} to define deformations $u_\eps$ by choosing for all $\eps\in (0,1)$,
\begin{align}\label{f1}
f_\eps(x) = f(x) := g(x_1) + \sum_{i=2}^{n} x_ie_i, \quad x\in \overline{Q}. 
\end{align}
This choice of $f$ is motivated by uniform bending of the individual stiffer layers in the two-dimensional setting, where the curve $g$ describes the bending of the mid-fibers, see Figure~\ref{Figure: Uniform bending}.

Then, Lemma~\ref{lem: general form of gradient} implies that for any constant $C> 2^p\norm{g''}_{L^\infty(0,1;\R^n)}^p$, 
\begin{align*}
\int_{\eps\Ystiff\cap\Omega} \dist^p(\nabla u_\eps, SO(n)) \dd{x} \leq  C\eps^{p},
\end{align*}
which shows that the sequence $(u_\eps)_\eps$ has finite elastic energy on the stiffer component for $\alpha=p$.
As for the gradient of the limit deformation $u$, we infer from version i) of Lemma~\ref{lem: general form of gradient} that $\nabla u_\eps\weakly \nabla u=\nabla f$ in $L^p(\Omega;\R^{n\times n})$.  In view of~\eqref{f1}, 
\begin{align*}
\nabla u(x)= g'(x_1) \otimes e_1 + \sum_{i=2}^{n} e_i \otimes e_i =
R(x) + \tilde a(x) \otimes e_n, \quad x\in \Omega,
\end{align*}
with $R(x)=g'(x_1)\otimes e_1 + \sum_{i=1}^{n-1} e_i\otimes e_i  + g'(x_1)^\perp\otimes e_n$ and $\tilde a(x)=e_n-g'(x_1)^\perp$. 
 Clearly, for general $g$, $\partial_1 R\neq 0$, so that the limit deformation $u$ does not have the form~\eqref{representation_u} obtained in Theorem~\ref{prop:rigidity} for the regime $\alpha>p$. 

We remark that the limit deformation $u$ is not locally volume preserving for $g$ with non-trivial curvature, since $\det \nabla u= g'\cdot e_1 \not\equiv 1$. 

\begin{figure}[ht] \label{Figure: Uniform bending}
\begin{center}
\begin{tikzpicture}
\begin{scope}[scale = 0.5]
    \clip (45:5)  -- (45:6) -- ($ (0,2) + (45:5cm) $) -- ($ (0,2) + (45:6cm) $)  -- ($ (0,4) + (45:5cm) $) -- ($ (0,4) + (45:6.2cm) $)  -- (0,11) -- ($ (0,4) + (135:6.2cm) $)  -- ($ (0,4) + (135:5cm) $) -- ($ (0,2) + (135:6cm) $)  -- ($ (0,2) + (135:5cm) $) -- (135:6) -- (135:5) -- cycle;

\draw[fill = gray!30] (6,4) arc (0:180 : 6cm);
\draw[thick] (0,4) -- ++(45:6);
\draw[thick] (0,4) -- ++(135:6);
\draw[fill = white ] (0,4) circle (5cm);
\draw[very thick] (5.5,4) arc (0:180 : 5.5cm);

\draw[fill = gray!30] (6,2) arc (0:180 : 6cm);
\draw[thick] (0,2) -- ++(45:6);
\draw[thick] (0,2) -- ++(135:6);
\draw[fill = white ] (0,2) circle (5cm);
\draw[thick] (5.5,2) arc (0:180 : 5.5cm);

\draw[fill = gray!30] (6,0) arc (0:180 : 6cm);
\draw[thick] (0,0) -- ++(45:6);
\draw[thick] (0,0) -- ++(135:6);
\draw[fill = white ] (0,0) circle (5cm);
\draw[thick] (5.5,0) arc (0:180 : 5.5cm);

\foreach \i in {0,3,...,90}
	\draw  ($ (0,2) + (45+\i:5cm) $) -- ($ (0,0) + (45+\i:6cm) $);
	
\foreach \i in {0,3,...,90}
	\draw  ($ (0,4) + (45+\i:5cm) $) -- ($ (0,2) + (45+\i:6cm) $);

\coordinate (a) at (45:5cm);
\coordinate (atwo) at (45:6cm);
\coordinate (alabel) at (42:5cm);
\coordinate (b) at (135:5cm);
\coordinate (c) at (0,7.5);

\end{scope}

\draw[white]  (a) -- (b);
\draw[<->] (alabel) -- ( $(alabel) + (atwo) - (a)$);
\draw ($(alabel)+(0.3,0.08)$) node{$\eps$};

\draw[->, thick] ($ (c) +(3,-0.3) $) -- +(2,0);
\draw ($ (c) +(4,0) $) node{$\eps \rightarrow 0$};
\draw(-2,5.0) node{$u_\eps$};
\draw(6.4,4.6) node{$g$};
\draw(7.9,1.5) node{$u(x) = g(x_1) + x_2e_2 +c$, $c\in \R^2$};

\begin{scope}[scale = 0.5, xshift = 16cm]
    \clip (45:5.5)-- ($ (0,4.2) + (45:5.5cm) $)  -- (0,11) -- ($ (0,4.2) + (135:5.5cm) $)  -- (135:5.5) -- cycle;
    
\draw[very thick, fill = gray!10] (5.5,4) arc (0:180 : 5.5cm);
\draw[thick, fill = white] (5.5,0) arc (0:180 : 5.5cm);
\draw[very thick] (45:5.5)-- ($ (0,4) + (45:5.5cm) $)  ;
\draw[very thick] ($ (0,4) + (135:5.5cm) $)  -- (135:5.5);
\end{scope}
\end{tikzpicture}
\end{center}
\begin{caption} {Illustration of the deformations of Example~\ref{ex1} for $n=2$, with uniform bending of the stiffer layers described by $g(t) = \sin (t - \tfrac{1}{2}) e_1+ \cos (t- \tfrac{1}{2}) e_2$ for $t\in [0,1]$.}
\end{caption}
\end{figure}
\end{example}

To recover limit deformations that satisfy the local volume constraint, a slightly more involved bending construction as in the next example is needed. 

\begin{example}[Macroscopically volume-preserving bending deformations]\label{ex2}
In the context of Lemma~\ref{lem: general form of gradient}, we consider for $\eps\in (0,1)$ the functions 
\begin{align*}
f_\eps(x)= f(x) := (x_n+ 1) g\Big(\frac{x_1}{x_n + 1}\Big) + \sum_{i=2}^{n-1} x_ie_i, \quad x\in \overline{Q},
\end{align*}
with $g: [0,1] \rightarrow {\rm span}\{e_1, e_n\} \subset \R^n$ a $C^2$-curve parametrized by arclength.

Then the sequence $(u_\eps)_\eps$ defined by~\eqref{defueps} in the stiffer component and by linear interpolation in the softer one satisfies
\begin{align*}
\int_{\eps\Ystiff\cap \Omega} \dist^p(\nabla u_\eps, SO(2))\dd{x} \leq  2^p\norm{g''}_{L^\infty(0,1;\R^n)}^p \eps^p,
\end{align*}
and we obtain that $\nabla u_\eps\weakly \nabla u=\nabla f$ in $L^p(\Omega;\R^{n\times n})$. 
 Due to
\begin{align*}
\nabla f(x) = g'\Big(\frac{x_1}{x_n+1}\Big)\otimes e_1 + \sum_{i=2}^{n-1} e_i\otimes e_i -\frac{x_1}{x_n+1}g'\Big(\frac{x_1}{x_n+1}\Big)\otimes e_n  + g\Big( \frac{x_1}{x_n+1}\Big)\otimes e_n, 
\end{align*}
one can rewrite the gradient of the limit deformation $u$ with the help of a map of rotations $R\in L^\infty(\Omega;SO(n))$ defined for $x\in \Omega$ by $R(x)e_1=g'(\frac{x_1}{x_n+1})$ and $R(x)e_i=e_i$ for $i=2, \ldots, n-1$. Precisely, 
\begin{align*}
\nabla u= R + \tilde a\otimes e_n
\end{align*} 
with $\tilde a(x) =  -\frac{x_1}{x_n+1}g'\big(\frac{x_1}{x_n+1}\big)  + g\big( \frac{x_1}{x_n+1}\big) -g' \big( \frac{x_1}{x_n+1}\big)^\perp$ for $x\in \Omega$. The rotations $R$ depend non-trivially on $x_1$, hence, the limit map $u$ is not in compliance with Theorem~\ref{prop:rigidity}.
Since $\det \nabla u=\det \nabla f = -g'\cdot g^\perp$, the deformation $u$ is locally volume preserving if we chose $g$ such that $g'\cdot g^\perp \equiv 1$.

An simple deformation of this type, which is intuitively inspired by the bending of a stack of paper, is depicted in Figure~\ref{Figure: volume-preserving bending}. 
 \begin{figure}[ht]
\begin{center}
\begin{tikzpicture}
\begin{scope}[scale = 0.5]
   \clip (40:4)  -- (40:5) -- ($ (0,-5.75) + (71.5:11.75) $) -- ($ (0,-5.75) +  (71.5:12.75) $)  -- ($ (0,-11.5) + (78.75:19.5)$) -- ($ (0,-11.5) + (78.75:20.7)$)  -- (0,12) -- ($ (0,-11.5) + (101.25:20.7) $)  -- ($ (0,-11.5) +  (101.25:19.5) $) -- ($ (0,-5.75) + (108.5:12.75) $)  --  ($ (0,-5.75) + (108.5:11.75) $)  -- (140:5) -- (140:4)-- cycle;

\draw[fill = gray!30] (20.5,-11.5) arc (0:180 : 20.5cm);
\draw[thick] (0,-11.5) -- ++(78.75:20.5);
\draw[thick] (0,-11.5) -- ++(101.25:20.5);
\draw[fill = white ] (0,-11.5) circle (19.5cm);
\draw[thick] (20,-11.5) arc (0:180 : 20cm);

\draw[fill = gray!30] (12.75,-5.75) arc (0:180 : 12.75cm);
\draw[thick] (0,-5.75) -- ++(71.5:12.75);
\draw[thick] (0,-5.75) -- ++(108.5:12.75);
\draw[fill = white ] (0,-5.75) circle (11.75cm);
\draw[thick] (12.25,-5.75) arc (0:180 : 12.25cm);

\draw[fill = gray!30] (5,0) arc (0:180 : 5cm);
\draw[thick] (0,0) -- ++(40:5);
\draw[thick] (0,0) -- ++(140:5);
\draw[fill = white ] (0,0) circle (4cm);
\draw[thick] (4.5,0) arc (0:180 : 4.5cm);

\foreach \i in {0,0.75,...,22.5}
	\draw  ($ (0,-5.75) + (71.5+1.644444*\i:12.75cm) $) -- ($ (0,-11.5) + (78.75+\i:19.5cm) $);
	
\foreach \i in {0,0.75,...,22.5}
	\draw  ($ (0,0) + (40+2.7027*1.64444*\i:5cm) $) -- ($ (0,-5.75) + (71.5+1.64444*\i:11.75cm) $);

\coordinate (a) at (40:4cm);
\coordinate (b) at (140:4cm);
\coordinate (atwo) at (40:5cm);
\coordinate (alabel) at (37.5:4cm);
\coordinate (c) at (0,6.5);
\end{scope}

\draw[white]  ($(a)+(0,-0.00)$) -- ($(b)+(0,-0.00)$);
\draw[<->] (alabel) -- ( $(alabel) + (atwo) - (a)$);
\draw ($(alabel)+(0.3,0.03)$) node{$\eps$};
\draw[->, thick] ($ (c) +(3,-0.3) $) -- +(2,0);
\draw ($ (c) +(4,0) $) node{$\eps \rightarrow 0$};

\draw(-2,4.7) node{$u_\eps$};
\draw(6.4,4.5) node{$u$};

\begin{scope}[scale = 0.5, xshift = 16cm]
\draw[thick, fill = gray!10] (-3.4472,2.892544) arc (140:40: 4.5cm) --  plot [domain=4.5:20,smooth,variable=\r] ({\r*sin(180*1.25/\r)},{\r*cos(180*1.25/\r)-(11.5/15.5)*(\r-4.5)}) -- (3.9018,8.1157) arc (78.75:101.25 : 20cm) -- plot [domain=20:4.5,smooth,variable=\r] (-{\r*sin(180*1.25/\r)},{\r*cos(180*1.25/\r)-(11.5/15.5)*(\r-4.5)});
\end{scope}
\end{tikzpicture}
\end{center}
\begin{caption}{Illustration of the deformations of Example~\ref{ex2} for $n=2$, with $g(t) = \sin (t-\frac{1}{2}) e_1 + \cos(t -\frac{1}{2}) e_2$ for $t\in [0,1]$.  Notice also that the limit deformation $u$ satisfies the local volume constraint, whereas the bending deformations of the individual layers in the left picture do not.}\label{Figure: volume-preserving bending} \end{caption}
\end{figure}
\end{example}

Next, we discuss an example in the regime $\alpha<p$, where macroscopic shortening in $e_1$-direction occurs due to wrinkling of the stiffer layers. A similar effect occurs in the context of plate theory, cf.~\cite[Section~5]{FJM02}.

\begin{example}[Wrinkling of stiffer layers]\label{ex3}
Let $\beta\in \R$, $\gamma\in (0,1)$, and $g:[0,1] \to {\rm span}\{e_1, e_n\} \subset \R^n$ be a $1$-periodic $C^2$-function with $|g'(t)|=1$ for all $t\in \R$. We define $g_\eps:[0,1]\to \R^n$ by $g_\eps(t) = \eps^\gamma g(\eps^{-\gamma}t)$ for $t\in [0,1]$ and $\eps\in (0,1)$, and observe that by the weak convergence of periodically oscillating sequences, 
\begin{align*}
g_\eps'\weakly \bar{g}':=\int_0^1 g'(t)\dd{t} = g(1)-g(0)\qquad \text{in $L^1(0,1;\R^n)$.} 
\end{align*}
Unless $g'$ is constant, $|\bar{g}'| <1$. Under these assumptions, the functions
\begin{align*}
f_\eps(x)  = g_\eps(x_1) + \beta x_n e_n + \sum_{i=2}^{n-1} x_i e_i, \quad x\in \overline{Q},
\end{align*}
meet the requirements of Lemma~\ref{lem: general form of gradient} with assumption $ii)$ and $F = \bar{g}'\otimes e_1 + \sum_{i=2}^{n-1} e_i\otimes e_i+  \beta e_n\otimes e_n$. 
Thus, for $u_\eps$ as in Lemma~\ref{lem: general form of gradient},
\begin{align*}
\int_{\eps\Ystiff \cap \Omega} \dist^p(\nabla u_\eps,SO(n)) \dd{x} \leq 2^p\eps^p \norm{g''_\eps}^p_{L^\infty(0,1;\R^n)} \leq 2^p\eps^{p(1-\gamma)} \norm{g''}^p_{L^\infty(0,1;\R^n)} \leq C\eps^{p(1-\gamma)},
\end{align*}
and $\nabla u_\eps\weakly \nabla u=F$ in $L^1(\Omega;\R^{n\times n})$. In particular, $|(\nabla u)e_1| = |Fe_1|= |\bar{g}'| <1$. 

Since $\det \nabla u = \det F = \beta (\bar{g}'\cdot e_1)$, (local) volume preservation of the limit deformation $u$ can be achieved by a suitable choice of $\beta$ and $g$. Graphically speaking, $\beta$ can be viewed as a stretching factor in $e_n$-direction that compensates the loss of length in $e_1$-direction due to the asymptotic shortening of the mid-fibers in the stiffer layers, so that overall volume is preserved. A specific case of this wrinkling construction is depicted in Figure~\ref{figure:ex3}.

 \begin{figure}[ht]
\begin{center}
\begin{tikzpicture}
\begin{scope}[scale = 0.5]

\foreach \j in {0,2.5,5} 
{

\foreach \i in {0}
{
\draw[fill = gray!30,draw =  gray!30] ($(0,\j)+(\i,0)+(3.853553,0.353553)$) -- ($(0,\j)+(\i,0)+(3.146447,-0.353553)$)  arc (225:315 : 1.56066) --  ($(0,\j)+(\i,0)+(5,0)+(-0.353553,0.353553)$) -- cycle;
\draw[thick, draw = white]  ($(0,\j)+(\i,0)+(3.853553,0.353553)$) --  ($(0,\j)+(\i,0)+(7,0)+(-0.353553,0.353553)$);
}

\foreach \i in {0,5}
{
\draw[fill = gray!30, draw =gray!30]  ($(0,\j)+(\i,0)+(0.353553,-0.353553)$) -- ($(0,\j)+(\i,0)+(1.353553,0.646447)$) -- ($(0,\j)+(\i,0)+(0.646447,1.353553)$) -- ($(0,\j)+(\i,0)+(-0.353553,0.353553)$) -- cycle;
\draw[fill = gray!30,draw =  gray!30] ($(0,\j)+(\i,0)+(2.146447,0.646447)$) -- ($(0,\j)+(\i,0)+(2.853553,1.353553)$) arc (45:135 : 1.56066) --  ($(0,\j)+(\i,0)+(1.353553,0.646447)$) -- cycle;
\draw[thick, draw = white] ($(0,\j)+(\i,0)+(1.353553,0.646447)$)--($(0,\j)+(\i,0)+(2.146447,0.646447)$);
\draw[thick] ($(0,\j)+(\i,0)+(0,0)$) -- ($(0,\j)+(\i,0)+(1,1)$); 
\draw ($(0,\j)+(\i,0)+(0.353553,-0.353553)$) -- ($(0,\j)+(\i,0)+(1.353553,0.646447)$);
\draw ($(0,\j)+(\i,0)+(-0.353553,0.353553)$) -- ($(0,\j)+(\i,0)+(0.646447,1.353553)$);

\draw[fill = gray!30, draw =  gray!30] ($(0,\j)+(\i,0)+(2.146447,0.646447)$) -- ($(0,\j)+(\i,0)+(3.146447,-0.353553)$) -- ($(0,\j)+(\i,0)+(3.853553,0.353553)$) -- ($(0,\j)+(\i,0)+(2.853553,1.353553)$) -- cycle;
\draw[thick] ($(0,\j)+(\i,0)+(2.5,1)$) -- ($(0,\j)+(\i,0)+(3.5,0)$);
\draw ($(0,\j)+(\i,0)+(2.146447,0.646447)$) -- ($(0,\j)+(\i,0)+(3.146447,-0.353553)$);
\draw ($(0,\j)+(\i,0)+(2.853553,1.353553)$) -- ($(0,\j)+(\i,0)+(3.853553,0.353553)$);

\draw ($(0,\j)+(\i,0)+(2.853553,1.353553)$) arc (45:135 : 1.56066);
\draw[thick] ($(0,\j)+(\i,0)+(2.5,1)$) arc (45:135 : 1.06066);
\draw[fill = white] ($(0,\j)+(\i,0)+(2.146447,0.646447)$)  arc (45:135 : 0.56066);}

\foreach \i in {0}
{
\draw[thick] ($(0,\j)+(\i,0)+(3.5,0)$) arc (225:315 : 1.06066);
\draw ($(0,\j)+(\i,0)+(3.146447,-0.353553)$)  arc (225:315 : 1.56066);
\draw[fill = white] ($(0,\j)+(\i,0)+(3.853553,0.353553)$) arc (225:315 : 0.56066);}

\draw ($(0,\j)+(0.353553,-0.353553)$) -- ($(0,\j)+(-0.353553,0.353553)$);
\draw ($(0,\j)+(8.146447,-0.353553)$) -- ($(0,\j)+(8.853553,0.353553)$);
}

\foreach \j in {0,2.5} 
{
	\foreach \i in {0,5}
 	{
		\foreach \r in {0,0.3,...,1}
		{
			   \draw ($(0,\j)+(\i,0)+(-0.353553,0.353553)+\r*(1,1)$) --($(0,\j)+(\i,2.5)+(0.353553,-0.353553)+\r*(1,1)$);
			   \draw ($(0,\j)+(\i,0)+(2.853553,1.353553)+\r*(1,-1)$) --($(0,\j)+(\i,2.5)+(2.146447,0.646447)+\r*(1,-1)$);
		}
		   \foreach \l in {55, 65,...,125}
		\draw  ($(0,\j)+(\i,0)+(1.75,0.25) + (\l : 1.56066) $) -- ($(0,\j)+(\i,0)+(1.75,2.75) + (\l : 0.56066) $);	
	}
}

   \foreach \l in {235, 245,...,305}
		\draw  ($(4.25,0.75) + (\l : 0.56066) $) -- ($(4.25,3.25) + (\l : 1.56066) $);	
   \foreach \l in {235, 245,...,305}
		\draw  ($(4.25,3.25) + (\l : 0.56066) $) -- ($(4.25,5.75) + (\l : 1.56066) $);

\draw[<->] ($(0,0)+(5,0)+(3.853553,0.353553)+(0.2,-0.2)$) -- ($(0,0)+(5,0)+(3.146447,-0.353553)+(0.2,-0.2)$);
\draw ($(0,0)+(5,0)+(3.5,0)+(0.45,-0.45)$) node{$\eps$};

\coordinate (c) at (4.25,4);
\end{scope}

\draw[->, thick] ($ (c) +(3,-0.3) $) -- +(2,0);
\draw ($ (c) +(4,0) $) node{$\eps \rightarrow 0$};

\begin{scope}[scale = 0.5, xshift = 16cm]
\draw[thick,fill = gray!10]  (0,0.5) rectangle (8.5,6);
\draw[dashed]  (-1,1) rectangle (9.5,5.5);

\draw(-15.3,7.3) node{$u_\eps$};
\draw(0,7) node{$u$};
\end{scope}
\end{tikzpicture}
\end{center}
\begin{caption} {Illustration of the deformation in Example~\ref{ex3} for $n=2$. \label{figure:ex3} 
}\end{caption}
\end{figure}
\end{example}

Our last example highlights the role of the local volume constraint of the limit deformation in the regime $\alpha\geq 0$. In particular, it shows that for $\alpha>p$ local volume preservation of the limit deformation is necessary to obtain asymptotic rigidity in the sense of Corollary~\ref{cor:asymptotic_rigidity}.

\begin{example}[Rotation of stiffer layers]\label{ex4}
Let $R \in C^1([0,2];SO(n))$ with $Re_i=e_i$ for $i=2, \ldots, n-1$. For each $\eps\in (0,1)$, we set
\begin{align*}
f_\eps(x)=f(x):= (x_1-\tfrac{1}{2}) R(x_n)e_1 + \tfrac{1}{2}e_1 +\sum_{i=2}^n x_ie_i, \quad x\in \overline{Q},
\end{align*}
and take $u_\eps$ as defined in Lemma~\ref{lem: general form of gradient}. 
Since $\partial_{11} f=0$, it follows from~\eqref{example_energyestimate} that 
\begin{align*}
\int_{\eps\Ystiff\cap \Omega} \dist^p(\nabla u_\eps,SO(n))\dd{x}=0
\end{align*} 
for any $\eps\in (0,1)$. 
Moreover, $\nabla u_\eps\weakly \nabla u= \nabla f$ in $L^p(\Omega;\R^{n\times n})$, so that
\begin{align*}
\nabla u(x) & = R(x_n) e_1 \otimes e_1 + \sum_{i=2}^n e_i\otimes e_i + \big(x_1 -\tfrac{1}{2}\big)R'(x_n)e_1\otimes e_n \\ 
& = R(x_n) + R'(x_n)x\otimes e_n   + d(x_n)\otimes e_n, 
\end{align*}
where $d(t) = - \frac{1}{2} R'(t) e_1 - R'(t) t e_n$ for $t\in (0,1)$. Hence, we obtain 
\begin{align*}
u(x) = R(x_n) x + b(x_n) 
\end{align*}
with $b(t) = -\frac{1}{2} Re_1 - \int_0^{t} sR'(s)e_n \dd{s} + c$ for $t\in (0,1)$ and $c\in \R^n$. 
 It is now immediate to see that $u$ has the form stated in Theorem~\ref{prop:rigidity}, but neither is $Re_n$ constant nor is the local volume condition satisfied in general.

In $2$d, this construction corresponds to a $x_2$-dependent rotation of the individual stiffer layers around their barycenters, see Figure~\ref{Figure: rotation}.

\begin{figure}[ht] 
\begin{center}
\begin{tikzpicture}
\begin{scope}[scale = 0.5]

\begin{scope}[yshift = 4cm, rotate = 16]
\draw[fill = gray!30] (-4,-0.5) rectangle (4,0.5);
\draw[thick] (-4,0)--(4,0);
\end{scope}

\begin{scope}[yshift = 2cm, rotate = 8]
\draw[fill = gray!30] (-4,-0.5) rectangle (4,0.5);
\draw[thick] (-4,0)--(4,0);
\end{scope}

\begin{scope}
\draw[fill = gray!30] (-4,-0.5) rectangle (4,0.5);
\draw[thick] (-4,0)--(4,0);
\end{scope}

\draw[very thick, dashed] (0,-0.5) -- (0,4.5);
\foreach \i in {-4, -3.75,...,4}
	\draw  (\i,0.5) -- ($(0,2)+ (8:\i cm)+(0.07,-0.495) $); 
	
\foreach \i in {-4, -3.75,...,4}
	\draw ($(0,2)+ (8:\i cm)-(0.07,-0.495) $)  -- ($(0,4)+ (16:\i cm)+(0.1378,-0.4806) $);

\coordinate (atwo) at (4.25,0);
\coordinate (alabel) at (4.3,0);
\coordinate (c) at (0,3);
\end{scope}

\draw[<->] ($(atwo) + (0,-0.25)$) -- ($(atwo) + (0,0.25)$) ;
\draw ($(atwo) +(0.2,0)$) node{$\eps$};
\draw[->, thick] ($ (c) +(3,-0.3) $) -- +(2,0);
\draw ($ (c) +(4,0) $) node{$\eps \rightarrow 0$};

\begin{scope}[scale = 0.5, xshift = 16cm]

\draw[fill = gray!10,thick]  plot [scale=1,domain=0:16,smooth,variable=\r]  ({-4*cos(\r)},{-4*sin(\r)+0.25*\r}) -- plot [scale=1,domain=16:0,smooth,variable=\r]  ({4*cos(\r)},{4*sin(\r)+0.25*\r}) -- cycle;

\draw[thick, dashed] (0,0) -- (0,4);  

\draw(-19.3,4.5) node{$u_\eps$};
\draw(-3.5,4) node{$u$};
\end{scope}
\end{tikzpicture}
\end{center}
\begin{caption}
{Illustration of the deformation of Example~\ref{ex4} for $n=2$, where the increasing rotation of the stiffer layers is described by $R\in C^2([0,2];SO(2))$ with $R(t)e_1=\cos(t) e_1 + \sin(t)e_2$ for $t\in [0,2]$.}\label{Figure: rotation}\end{caption}
\end{figure}
\end{example}

We conclude this section with Table~\ref{Fig1}, which illustrates at one glance our findings in different scaling regimes for two space dimensions. Notice that any $(2\times 2)$-matrix can be expressed as $R(\beta\Ibb+a\otimes e_2)$ with $R\in SO(2)$, $\beta\in \R$ and $a\in \R^2$. 
\begin{figure}[ht]
\begin{tabular}{|c||l|l|}\hline
   & (a) $\nabla u$ & (b) $\nabla u$ with $\det \nabla u=1$ \\ \hline\hline
$\alpha=0$ & 
Characterization: & Characterization: \\
& no further restriction on $R, \beta, a$& $\beta^2+ \beta a_2=1$\\ \hline
$\alpha\in (0,p)$ & Explicit construction:  &\\
& $\beta\neq 1$  & \\ 
& see Example~\ref{ex3} &  \\ \hline
$\alpha=p$ & Explicit construction: & Explicit construction: \\
& $\nabla' R=\partial_1 R\neq 0$ & $R\neq {\rm const.}$ \\ 
& see Example~\ref{ex1} & see Example~\ref{ex2}  \\ \hline
$\alpha \in (p, \infty)$ & Characterization: & Characterization: \\ 
& $R\in W^{1,p}(\Omega;SO(2))$ with $\partial_1 R=0$, & $R={\rm const.}$, $\beta=1$   \\
& $\beta=1$, $\partial_1 a = R^T\curl R$ 

& $ a\,||\,e_1, \partial_1 a= 0$\\ 
& see Theorem~\ref{prop:rigidity} & see Corollary~\ref{cor:asymptotic_rigidity}, Remark~\ref{rem:2d}  \\ \cline{2-3}
& Explicit construction: & \\  
& $R \neq {\rm const.}$ & \\ 
& see Example~\ref{ex4}& \\ \hline
\end{tabular}
\caption{Overview of the results on the asymptotic behavior of weakly converging sequences $(u_\eps)_\eps\subset W^{1,p}(\Omega;\R^2)$ satisfying~\eqref{approx_diffinclusion_intro} in the different scaling regimes for $n=2$, (a) without and (b) with local volume constraint on the limit map $u$. It is used here that for any $u\in W^{1,p}(\Omega;\R^2)$ there are $R\in L^\infty(\Omega;SO(2))$, $\beta\in L^p(\Omega)$ and $a\in L^p(\Omega;\R^2)$ with $\nabla u= R(\beta\Ibb+a\otimes e_2)$.
}\label{Fig1} 
\end{figure}


\section{Proof of necessity in Theorem~\ref{prop:rigidity}}\label{sec:proof}

We will show in this section that weak limits of bounded energy sequences in the context of our model for layered materials with stiff and soft components have a strongly one-dimensional character. To make this more precise, we first introduce the following terminology.   
A measurable function $f:\Omega \to \R^m$, where $\Omega\subset \R^n$ is an open set, is said to be locally one-dimensional in $e_n$-direction if for every $x\in \Omega$ there is an open cuboid $Q_x\subset \Omega$ with $x\in Q_x$ such that for all $y, z\in Q_x$, 
\begin{align}\label{1d}
f(y) = f(z)\quad \text{ if $y_n=z_n$.}
\end{align}  
We call $f$ (globally) one-dimensional in $e_n$-direction if~\eqref{1d} holds for all $y, z\in \Omega$. 
For $f\in W^{1,p}_{\rm loc}(\Omega;\R^m)$ with $p\geq 1$ local one-dimensionality in $e_n$-direction of 
$f$, which means that there exists a representative of $f$ with the property, is equivalent to the condition $\nabla' f=0$, as can be seen from standard mollification argument. Hence, if $\nabla' f=0$, the function $f$ can be identified locally (i.e.~for any $x\in \Omega$ on an open cuboid $Q_x\subset \Omega$ containing $x$) with a one-dimensional $W^{1,p}$-function. Since the latter is absolutely continuous, it follows that $f$ is continuous. 

The next result and its implications discussed subsequent to its proof generalize the necessity statement ofTheorem~\ref{prop:rigidity} relaxing the assumptions on the domain.

\begin{theorem}\label{theo:rigidity_general} 
Let $\Omega\subset \R^n$ with $n\geq 2$ be a bounded open set, $1<p<\infty$  
and $\alpha>p$.
Furthermore, let $(u_\eps)_\eps \subset W^{1,p}(\Omega;\R^n)$ be such that for all $\eps>0$,
\begin{align}\label{approx_diffinclusion}
\int_{\eps\Ystiff\cap \Omega} \dist^p(\nabla u_\eps, SO(n))\dd{x} \leq C\eps^\alpha
\end{align}
with a constant $C>0$, and $u_\eps\weakly u$ in $W^{1,p}(\Omega;\R^n)$ for some $u\in W^{1,p}(\Omega;\R^n)$.

Then there exist $R\in W^{1, p}_{\rm loc}(\Omega; SO(n))$ with $\nabla' R=0$ and $b\in W^{1, p}_{\rm loc}(\Omega; \R^n)$ with $\nabla' b = 0$ such that 
\begin{align}\label{representation_u}
u(x) = R(x)x+b(x) \quad \text{for $x \in \Omega$.}
\end{align}
\end{theorem}

\begin{remark}\label{rem:necessity}
a) Notice that the functions $R$ and $b$ are both locally one-dimensional in $e_n$-direction and continuous. In particular, $u\in C^0(\Omega;\R^n)$.\\[-0.2cm] 

b) It is straightforward to generalize Theorem~\ref{theo:rigidity_general} to a $(p, q)$-version. Precisely, if $p$ in~\eqref{approx_diffinclusion} is replaced with any $1<q<\infty$, the same conclusion remains true under the assumption that $\alpha>q$, cf.~\cite[Section~3.3]{Chr18}. For a discussion of the case $p=1$, see Remark \ref{rem:rigidity} \\[-0.2cm]

c) One can show that the statement of Theorem~\ref{theo:rigidity_general} remains true if the relative thickness of the softer layers $\lambda\in (0,1)$ depends on $\epsilon$ (then denoted by $\lambda_\eps$) in such a way that $1-\lambda_\eps \gg \eps^{\frac{\alpha}{p}-1}$.  For more details, we refer to~\cite[Theorem~3.3.1]{Chr18}.
\end{remark}

Theorem~\ref{theo:rigidity_general} builds on two classical results, which we recall here for the readers' convenience. The first one is the quantified rigidity result for Sobolev functions established in~\cite[Theorem~3.1]{FJM02}, cf. also~\cite{CoS06, CDM14, Con03} for generalizations to other $W^{1, p}$-settings.

\begin{theorem}[Quantitative rigidity estimate]\label{prop:FJM}
Let $U \subset \R^n$ with $n\geq 2$ be a bounded Lipschitz domain and $1< p<\infty$. Then there exists a constant $C=C(U, p)>0$ with the property that for each $u \in W^{1,p}(U; \R^n)$ there is a rotation $R \in SO(n)$ such that 
\begin{align*}
\|\nabla u - R\|_{L^p(U;\R^{n\times n})} \leq C \|\dist(\nabla u, SO(n))\|_{L^p(U)}.
\end{align*}
\end{theorem}

A straightforward scaling argument shows that the constant $C$ remains unaffected by uniform scaling and translation of $U$. Applying the above theorem to increasingly thinner domains, however, leads to degenerating constants. If $U= P_\eps=O \times \eps I\subset \R^n$ with $\eps>0$, $O\subset \R^{n-1}$ a cube and $I\subset \R$ a bounded open interval
one obtains that 
\begin{align}\label{scalingPeps}
C(P_\eps, p)=\eps^{-1}C(P_1, p),
\end{align} 
see~\cite[Section~4]{FJM06} and~\cite[Section~3.5.1]{Chr18}. 

The second tool is the Fr\'echet-Kolmogorov theorem, a compactness result for $L^p$-functions, see e.g.~\cite[Sections~2.15, U.2]{Alt16} and \cite{HaH10}. 
Here, we will apply it only in the basic version formulated in the next lemma, that is, for families of functions of one real variable with uniformly bounded essential supremum.   

\begin{lemma}\label{theo:FrechetKolmogorov}
Let $J, J'\subset \R$ be open, bounded intervals with $J\subset\subset J'$ and $1\leq p<\infty$. If the sequence $(f_\eps)_\eps$ is uniformly bounded in $L^\infty(J';\R^m)$ satisfying 

\begin{align*}
\lim_{|\xi|\to 0} \sup_{\eps>0} \int_{J} \abs{f_\eps(t+\xi) -f_\eps(t)}^p \dd{t} = 0,
\end{align*}
then $(f_\eps)_\eps$ is relatively compact in $L^p(J;\R^m)$. 
\end{lemma}

Regarding structure, the following proof proceeds along the lines of~\cite[Proposition~2.1]{ChK17}, which, as mentioned in the introduction, constitutes a special case of Theorem~\ref{prop:rigidity}. Yet, the individual steps are more involved and require new, refined arguments to relax the assumption of the stiff layers being fully rigid and to overcome the restriction to two space dimensions.

\begin{proof}[Proof of Theorem~\ref{theo:rigidity_general}]
Let $Q=O\times J\subset \Omega$ be a cuboid with $O\subset \R^{n-1}$ an open cube of side length $l>0$ and $J\subset \R$ an open interval. 
Suppose that there exist open intervals $J', J''$ with $J\subset\subset J'\subset\subset J''$ and $Q'' := O\times J''\subset \Omega$. Moreover, let $Q':= O\times J'$.
We define horizontal strips by setting
\begin{align*}
P_\eps^i= (\R^{n-1}\times \eps[i, i+1)) \cap Q'' \qquad \text{for~$ i \in \Z$ and $\eps>0$}.
\end{align*}
The index set $I_\eps$ contains all $i\in \Z$ with $\abs{P_\eps^i}=\eps |O|$, and we assume $\eps>0$ to be small enough, so that  $Q\subset Q' \subset \bigcup_{i \in I_\eps} P^i_\eps \subset Q''$.  

For the proof, 
it suffices to show the existence of $R\in W^{1,p}(Q;SO(n))$ and $b\in W^{1,p}(Q;\R^n)$ with $\nabla' R=0$ and $\nabla'b=0$ in $Q$, respectively, such that the characterization~\eqref{representation_u} holds  
for~$x\in Q$.
Then we can approximate $\Omega$ from inside with overlapping cuboids to obtain the same statements for any compact $K\subset \Omega$. Indeed, the resulting characterizations in terms of $R$ and $b$ coincide on the overlapping parts. Finally, exhausting $\Omega$ with compact nested subsets proves Theorem~\ref{theo:rigidity_general} in the stated generality. 

In the following, the constants $C>0$ depend at most on $n, p, \lambda, \Omega$ and $c$ from~\eqref{approx_diffinclusion}, in particular, they are independent of $\eps$, $l$ and $J$. 

\textit{Step~1: Layerwise approximation by rigid body motions.} In this first step, we will construct a sequence of piecewise affine functions $(w_\eps)_\eps$ such that the restriction of each $w_\eps$ to a strip $P_\eps^i$ is a rigid body motion and 
\begin{align}\label{convergence}
\lim_{\eps\to 0}\norm{u_\eps-w_\eps}_{L^p(Q';\R^n)}=0.
\end{align}
Applying Theorem~\ref{prop:FJM} (under consideration of  the scaling behavior of the constant according to~\eqref{scalingPeps}) to the individual stiff layers yields the existence of $C>0$ and of rotations $R^i_\eps \in SO(n)$ for every $i\in I_\eps$ such that
\begin{align}\label{FJMestimate}
\|\nabla u_\eps-R_\eps^i\|_{L^p(\eps \Ystiff\cap P_\eps^i;\R^{n\times n})} \leq C\eps^{-1}\|\dist(\nabla u_\eps, SO(n))\|_{L^p(\eps \Ystiff\cap P_\eps^i)}.
\end{align}

Let $w_\eps \in L^\infty(Q';\R^n)$ be defined by $w_\eps=\sigma_\eps + b_\eps$, where 
\begin{align}\label{sigmaeps}
\sigma_\eps(x) =\sum_{i\in I_\eps}(R_\eps^ix)\mathbbm{1}_{P_\eps^i \cap Q'}(x), \quad \text{ $x\in Q'$},
\end{align}
and
\begin{align*}
b_\eps=\sum_{i\in I_\eps} b_\eps^i\mathbbm{1}_{P_\eps^i\cap Q'}\quad\text{with $b_\eps^i = 
  \dashint_{\eps \Ystiff\cap P_\eps^i} u_\eps - R_\eps^ix \dd x$.} 
\end{align*}

The specific choice of the values $b_\eps^i$ implies that $\int_{\eps \Ystiff\cap P_\eps^i} u_\eps-w_\eps\dd{x}=0$, and therefore allows us to apply Poincar\'{e}'s inequality to $u_\eps-w_\eps$ on each stiff layer. Hence, one obtains for every $i\in I_\eps$ that 
\begin{align}\label{est44}
\norm{u_\eps - w_\eps}_{L^p(\eps \Ystiff\cap P_\eps^i;\R^n)}  &\leq  C  \norm{\nabla u_\eps - R_\eps^i}_{L^p(\eps \Ystiff\cap P_\eps^i;\R^{n\times n})}, 
\end{align}
see e.g.~\cite[Section 7.8]{GiT01} for details on the domain dependence of the Poincar\'{e} constant.

Next we derive a corresponding bound on the softer layers. By a shifting argument, this problem can be reduced to estimate~\eqref{est44} for the stiff layers. The error is given in terms of difference quotients in $e_n$-direction of $u_\eps-w_\eps$, which we control uniformly.
More precisely, for fixed $i\in I_\eps$ we cover $\eps \Ysoft\cap P_\eps^i$ with finitely many shifted copies of $\eps\Ystiff\cap P_\eps^i$, that is, we consider $O_{\eps, k}^i =\eps \Ystiff\cap P_\eps^i -\delta_k e_n$ with $0<\delta_k \leq \lambda \eps$ and $k=1, \ldots, N:=\lceil\tfrac{\lambda}{1-\lambda}\rceil$ such that $\eps\Ystiff \cap P_\eps^i \subset \bigcup_{k=1}^N O_{\eps, k}^i$. Then, 
\begin{align*}
\int_{O_{\eps, k}^i} \abs{u_\eps - w_\eps}^p \dd{x} & \leq C \int_{\eps\Ystiff\cap P_\eps^i}\abs{u_\eps-w_\eps}^p \dd{x} \\ & \qquad \qquad \qquad + C \int_{\eps \Ystiff\cap P_\eps^i} |(u_\eps-w_\eps)(x) - (u_\eps-w_\eps)(x-\delta_k e_n)|^p \dd{x}\\
& \leq C \bigl( \norm{u_\eps-w_\eps}^p_{L^p(\eps\Ystiff\cap P_\eps^i;\R^n)}
+ \delta_k^p \norm{\partial_n u_\eps -  \partial_n w_\eps}_{L^p(P_\eps^i;\R^n)}^p\bigr). 
\end{align*}

Here, we have used a one-dimensional difference quotient estimate with respect to the $x_n$-variable. 
Summing over the $N$ covering cuboids then leads to
\begin{align*}
\int_{\eps \Ysoft\cap P_\eps^i}\abs{u_\eps-w_\eps}^p \dd{x} & \leq C \bigl(\| u_\eps-w_\eps \|_{L^p(\eps \Ystiff\cap P_\eps^i;\R^n)}^p +  \eps^p \norm{\nabla u_\eps}^p_{L^p(P_\eps^i;\R^{n\times n})} + \eps^p |P_\eps^i| \bigl).
\end{align*}

Finally, we take the sum over $i\in I_\eps$ to deduce from~\eqref{est44} and~\eqref{FJMestimate} that
\begin{align*}
\int_{Q'} \abs{u_\eps-w_\eps}^p\dd{x} 
\leq C \bigr( \eps^{-p}\|\dist(\nabla u_\eps, SO(n) \|^p_{L^p(\eps\Ystiff\cap \Omega)} + \eps^p \|u_\eps\|_{W^{1,p}(\Omega;\R^n)}^p  + \eps^p |\Omega|\bigl).
\end{align*}
Therefore, by~\eqref{approx_diffinclusion} and the uniform boundedness of $(u_\eps)_\eps$ in $W^{1,p}(\Omega;\R^n)$,
\begin{align}\label{est2}
\norm{u_\eps - w_\eps}_{L^p(Q';\R^n)} \leq C(\eps^{\frac{\alpha}{p} -1} + \eps).
\end{align}
Since $\alpha>0$, this implies~\eqref{convergence}. 

\textit{Step~2: Compactness of the approximating rigid body motions.} 
Consider for $\eps>0$ the piecewise constant one-dimensional auxiliary function $\Sigma_\eps: J'\to SO(n)$ defined by
\begin{align}\label{Sigmaeps}
\Sigma_\eps(t)=\sum_{i\in I_\eps} R_\eps^i \mathbbm{1}_{\eps[i, i+1)}(t), 
\qquad t\in J',
\end{align}
with $R_\eps^i$ as in Step~1. In relation to~\eqref{sigmaeps}, it holds that $\sigma_\eps(x) = \Sigma_\eps(x_n)x$ for $x\in Q'$.

\textit{Step~2a: Estimate for rotations on different strips.}
Next we will show that for every $\xi\in \R$ such that $J \cup (J+\xi)\subset J'$,
\begin{align}\label{keyestimate}
\|\Sigma_\eps(\cdot +\xi) - \Sigma_\eps\|_{L^p(J; \R^{n\times n})} \leq  C l^{-p-n+1}\big( \norm{u_\eps-w_\eps}_{L^p(Q';\R^n)} + \xi\|u_\eps\|_{W^{1,p}(\Omega;\R^n)} \big).
\end{align}
To this end, we estimate the expression $\|w_\eps (\cdot + \xi e_2) - w_\eps\|_{L^p(Q;\R^n)}$ from above and below. 

The upper bound follows from
\begin{align}\label{est46}
\|w_\eps (\cdot + \xi e_2) - w_\eps\|_{L^p(Q;\R^n)} &
\leq \|w_\eps -u_\eps\|_{L^p(Q';\R^n)}   + \|w_\eps(\cdot +\xi e_n) -u_\eps(\cdot+\xi e_n)\|_{L^p(Q;\R^n)} \nonumber \\ & \qquad\qquad\qquad  + \|u_\eps(\cdot+\xi e_n) - u_\eps\|_{L^p(Q;\R^n)} \nonumber \\
&\leq 2\|w_\eps -u_\eps\|_{L^p(Q';\R^n)} + \xi  \|\partial_n u_\eps\|_{L^p(Q';\R^n)}.
\end{align}

For the lower bound, we set $d_{\eps, \xi}^i = b_\eps^{i+ \lfloor\frac{\xi}{\eps}\rfloor} -b_\eps^i + \xi R_\eps^{i+\lfloor \frac{\xi}{\eps} \rfloor}e_n$ and use Lemma~\ref{lem:estimateR1R2} to derive that 
\begin{align}\label{est47}
&\|w_\eps (\cdot + \xi e_n) - w_\eps\|^p_{L^p(Q;\R^n)} 
=  \sum_{i\in I_\eps} \int_{P_\eps^{i}\cap Q}
 \bigl|(R^{i + \lfloor\frac{\xi}{\eps}\rfloor}_\eps - R^{i}_\eps) x +d^{i}_{\eps, \xi}\bigr|^p\dd x \nonumber \\
&\qquad\qquad  \geq  C l^p \sum_{i\in I_\eps} |R_\eps^{i+\lfloor\frac{\xi}{\eps}\rfloor}-R_\eps^{i}|^p  |P_\eps^{i}\cap B'|  \geq C l^p \norm{\Sigma_\eps(\cdot + \xi) -\Sigma_\eps}^p_{L^p(J;\R^{n\times n})}.
\end{align}

Combining~\eqref{est46} and \eqref{est47} gives~\eqref{keyestimate}.

\textit{Step~2b: Application of the Fr\'echet-Kolmogorov theorem.} 
To establish strong $L^p$-convergence of $(\Sigma_\eps)_\eps$ as $\eps\to0$, observe that in view of~\eqref{est2} and the uniform boundedness of $(u_\eps )_\eps$ in $W^{1,p}(\Omega;\R^{n})$, estimate~\eqref{approx_diffinclusion} turns into
\begin{align}\label{est_rigid}
\|\Sigma_\eps(\cdot +\xi) - \Sigma_\eps\|_{L^p(J;\R^{n\times n})} \leq Cl^{-p-n+1}(\xi +\eps^{\frac{\alpha}{p} -1}).
\end{align}
It is standard to verify (see~e.g.~\cite[Proof of Theorem 4.1]{FJM02} for an analogous argument) that then
\begin{align*}
\limsup_{|\xi|\rightarrow 0} \sup_{\eps>0}\|\Sigma_{\eps}(\cdot +\xi) - \Sigma_{\eps}\|_{L^p(J;\R^{n\times n})} = 0. 
\end{align*}
Hence, by Theorem~\ref{theo:FrechetKolmogorov}, there exist a subsequence (not relabeled) and a $\Sigma_0 \in L^p(J;\R^{n\times n})$ such that
\begin{align}\label{convergence_Sigmaeps}
\Sigma_{\eps} \to \Sigma_0\qquad \text{in $L^p(J;\R^{n\times n})$.}
\end{align}
Note that $\Sigma_0$ may still depend on the subsequence at this point. In Step~3, $\Sigma_0$ will be characterized in terms of the limit function $u$, which makes $\Sigma_0$ unique and the above argument independent of the choice of subsequences. 
Due to the strong $L^p$-convergence of $(\Sigma_{\eps})_\eps$, which preserves lengths and angles almost everywhere, we conclude that $\Sigma_0\in SO(n)$ a.e.~in $J'$.

\textit{Step~2c: Regularity of $\Sigma_0$.}
As a result of~\eqref{est_rigid}, we obtain an estimate on the difference quotients of $\Sigma_{\eps}$, precisely
\begin{align*}
\int_J \Big|\frac{\Sigma_{\eps}(t+ \xi) - \Sigma_{\eps}(t)}{\xi}\Big|^p \dd t \leq Cl^{-p-n+1}\big(1+ \xi^{-p} \eps^{\alpha-p}\big).
\end{align*}
Passing to the limit $j\to \infty$ results in
\begin{align}\label{est_diffquotient}
\int_J \Big|\frac{\Sigma_0(t+ \xi) - \Sigma_0(t)}{\xi}\Big|^p \dd t \leq Cl^{-p-n+1},
\end{align}
which shows that $\Sigma_0\in W^{1,p}(J;\R^{n\times n})$, 
see e.g.~\cite[Theorem~10.55]{Leo09}.

\textit{Step~3: Representation of the limit function $u$.}
Recall the definitions of $\sigma_\eps$ in~\eqref{sigmaeps} and $\Sigma_\eps$ in~\eqref{Sigmaeps}. With $\sigma_0(x) = \Sigma_0(x_n)x$ for $x\in Q$ one has that
\begin{align*}
\int_Q \abs{\sigma_{\eps}-\sigma_0}^p\dd{x} \leq \sum_{i\in I_{\eps}} \int_{P_{\eps}\cap Q} \abs{R_{\eps}^i-\Sigma_0(x_n)}^p\abs{x}^p\dd{x} \leq C \int_J |\Sigma_{\eps}-\Sigma_0|^p\dd{t},
\end{align*}
Then, by~\eqref{convergence_Sigmaeps}, 
\begin{align}\label{con33}
\sigma_{\eps} \to \sigma_0 \qquad\text{in $L^p(Q;\R^n)$,}
\end{align}
Since $b_\eps = w_\eps-\sigma_\eps = (w_\eps-u_\eps) + u_\eps -\sigma_\eps$ we find in view of~\eqref{convergence},~\eqref{con33} and the convergence $u_\eps\to u$ in $L^p(\Omega;\R^n)$ by the compact embedding of $W^{1,p}(Q;\R^n)$ into $L^p(Q;\R^n)$ that 
\begin{align*}
b_\eps\to u-\sigma_0 =:b\quad \text{in $L^p(Q;\R^n)$.} 
\end{align*}
Due to the regularity of $u$ and $\sigma_0$, it follows that $b\in W^{1,p}(Q;\R^n)$.  
Since $b_\eps$ is independent of the $x'$-variables, the same is true for $b$. 
Finally, defining 
\begin{align}\label{def_R}
R(x)=\Sigma_0(x_n)\quad \text{ for $x\in Q$,}
\end{align} 
proves the desired representation of $u$.
\end{proof}

\begin{remark}\label{rem:rigidity}
a) Setting $p=1$ in Remark~\ref{rem:necessity}\,b) in combination with Theorem~\ref{theo:rigidity_general} leads to the representation~\eqref{representation_u} with $R\in BV_{\rm loc}(\Omega;SO(n))$ and $b\in BV_{\rm loc}(\Omega;\R^n)$ satisfying $D'R=0$ and $D'b=0$, respectively. The reasoning is the same as for $p>1$, but instead of getting $\Sigma_0\in W^{1,1}(J;\R^{n\times n})$ from~\eqref{est_diffquotient}, we can only deduce that $\Sigma_0 \in BV(J; \R^{n\times n})$, see e.g.~\cite[Corollary 2.43]{Leo09}. \\[-0.2cm]

b) Notice that in view of~\eqref{def_R} and \eqref{est_diffquotient} it holds that
\begin{align}\label{est_diffquotient2}
\norm{R}^p_{W^{1,p}(Q;\R^{n\times n})} \leq C(1+ l^{-p}).
\end{align}
This estimate is not uniform for all cuboids $Q\subset \Omega$ as used in the proof of Theorem~\ref{theo:rigidity_general}. In fact, the bound becomes large for cuboids with small cross-section. One can therefore not expect in general that the weak derivatives of $R$ be $p$-integrable on all of $\Omega$.\\[-0.2cm] 

c) If $\Omega$ in Theorem~\ref{theo:rigidity_general} is of the form $\Omega=O\times I$ with $O\subset \R^{n-1}$ an open cube of side length $l>0$ and $I\subset \R$ an interval, then the proof shows that $R\in W^{1,p}(\Omega;SO(n))$, and hence also $b\in W^{1,p}(\Omega;\R^n)$, for any $p>1$. 

Indeed, let us choose intervals $J_k\subset\subset I$ for $k\in \N$ such that $J_k\subset J_{k+1}$ and $I=\bigcup_{k=1}^\infty J_k$ and set   $Q_k=O\times J_k$. Then by estimate~\eqref{est_diffquotient2}, 
\begin{align}\label{bound1}
\norm{R}_{W^{1,p}(Q_k;\R^{n\times n})}^p\leq C
\end{align}
with $C>0$ independent of $k$. Since the cuboids $Q_k$ exhaust $\Omega$, the uniform bound~\eqref{bound1} yields that $R\in W^{1,p}(\Omega;SO(n))$.
\end{remark}

The observation of Remark~\ref{rem:rigidity}\,c) can be extended to a larger class of Lipschitz domains. 
In fact, under suitable additional assumptions on $\Omega$, namely connectedness of cross-sections and a flatness property, which are introduced in Definitions~\ref{def:connected}~and~\ref{def:flat}, we can drop the restriction to local $W^{1,p}$-regularity of $R$ and $b$ in Theorem~\ref{theo:rigidity_general}, as Corollary~\ref{cor:rigidity} below shows.
\begin{definition}[Connectedness of cross-sections]\label{def:connected} 
An open set $\Omega\subset \R^n$ is called cross-section connected if for any $t\in \R$ the intersection $\Omega_t$ of $\Omega$ with the hyperplane $H_t=\{x\in \R^n: x_n =t\}$ is connected. 
\end{definition}

Clearly, every convex set is cross-section connected, but also cylinders and cones in $\R^n$ (oriented in $e_n$-direction) with non-convex cross section are. In Fig.~\ref{fig:domains}\,a), b) we give a two-dimensional example for illustration. 
An important property of domains $\Omega$ as in Definition~\ref{def:connected} is that any locally one-dimensional vector (and matrix) field in $e_n$-direction defined on $\Omega$ is already globally one-dimensional in $e_n$-direction, cf.~Lemma~\ref{lem:extension}. 

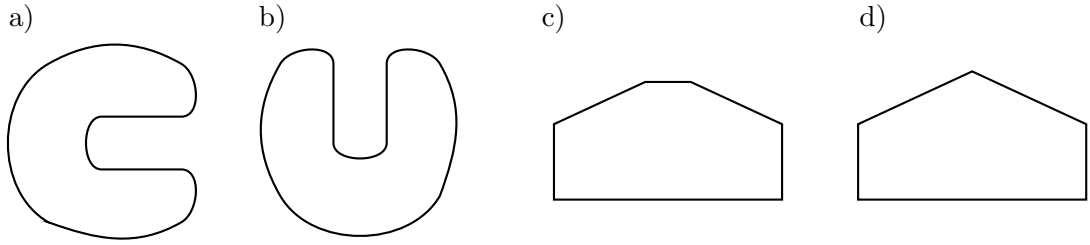
\begin{figure}
\begin{tikzpicture}

\def\material{(0.5,0) to[out=340, in=210] (3,0) to[out=30, in=0] (3,1) to (1.5,1) to[out=180, in=180] (1.5,2) to (3,2) to[out=0, in=330] (3,3) to[out=150, in=30] (0.5,3) to[out=210, in=150] cycle}

\begin{scope} 
        \draw[ thick, black ,fill opacity=.3, scale = 0.7] \material;
\end{scope}

\begin{scope}[xshift =5.5cm]
        \draw[ thick, black ,fill opacity=.3, scale = 1, rotate=90,scale=0.7] \material;
\end{scope}
\begin{scope}[yshift = -0.7cm, xshift = 7cm]
	\def\material{ (0,1) -- (3,1) -- (3,2) -- (1.8,2.56) -- (1.2,2.56) -- (0,2) -- cycle}
        \draw[ thick, black ,fill opacity=.3, scale = 1] \material;
\end{scope}
\begin{scope}[yshift = -0.7cm, xshift =11cm]

\def\material{ (0,1) -- (3,1) -- (3,2) -- (1.5,2.7) -- (0,2) -- cycle}
        \draw[ thick, black ,fill opacity=.3, scale = 1] \material;
\end{scope}
\draw (0,2.7) node{a)};
\draw (3.3,2.7) node{b)};
\draw (7,2.7) node{c)};
\draw (11.2,2.7) node{d)};
\end{tikzpicture}
\caption{Illustration of Definitions \ref{def:connected} and \ref{def:flat}: Examples of bounded Lipschitz domains that are a) cross-section connected, b) not cross-section connected, c) flat, d) not flat. }
\label{fig:domains}
\end{figure}

\begin{definition}[Flatness]\label{def:flat} 
We call an open set $\Omega\subset \R^n$ flat, if for all $t\in \R$ the intersection of $\overline{\Omega}$ with the hyperplane $H_t=\{x\in \R^n: x_n=t\}$ is either empty or has nonempty relative interior.
\end{definition}
The intuitive geometric interpretation of flatness of bounded domains is that it rules out sets with sharp or rounded corners and peaks pointing in the direction of $e_n$. Simple examples include cylinders with axis parallel to $e_n$, whereas cones with the same orientation are not flat, see also Fig.~\ref{fig:domains}\,c), d).  
A bounded Lipschitz domain $\Omega\subset\R^n$ does in general not satisfy the condition of Definition~\eqref{def:flat}, but it can be turned into a flat Lipschitz domain by cutting it off on top and bottom, i.e.,~by taking $(\R^{n-1}\times (a, b))\cap \Omega$, where $a, b\in\R$ with $a<b$ are such that the cross sections $\Omega_a$ and $\Omega_b$ are non-empty. 
\begin{corollary}\label{cor:rigidity} 
In addition to the assumptions of Theorem~\ref{theo:rigidity_general}, let $\Omega\subset \R^n$ be a flat and cross-section connected Lipschitz domain. Then 
the representation~\eqref{representation_u} holds with $R\in W^{1,p}(\Omega;SO(n))$ and $b\in W^{1,p}(\Omega;\R^n)$.
\end{corollary}

\begin{proof}
Let $Q_\Omega$ be the smallest open cuboid containing $\Omega$ and let $a, b\in \R$ with $a<b$ and $O_\Omega\subset \R^{n-1}$ be such that $Q_\Omega=O_\Omega\times J_\Omega$ with $J_\Omega=(a, b)$. 
We observe first that due to the connectedness of the cross-sections of $\Omega$, the map $R$ from~\eqref{representation_u} is globally one-dimensional in $e_n$-direction and can thus be identified with a one-dimensional function $\Sigma \in W^{1,p}_{\rm loc}(J_\Omega;SO(n))$, see Lemma~\ref{lem:extension} and~Remark~\ref{rem:extension}.

Moreover, since $\Omega$ is a flat Lipschitz domain there exist $x_a\in \Omega_a$ and $x_b\in \Omega_b$ along with open cuboids $Q_a=O\times (a, a+r)$ and $Q_b=O\times (b-r,b)$ of height $r>0$ and cross-section $O\subset \R^{n-1}$ such that $Q_a\cap Q_\Omega\subset \Omega$ and $Q_b\cap Q_\Omega\subset \Omega$. Applying Remark~\ref{rem:rigidity}\,c) to the restrictions $R_a=R|_{Q_a}$ and $R_b=R|_{Q_b}$ gives that $R_a\in W^{1,p}(Q_a;SO(n))$ and $R_b\in W^{1, p}(Q_b;SO(n))$, which correspond to elements in $\Sigma_a\in W^{1,p}(a, a+r;SO(n))$ and $\Sigma_b\in W^{1,p}(b-r, b;SO(n))$, respectively. 
Hence, $\Sigma\in W^{1,p}(J_\Omega;SO(n))$ and $R\in W^{1,p}(Q_\Omega;SO(n))$, thus also $R\in W^{1,p}(\Omega;SO(n))$. 

Since $b=u-Rx$ with $u\in W^{1,p}(\Omega;\R^m)$, one immediately gets the desired statement for $b$.
\end{proof}

We conclude this section with the following specialization of Corollary~\ref{cor:rigidity}, which involves the additional condition that the limit map is locally volume preserving. 

\begin{corollary}\label{cor:asymptotic_rigidity}
In addition to the assumptions on $\Omega$, $(u_\eps)_\eps$ and $u$ in Corollary~\ref{cor:rigidity}, let $u \in W^{1,r}(\Omega;\R^n)$ for $r \geq n$ be such that $\det \nabla u = 1$ a.e.~in $\Omega$. 

Then the limit representation in~\eqref{representation_u} holds with $Re_n$ constant. If $\Omega$ is simply connected, one has in particular that 
\begin{align}\label{representation_nablau}
\nabla u=QS(\Ibb +a\otimes e_n),
\end{align}
where $Q\in SO(n)$, $S=\diag (S', 1)$ with $S'\in W^{1,p}(\Omega; SO(n-1))$ satisfying $\nabla' S'=0$ 
and $a\in L^{\max\{r, p\}}(\Omega;\R^n)$ with $D'a=(S')^T(\partial_n S')$ and  $a_n=0$.
\end{corollary}

\begin{proof}[Proof of Corollary~\ref{cor:asymptotic_rigidity}] By Theorem~\ref{theo:rigidity_general}, we know that $u$ has the representation~\eqref{representation_u}. Hence, 
\begin{align*}
\nabla u = R+(\partial_n R) x \otimes e_n +\partial_n b\otimes e_n = R+ \tilde a\otimes e_n = R(\Ibb +a\otimes e_n)
\end{align*}
with $\tilde a= (\partial_n R)x + \partial_n b$ and $a=R^T\tilde a$. Since $\det \nabla u=\det (R(\Ibb+a\otimes e_n))= 1+a_n$, we conclude in view of the local volume preservation constraint that $a_n=0$.

Differentiating the identity $0= a_n= \tilde a\cdot Re_n$ with respect to the $i$th variable for $i\in \{1, \ldots, n-1\}$, while taking into account that $\nabla' R=0$ and $\nabla' b=0$, implies that $\partial_n (Re_i)\cdot Re_n =0$.
Since $Re_i$ is orthogonal on $Re_n$ pointwise almost everywhere, it follows from the product rule that 
\begin{align*}
\partial_n (Re_n) \cdot Re_i=0\qquad \text{for $i=1, \ldots, n-1$.}
\end{align*}
Together with 
\begin{align*}
0= \tfrac{1}{2}\partial_n |Re_n|^2 = \partial_n (Re_n) \cdot Re_n,
\end{align*}
we obtain that $\partial_n (Re_n)=0$. Hence, $Re_n$ is constant, and $R$ splits multiplicatively into the product of $Q$ and $S$ as in the statement.

Finally, the restriction on the distributional derivatives of $a$ with respect to the first ${n-1}$ variables follows via straightforward calculation from the gradient structure of $\nabla u$, which requires that $\curl \nabla u= 0$.  
\end{proof}

\begin{remark}\label{rem:2d}
If $n=2$, the gradient representation of $u$ in~\eqref{representation_nablau} becomes
\begin{align*}
\nabla u=Q(\Ibb+ \gamma e_1\otimes e_2),
\end{align*}
with $Q\in SO(2)$ and $\gamma\in L^p(\Omega)$ with $\partial_1\gamma=0$, cf.~also~\cite[Proposition~2.1]{ChK17}. 
In the two-dimensional setting, the class of limit deformations $u$ of $\Omega$ is highly restricted, in fact, only horizontal shearing and global rotation can occur.
\end{remark}


\section{sufficiency statement in Theorem~\ref{prop:rigidity}}\label{sec:suff}

Our starting point in this section are functions $u\in W^{1, p}(\Omega;\R^n)$ with gradients of the form
\begin{align}\label{eq:form_limitgradient}
\nabla u(x) = R(x) +  \partial_n R(x) x  \otimes e_n + d(x)\otimes e_n, \qquad \text{$x\in \Omega$,}
\end{align}
where $R\in W^{1,p}(\Omega;SO(n))$ and $d \in L^p(\Omega;\R^n)$ with $\nabla' R=0$ and $D' d =0$, respectively. 
If not mentioned otherwise, $1<p<\infty$ and $\Omega\subset \R^n$ is a bounded domain. 

We will show how such $u$ (under suitable technical assumptions) can be approximated in the sense of weak convergence in $W^{1,p}(\Omega;\R^{n\times n})$ by functions $u_\eps$ that are defined on a layered domain with length scale of oscillations $\eps$
and coincide with rigid body motions on the stiff components. 
This in particular proves
~Theorem~\ref{prop:rigidity}\,$ii)$. 

Before stating the general result, let us consider a simple example for motivation. 
If $u$ is affine, then $\nabla u=F$ for some $F\in \R^{n\times n}$ and there exist a matrix $R_F\in SO(n)$ and a vector $d_F\in \R^n$ such that 
$\nabla u = F= R_F+d_F\otimes e_n$. This motivates the definition 
\begin{align}\label{def:Acal}
\Acal=\{F\in \R^{n\times n}:  F = R_F + d_F \otimes e_n \text{ with $R_F \in SO(n)$ and $d_F \in \R^n$}\}.
\end{align}
Moreover, we set 
\begin{align}\label{Flambda}
F_\lambda=R_F+\tfrac{1}{\lambda} d_F\otimes e_n \quad \text{for $F\in \Acal$.}
\end{align}
In the affine case, the construction of a suitable approximation is particularly simple. The idea is to compensate for the stiff layers by performing stronger deformations on the softer layers, which leads to the following laminate construction. For $\eps>0$, let $v_\epsilon^F \in  W^{1,\infty}(\Omega;\R^n)$ be such that 
\begin{align}\label{eq:affine}
\nabla v_\epsilon^F =  R_F \mathbbm{1}_{\eps\Ystiff\cap \Omega} + F_\lambda \mathbbm{1}_{\eps\Ysoft\cap \Omega.} = \begin{cases} R_F &\text{on~} \epsilon\Ystiff\cap \Omega, \\
F_\lambda &\text{on~}  \epsilon\Ysoft\cap \Omega.
\end{cases}  
\end{align}
Then, $\nabla v_\eps^F \in SO(n)$ a.e.~in $\eps\Ystiff \cap \Omega$ and $\nabla v_\eps^F = R+ \frac{1}{\lambda} \mathbbm{1}_{\eps\Ysoft} d\otimes e_n\weakly \nabla u$ in $L^p(\Omega;\R^{n\times n})$ as a consequence of the weak convergence of highly oscillating sequences (see e.g.~\cite[Section~2.3]{CiD99}). Finally, we set $u_\eps=v_\eps^F$ for all $\eps$ to obtain the desired approximating functions in this special case.

The construction behind the general approximation result is inspired by the case of affine limits. In view of~\eqref{Flambda}, we have 
\begin{align}\label{Ulambda}
(\nabla u)_\lambda(x) = (\nabla u(x))_\lambda= R(x)+\tfrac{1}{\lambda}(\partial_n R)(x)x \otimes e_n + \tfrac{1}{\lambda}d(x)\otimes e_n, \quad x\in \Omega. 
\end{align}

\begin{proposition}\label{prop:restricted_approximation}
Let $\Omega\subset \R^n$ be a bounded, flat and cross-section connected Lipschitz domain and let $u\in W^{1,p}(\Omega;\R^n)$ with $\nabla u$
 as in~\eqref{eq:form_limitgradient}.
Then there exists a sequence $(u_\eps)_\eps\subset W^{1,p}(\Omega;\R^n)$ with $\nabla u_\eps\in SO(n)$ a.e.~in $\eps\Ystiff\cap \Omega$ such that $\nabla u_\eps\weakly \nabla u$ in $L^{p}(\Omega;\R^{n\times n})$.

More specifically, there is $(R_\epsilon)_\epsilon \subset W^{1,p}(\Omega;SO(n))$ with $\nabla' R_\eps =0$ on $\Omega$ and $\partial_n  R_\eps =0$ on $\epsilon \Ystiff\cap \Omega$ such that
\begin{align} \label{eq:approximation1}
\nabla u_\epsilon = R_\epsilon \text{~in~} \epsilon\Ystiff \cap \Omega, 
\end{align}
and 
\begin{align} \label{eq:approximation2}
R_\epsilon \rightharpoonup R \text{~in~} W^{1,p}(\Omega;\R^{n\times n}) \quad \text{and} \quad \|\nabla u_\epsilon - (\nabla u)_\lambda\|_{L^p(\epsilon\Ysoft \cap\Omega;\R^{n\times n})} \rightarrow 0 \text{~as~} \epsilon \rightarrow 0.
\end{align}
\end{proposition}

\begin{remark}\label{rem:domain}
The same result still holds also under relaxed conditions on a bounded Lipschitz domain $\Omega$, namely when $\Omega$ can be partitioned into finitely many components that are flat and cross-section connected. 
More details can be found in~\cite[Section~4.2]{Chr18}. 
\end{remark}

\begin{proof}
Let $Q_\Omega$ denote the smallest cuboid containing $\Omega$.
By~\eqref{eq:form_limitgradient} and Lemma~\ref{lem:extension} (see also Remark~\ref{rem:extension}\,b)), we may assume after constant extension orthogonal to $e_n$ that $R\in W^{1, p}(Q_\Omega;SO(n))$ is globally one-dimensional in $e_n$-direction and continuous. 
Upon writing $Q_\Omega=O_\Omega\times J_\Omega$ with $O_\Omega\subset \R^{n-1}$ an open cuboid and $J_\Omega\subset \R$ an open, bounded interval, there is a one-dimensional function $\Sigma\in W^{1,p}(J_Q; SO(n))$ such that $R(x)= \Sigma(x_n)$ for~$x\in Q_\Omega$. 

Let $(\Sigma_\eps)_\eps \subset W^{1,p}(J_\Omega;SO(n))$ be the approximating sequence for $\Sigma$ resulting from Lemma~\ref{lem:approximation_R} below, that is, $\Sigma_\eps \rightharpoonup \Sigma$ in $W^{1,p}(J_\Omega;\R^{n\times n})$ and $\Sigma_\eps' = 0$ in $\eps I_{\rm stiff} \cap J_\Omega$. Moreover, the convergence~\eqref{eq:Approximation} holds. 
We set $R_\eps(x)=\Sigma_\eps(x_n)$ for $x\in Q_\Omega$, so that $R_\eps\in W^{1,p}(Q_\Omega;SO(n))$ with $\nabla' R_\eps=0$, and
define 
\begin{align*}
U_\epsilon = R_\epsilon\mathbbm{1}_{\epsilon\Ystiff \cap Q_\Omega} +U_{\lambda,\eps}\mathbbm{1}_{\epsilon\Ysoft\cap Q_\Omega},
\end{align*}
where \begin{align*}
U_{\lambda,\eps}(x) = R_\eps(x)+ (\partial_n R_\epsilon)(x) x\otimes e_n + \tfrac{1}{\lambda} d(x)\otimes e_n, \quad x\in Q_\Omega.
\end{align*}
We claim that for each $\eps >0$ the function $U_\eps$ has gradient structure, meaning that there exists a potential $u_\eps\in W^{1,p}(Q_\Omega;\R^n)$ with $\nabla u_\eps = U_\eps$. To see this, it suffices to show that the the distributional curl of $U_\eps$ vanishes on $Q_\Omega$. We remark that $Q_\Omega$ as a cuboid is simply connected. Indeed, let $\varphi \in C^\infty_c(Q_\Omega;\R^n)$ and $k, l\in \{1, \ldots, n\}$ with $k<l$. Due to $\nabla'R_\eps=0$ and $\nabla' d=0$, one obtains in the case $l<n$ that
\begin{align*}
\int_{Q_\Omega} U_\epsilon e_k \cdot \partial_l \varphi \dd x - \int_{Q_\Omega} U_\eps e_l \cdot \partial_k \varphi \dd x =\int_{Q_\Omega} R_\epsilon e_k \cdot \partial_l \varphi \dd x - \int_{Q_\Omega} R_\eps e_l \cdot \partial_k \varphi \dd x = 0,
\end{align*} 
and for $l=n$, along with $\partial_n R_\eps=0$ on $\eps\Ystiff\cap Q_\Omega$, that
\begin{align*}
\int_{Q_\Omega} U_\epsilon e_k \cdot \partial_n \varphi \dd x - \int_{Q_\Omega} U_\eps e_n \cdot \partial_k \varphi \dd x
&= \int_{Q_\Omega} R_\eps e_k \cdot \partial_n \varphi \dd x  - \int_{Q_\Omega} \mathbbm{1}_{\eps\Ysoft} (\partial_n R_\eps) x  \cdot \partial_k \varphi \dd x\\
&= \int_{Q_\Omega} R_\eps e_k\cdot \partial_n \varphi \dd x + \int_{Q_\Omega} (\partial_n R_\eps)e_k\cdot \varphi \dd x=0
\end{align*} 
Thus, $\curl U_\eps=0$ as desired.  

After restricting $u_\eps$ and $R_\eps$ to $\Omega$, the statements~\eqref{eq:approximation1} and~\eqref{eq:approximation2} follow now directly from the properties of the sequence $(\Sigma_\eps)_\eps$. 
\end{proof}

The proof of the previous proposition builds on the following structure preserving approximation result for one-dimensional functions with values in the set of rotations.  
Let us denote by $\Isoft$ the $1$-periodic extensions of the interval $(0, \lambda)$ to the real line, which
corresponds to a one-dimensional section of $\Ystiff$ in $e_n$-direction, that is~$\Ystiff = \R^{n-1}\times \Istiff$. Besides, we set $\Istiff= \R^n\setminus \Isoft$.

\begin{lemma} \label{lem:approximation_R}
Let $J\subset\R$ be an open and bounded interval, $1\leq p<\infty$ and $\Sigma\in W^{1,p}(J;SO(n))$. Then there exists a sequence $(\Sigma_\eps)_\eps\subset W^{1,p}(J;SO(n))$ with 
\begin{align*}
\Sigma_\eps'= 0  \quad \text{a.e. in~} \epsilon\Istiff \cap J
\end{align*}
such that $\Sigma_\eps \weakly \Sigma$ in $W^{1,p}(J; \R^{n\times n})$. Furthermore, 
\begin{align} \label{eq:Approximation}
 \| \Sigma_\epsilon'- \tfrac{1}{\lambda} \Sigma'\|_{L^p(\epsilon\Isoft \cap J; \R^{n\times n})} \rightarrow 0 \quad \text{as~} \epsilon \rightarrow 0. 
\end{align}
\end{lemma}

\begin{proof}
Instead of trying to approximate $\Sigma$ directly with $SO(n)$-valued functions, it seems easier to parametrize $\Sigma$ in a suitable way. Intuitively speaking, the idea is to stop the parametrization on the stiff layers and accelerate it on the softer ones. 

More precisely, for every $\eps>0$, take $\varphi_\eps: \R \to \R$ as the piecewise affine function defined by
\begin{align*}
\varphi_\eps(t) = \eps \lceil t\rceil \quad \text{for $t \in \eps\Istiff$,}
\end{align*}
and by linear interpolation on $\eps\Isoft$, see Fig~\ref{fig:varphi}. By construction, one has that
\begin{align}\label{derivative}
\varphi_\eps' = \tfrac{1}{\lambda} \qquad \text{ on $\eps \Isoft$,}
\end{align}
and $(\varphi_\epsilon)_\epsilon$ converges locally uniformly to the identity function on $\R$ for $\eps\to 0$.

\begin{figure}
\begin{tikzpicture}
\draw[thick] (0,0) -- (1, 1.5) -- (1.5,1.5) -- (2.5,3) -- (3,3);
\draw[dashed] (-0.15,-0.15) -- (3.15,3.15);
\draw (1.3,2.3) node{$\varphi_\epsilon$};
\draw (2.6,1.9) node{$\id_{\R}$};
\draw (3.9,-0.03) node{$t$};

\draw[->] (-0.5,0) -- (3.5,0);
\draw[->] (0,0) -- (0,3.2);
\draw (0,-0.1) -- (0,0.1);
\draw (1,-0.1) -- (1,0.1);
\draw (1.5,-0.1) -- (1.5,0.1);
\draw (2.5,-0.1) -- (2.5,0.1);
\draw (3,-0.1) -- (3,0.1);
\draw (-0.1,1.5) -- (0.1,1.5);
\draw (-0.3,1.5) node{$\epsilon$};
\draw (-0.3,3.0) node{$\R$};

\draw (0,-0.3) node{$0$};
\draw (1,-0.3) node{$\epsilon \lambda$};
\draw (1.5,-0.3) node{$\epsilon \vphantom{\lambda}$};
\draw (3,-0.3) node{$2\epsilon$};
\end{tikzpicture}
\caption{Illustration of the re-parametrization function $\varphi_\eps$. }\label{fig:varphi}
\end{figure}
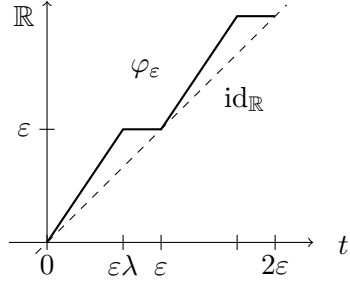

First, we extend the function $\Sigma$ from $J$ to an open real interval $J'$ that contains $J$ compactly. 
In fact, via reflection one obtains $\Sigma \in W^{1,p}(J';SO(n))$ (not renamed) with 
\begin{align*}
\norm{\Sigma}_{W^{1,p}(J', SO(n))}\leq c \norm{\Sigma}_{W^{1,p}(J, SO(n))}<\infty,
\end{align*}
where $c>0$ depends only on $J'$. 

Next, we define $\Sigma_\eps:J\to SO(n)$ by $\Sigma_\eps = \Sigma \circ \varphi_\eps$
for sufficiently small $\eps$.
Notice that $\Sigma_\eps$ is well-defined, since $\varphi_\eps(J)=\varphi_\eps(\eps\Isoft\cap J)\subset J'$ if $\eps$ is small enough. As the composition of an absolutely continuous function with a monotone Lipschitz function, $\Sigma_\eps$ is absolutely continuous. In particular, the chain rule holds (see e.g.~\cite[Theorem 3.44]{Leo09}), i.e.
\begin{align}\label{chainrule}
\Sigma_\eps' = (\Sigma'\circ\varphi_\eps) \varphi_\eps',
\end{align} 
and thus, $\Sigma_\eps\in W^{1,p}(J; SO(n))$. 
Since $\Sigma_\eps\to \Sigma$ pointwise and the functions $|\Sigma_\eps|^2\leq n$ a.e.~in $J$, it follows from Lebesgue's dominated convergence theorem that $\Sigma_\eps\to \Sigma$ in $L^p(J;\R^{n\times n})$. 

For the asserted weak convergence of $(\Sigma_\eps)_\eps$ in $W^{1,p}(J;\R^{n\times n})$, it suffices according to Urysohn's lemma to show that the sequence $(\Sigma_\eps')_\eps$ is uniformly bounded in $L^p(J;\R^{n\times n})$. Indeed,  
\begin{align*}
\norm{\Sigma_\eps'}_{L^p(J;\R^{n\times n})}^p & = \int_{\eps I_{\rm soft}\cap J} |\Sigma'(\varphi_\eps)\varphi_\eps'|^2\dd t =  \frac{1}{\lambda^2} \int_{\eps I_{\rm soft}\cap J} |\Sigma'(\varphi_\eps)|^2\dd t \\ & = \frac{1}{\lambda^2} \sum_{i\in \Z} \int_{\eps(i, i+\lambda)\cap J} |\Sigma'(\varphi_\eps)|^2\dd{t}  \leq \frac{1}{\lambda^2} \sum_{i\in \Z} \int_{\eps(i, i+1)\cap J'} |\Sigma'|^2 |\varphi_\eps'|^{-1} \dd t \\ &= \frac{1}{\lambda} \norm{\Sigma'}^2_{L^2(J';\R^{n\times n})}.
\end{align*}
Here we have exploited~\eqref{chainrule} and~\eqref{derivative}, the fact that $\Sigma_{\eps}$ is constant on $\eps \Istiff$, as well as the chain rule and transformation formula on the (finitely many) connected components of $\eps I_{\rm soft}$, where the restriction of $\varphi_\eps$ is invertible.

To show~\eqref{eq:Approximation}, we approximate $\Sigma'$ in $L^p(J';\R^{n\times n})$ by a sequence $(g_j)_j\subset C^\infty_c(J';\R^{n\times n})$. By change of variables on the connected components of $\eps I_{\rm soft}$ it follows that
\begin{align*}
 \norm{\Sigma'\circ \varphi_\eps - g_j\circ \varphi_\eps}_{L^p(\eps I_{\rm soft}\cap J;\R^{n\times n})}  \leq  \norm{\Sigma' - g_j}_{L^p(J';\R^{n\times n})},
 \end{align*}
and therefore
\begin{align}\label{eq543}
\norm{\Sigma'\circ \varphi_\eps -\Sigma'}_{L^p(\eps I_{\rm soft}\cap J;\R^{n\times n})} \leq \norm{g_j\circ \varphi_\eps - g_j}_{L^p(J;\R^{n\times n})}  +2 \norm{g_j-\Sigma'}_{L^p(J';\R^{n\times n})}.
\end{align}
Since $g_j\circ \varphi_\eps\to g_j$ in $L^p(J;\R^{n\times n})$ for every $j\in \N$ by dominated convergence, 
passing to the limits $\eps\to 0$ and $j\to \infty$ (in this order) in~\eqref{eq543} proves~\eqref{eq:Approximation}. 
\end{proof}

\section{Homogenization of layered high-contrast materials}
 
Before proving Theorem~\ref{theo:hom}, formulated below, we introduce the setting and precise assumptions. 
Throughout this section, $\Omega\subset\R^n$ is a bounded Lipschitz domain that satisfies the flatness condition and connectedness property of Definitions~\ref{def:flat} and~\ref{def:connected}, respectively, and $p > n$. 
For $\eps > 0$ and $\alpha>0$ we consider the heterogeneous energy density $W_\eps^\alpha: \Omega\times \R^{n\times n}\to [0, \infty)$ given by
\begin{align*}
W_\eps^\alpha(x, F) = \begin{cases} \displaystyle \eps^{-\alpha} \Wstiff(F) &\text{if~} x \in \eps\Ystiff\cap \Omega,   \\[0.2cm] \Wsoft (F) &\text{if~} x \in \eps\Ysoft\cap \Omega, \end{cases}
\end{align*}
where $\Wstiff, \Wsoft:\R^{n\times n}\to [0, \infty)$ are continuous functions that satisfy the following conditions regarding convexity, growth and coercivity, and local Lipschitz continuity:
\vspace{0.1cm}
\begin{itemize}
\item[$(H1)$] $\Wsoft^{\rm qc}$ is polyconvex;\\[-0.3cm]
\item[$(H2)$] $c|F|^p - \frac{1}{C} \leq \Wsoft(F) \leq C(1+|F|^p)$ for all $F\in \R^{n\times n}$ with constants $C, c >0$; \\[-0.3cm]
\item[$(H3)$] $|\Wsoft(F) -\Wsoft(G)| \leq L(1+|F|^{p-1}+|G|^{p-1})|F-G|$ for all $F, G\in \R^{n\times n}$ with $L>0$;\\[-0.3cm]
\item[$(H4)$] $\Wstiff(F)\geq k \dist^{p}(F, {SO(n)})$ for all $F\in \R^{n\times n}$ with a constant $k>0$.
\end{itemize}
\vspace{0.1cm}

An equivalent way of expressing $(H1)$ is by 
\begin{align}\label{H1_equi}
\Wsoft^{\rm qc} = \Wsoft^{\rm pc},
\end{align} 
where $\Wsoft^{\rm qc}$ and $\Wsoft^{\rm pc}$ are the quasiconvex and polyconvex envelopes of $\Wsoft$, that is, the largest quasiconvex and polyconvex functions below $\Wsoft$. 
For a detailed introduction to generalized notions of convexity and the corresponding generalized convexifications we refer to~\cite{Dac08}. Let us just recall briefly that a continuous function $W:\R^{n\times n}\to \R$ with standard $p$-growth (i.e., with an the upper bound  as in $(H2)$) is quasiconvex if for any $F\in \R^{n\times n}$,
\begin{align}\label{qc}
\inf_{\varphi\in W_0^{1,p}((0,1)^n;\R^n)} \dashint_{(0,1)^n} W(F+ \nabla \varphi )\dd{x}\geq W(F). 
\end{align}
Moreover, a continuous $W:\R^{n\times n}\to \R$ is polyconvex if there exists a convex function $g:\R^{\tau(n)}\to \R$ such that 
\begin{align*}
W(F)=g(\Mcal(F))\qquad \text{for all $F\in \R^{n\times n}$,}
\end{align*} 
where $\Mcal(F)\in \R^{\tau(n)}$ with $\tau(n) = \sum_{i=1}^n \begin{pmatrix}n \\ k\end{pmatrix}$ is the vector of minors of $F$. 

We remark that explicit formulas for quasiconvex envelopes are in general hard to obtain. This is why  quasiconvexifications are rather rare in the literature, see e.g.~\cite{LDR95, CoT05, CDK13} for a few examples (including extended-valued densities). A common strategy is to determine upper and lower bounds in terms of rank-one and polyconvex envelopes and to show that the latter two match. Hence, in those cases where relaxations are explicitly known, $(H1)$ is usually satisfied. 

\begin{example}\label{ex:SaintVenantKirchhoff}
Let $n=2$ or $n=3$. The Saint Venant-Kirchhoff stored energy function,
\begin{align*}
W_{SK}(F) = \frac{\lambda}{4} |F^TF - \Ibb|^2 +\frac{\mu}{8}(|F|^2-n)^2,\qquad F\in \R^{n\times n},
\end{align*}
with the Lam\'e constants $\lambda, \mu>0$, is one of the simplest energy densities of relevance in hyperelasticity (see e.g.~\cite[Section 28]{Gur81}), and meets requirements for $W_{\rm soft}$.
It is straightforward to see that $W_{SK}$ has standard growth $(H2)$ with $p=4$ and is locally Lipschitz continuous in the sense of $(H3)$. 
In~\cite{LDR95}, Le Dret and Raoult give an explicit expression of the quasiconvexification $W^{\rm qc}_{SK}$, which coincides with the convex, polyconvex and rank-one convex envelopes. Thus, in particular, $(H1)$ is satisfied, too. 
\end{example}

Furthermore, let $E_\eps: L_0^p(\Omega;\R^n) \rightarrow \R \cup \{\infty\}$ be the integral functional with density $W_\eps^\alpha$, i.e.~
\begin{align}\label{energy_inhom_epsilon}
E_\eps(u) = \int_\Omega W_\eps^\alpha(x, \nabla u)\dd x
\end{align}
if $u\in W^{1, p}(\Omega;\R^n)$ and $E_\eps(u)=\infty$ otherwise in $L_0^p(\Omega;\R^n)$.  

Recalling that
$\Acal=\{F\in \R^{n\times n}:  F = R_F + d_F \otimes e_n \text{ with $R_F \in SO(n)$ and $d_F \in \R^n$}\}$ (cf.~\eqref{def:Acal}), we define for $F\in \Acal$,
\begin{align}\label{Whom}
W_{\rm hom}(F) =  \lambda \Wsoft^{\rm{qc}}(F_\lambda)= \lambda \inf_{\varphi\in W_0^{1,p}((0,1)^n;\R^n)} \dashint_{(0,1)^n} \Wsoft \big(F_\lambda + \nabla \varphi \big)\dd{x},
\end{align}
where $F_\lambda =R_F + \tfrac{1}{\lambda}d_F\otimes e_n =\tfrac{1}{\lambda}(F-(1-\lambda)R_F) \in \Acal$.

 Now we are ready to formulate the main theorem of this section. Theorem~\ref{theo:hom} provides a characterization of the effective behavior of the bilayered materials modeled by~\eqref{energy_inhom_epsilon} by homogenization via $\Gamma$-convergence for vanishing layer thickness. The limit problem shows a splitting of the effects of the heterogeneities and relaxation of microstructures on the softer components. With regards to homogenization, the resulting formulas are explicit and can be expressed in terms of the relative layer thickness.  Provided the relaxation of $W_{\rm soft}$ is known, $W_{\rm hom}$ is even fully explicit.  

\begin{theorem}\label{theo:hom}
 If $\alpha>p$, the family $(E_\eps)_\eps$ as in~\eqref{energy_inhom_epsilon} converges in the sense of $\Gamma$-convergence regarding the strong $L^p$-topology to the limit functional $E_{\rm hom}: L^p_0(\Omega;\R^n) \to \R_\infty$ given by
\begin{align*}
E_{\rm hom}(u) = \begin{cases} \displaystyle \int_\Omega W_{\rm hom}(\nabla u)\dd x & \text{if $u(x) = R(x)x+b(x)$ with $R \in W^{1,p}(\Omega;SO(n))$ such} \\[-0.2cm] & \text{that $\nabla' R=0$ and $b \in W^{1,p}(\Omega;\R^n)$ such that $\nabla' b=0$,} \\ 
\infty & \text{otherwise.}
\end{cases}
\end{align*}
Precisely, this means that the following two conditions are satisfied:
\begin{itemize}
\item[$i)$] \textit{(Lower bound)} For each $u \in L^p_0(\Omega;\R^n)$ and any sequence $(u_\epsilon)_\epsilon \subset L^p_0(\Omega;\R^n)$ with $u_\epsilon \rightarrow u$ in $L^p(\Omega;\R^n)$ as $\epsilon \rightarrow 0$ it holds that
\begin{align*}
\liminf_{\epsilon \rightarrow 0} E_\epsilon(u_\epsilon) \geq E_\mathrm{hom}(u);
\end{align*}
\item[$ii)$] \textit{(Existence of recovery sequence)} For each $u \in L^p_0(\Omega;\R^n)$ there exists a sequence $(u_\epsilon)_\epsilon \subset L^p_0(\Omega;\R^n)$ with $u_\epsilon \rightarrow u$ in $L^p(\Omega;\R^n)$ as $\epsilon \rightarrow 0$ such that
\begin{align*}
\lim_{\epsilon \rightarrow 0} E_\epsilon(u_\epsilon) = E_\mathrm{hom}(u). 
\end{align*} 
\end{itemize}

Moreover, any sequence $(u_\eps)_\eps\subset L^p_0(\Omega;\R^n)$ of uniformly bounded energy for $(E_\eps)_\eps$, that is~$E_\eps(u_\eps)<C$ for all $\eps>0$, is relatively compact in $L^{p}(\Omega;\R^n)$. 
\end{theorem}

\begin{remark}\label{rem:homogenization} 

a) If $\Wsoft$ is convex, then $\Wsoft^{\rm qc}=\Wsoft^{\rm c}= \Wsoft$, so that $W_{\rm hom}(F) = \lambda \Wsoft(F_\lambda)$ for $F\in \Acal$. In this case, the proof of Theorem~\ref{theo:hom} can be simplified as indicated below.\\[-0.2cm]

b) It is well-known that the definition of quasiconvexity in~\eqref{qc}, as well as the representation formula for the quasiconvex envelope $W^{\rm qc}$, is independent of the choice of the domain, see e.g.~\cite[Proposition 5.11]{Dac08}. Therefore, we have for any open set $O\subset \R^n$ that
\begin{align*}
\Wsoft^{\rm qc}(F) = \inf_{\varphi\in W_0^{1,p}(O;\R^n)}\dashint_{O} \Wsoft(F+ \nabla \varphi)\dd{y}, \qquad F\in \R^{n\times n}. 
\end{align*}
Alternatively, $\Wsoft^{\rm qc}$ can be expressed with the help periodic perturbations on a cube $Q\subset\R^n$ as 
\begin{align*}
\Wsoft^{\rm qc}(F) = \inf_{\varphi\in W_{\#}^{1,p}(Q;\R^n)} \dashint_{Q} \Wsoft(F+\nabla \varphi)\dd{y}, \qquad F\in \R^{n\times n},
\end{align*} 
see e.g.~\cite[Proposition 4.19]{Mul99} or~\cite[Proposition 5.13]{Dac08}.\\[-0.2cm] 

c) The homogenized energy density $\Whom$ is non-negative and inherits the property~$(H2)$ from $\Wsoft$. This follows from the fact that $\Wsoft^{\rm qc}$ has standard $p$-growth, because  $\Wsoft$ has, along with the estimate  
\begin{align}\label{est37}
\tfrac{1}{2 \lambda} |F| - \tfrac{n}{\lambda} \leq |F_\lambda|\leq \tfrac{1}{\lambda}(|F| + 1)\quad \text{for $F\in \Acal$.}
\end{align}

Moreover, $W_{\rm hom}$ is locally Lipschitz continuous in the sense that, just as $\Whom$, it satisfies hypothesis $(H3)$. Precisely, one can find $L_{\rm hom}>0$ such that 
\begin{align}\label{locLip_Whom}
|\Whom(F) -\Whom(G)| \leq L_{\rm hom}(1+|F|^{p-1}+|G|^{p-1})|F-G| \quad \text{for all $F, G\in \Acal$. }
\end{align} 
To see this, we exploit that the property $(H3)$ carries over from $\Wsoft$ to $\Wsoft^{\rm qc}$ (cf.~e.g.~\cite[Lemma 2.1\,c)]{Mue87}). Hence, 
\begin{align*} 
|\Whom(F) -\Whom(G)|\leq \lambda |\Wsoft^{\rm qc}(F_\lambda) -\Wsoft^{\rm qc}(G_\lambda)| \leq \lambda \tilde L(1+|F_\lambda|^{p-1}+|G_\lambda|^{p-1})|F_\lambda-G_\lambda| 
\end{align*} 
for $F, G\in \Acal$ with a constant $\tilde L>0$. In view of~\eqref{est37}, it only remains to estimate $|F_\lambda-G_\lambda|$ suitably from above by $|F-G|$. We observe that
\begin{align*}
|F_\lambda -G_\lambda| \leq \tfrac{1}{\lambda} |F-G| + \tfrac{1-\lambda}{\lambda} (|\widehat{F}-\widehat{G}| + |R_Fe_n - R_Ge_n|) \leq \tfrac{2-\lambda}{\lambda}|F-G| + \tfrac{1-\lambda}{\lambda} |R_Fe_n-R_Ge_n|,
\end{align*}
where $\widehat{A}$ stands for the $n\times (n-1)$-matrix that results from removing the last column of $A\in \R^{n\times n}$.
We denote the $n$-dimensional cross product of vectors $v_1, \dots, v_n\in \R^n$ by $v_1\times \ldots \times v_{n-1} = \times_{i=1}^{n-1} v_i\in \R^n$. The latter is by definition the uniquely determined vector that is orthogonal on the hyperplane spanned by $v_1, \ldots, v_{n-1}$ such that the orientation of $v_1, \ldots, v_{n-1}, \times_{i=1}^{n-1}v_i$ is positive and its norm is the volume of the parallelotope associated with $v_1, \ldots, v_{n-1}$. 
For every rotation $R\in SO(n)$, one has that $Re_n=\times_{i=1}^{n-1} Re_i$. 

The multilinearity of the cross product in $\R^n$ and the fact that $|R_Fe_i| = |R_Ge_i|=1$ for $i=1, \ldots, n$ allows us to obtain iteratively that 
\begin{align*}
 & |R_Fe_n-R_Ge_n| = |\times_{i=1}^{n-1} R_F e_{i} - \times_{i=1}^{n-1} R_Ge_i | \\ & \qquad \leq |R_Fe_1-R_Ge_1| + |R_Ge_1\times R_G e_2\times \ldots \times R_Ge_{n-1} - R_Ge_1\times R_Fe_2\times \ldots \times R_Fe_{n-1}| \\  &  \qquad \leq \ldots \leq \sum_{i=1}^{n-1}|R_F e_i-R_G e_i|\leq (n-1) |\widehat{R_F}-\widehat{R_G}|\leq (n-1) |F-G|.
\end{align*}

Finally, we combine the above estimates to deduce the desired local Lipschitz property~\eqref{locLip_Whom}. \\[-0.2cm]
 
d) As mentioned in the introduction, proving a $\Gamma$-limit homogenization result as above without the hypothesis $(H1)$ is an open problem. In any case, Theorem~\ref{theo:hom} provides an upper bound on the $\Gamma$-limit (if existent) in that situation.
\end{remark}

We subdivide the proof of Theorem~\ref{theo:hom} into three main parts. After showing compactness, we first determine the homogenization $\Gamma$-limit for all affine functions, and then prove the general statement via a localization argument. 
Note that the specific structure of the admissible limit deformations as characterized in Theorem~\ref{theo:rigidity_general}, in particular the resulting multiplicative separation of $x'$ and $x_n$-variables in~\eqref{eq:form_limitgradient}, is key. This observation allows us to construct an approximation that respects the (asymptotic) constraints on the stiff layers, cf.~Proposition~\ref{prop:restricted_approximation}.

The first part of the proof is standard, yet, we sketch it here for the readers' convenience. 
\begin{proof}[Proof of Theorem~\ref{theo:hom} (Part~I): Compactness.] Let $(u_\eps)_\eps\subset L^p_0(\Omega;\R^n)$ be such that $E_\eps(u_\eps)<C$ for all $\eps>0$. Then, since $\dist(F, SO(n))\geq |F| - \sqrt{n}$ for all $F\in \R^{n\times n}$, the lower bounds on $\Wsoft$ and $\Wstiff$ in $(H2)$ and $(H4)$, imply that $(\nabla u_\eps)_\eps$ is uniformly bounded in $L^p(\Omega;\R^n)$. The stated relative compactness of $(u_\eps)_\eps$ in $L^p(\Omega;\R^n)$ follows now from Poincar\'{e}'s inequality, which shows that 
\begin{align*}
\norm{u_\eps}_{W^{1,p}(\Omega;\R^n)}\leq C\qquad \text{for all $\eps>0$,}
\end{align*} 
along with the compact embedding $W^{1,p}(\Omega;\R^n) \hookrightarrow\hookrightarrow L^p(\Omega;\R^n)$.
\end{proof}
 
\begin{proof}[Proof of Theorem~\ref{theo:hom} (Part~II): Affine case.]
Suppose that $u\in W^{1,p}(\Omega;\R^n)\cap L_0^p(\Omega;\R^n)$ with $E_{\rm hom}(u)<\infty$ is affine. Hence, there is $F\in\R^{n\times n}$ with $F\in \Acal$, cf.~\eqref{def:Acal}. 

\textit{Step~1: Upper bound.} 
The construction of a recovery sequence for $u$ as above, that is, finding $(u_\eps)_\eps\subset L^p_0(\Omega;\R^n)$ with
\begin{align}\label{recov_affine}
u_\eps\to u \text{\ in $L^p(\Omega;\R^m)$} \quad \text{and} \quad  E_\eps(u_\eps) \to E_{\rm hom}(u) \qquad \text{as $\eps\to 0$,}
\end{align} 

requires a careful adaptation of by now classical techniques, see e.g.~\cite{Mul87}. Indeed, instead of glueing small-scale oscillations on top of an affine function, the former are glued onto an appropriate laminate, namely the one constructed in~\eqref{eq:affine}. 

Let $\delta>0$. In view of Remark~\ref{rem:homogenization}\,b), one can find $\varphi_\delta\in W^{1,p}_0(\Ysoft;\R^n)$ such that 
\begin{align}\label{est_Whom}
\Wsoft^{\rm qc}(F_\lambda) \leq  \dashint_{\Ysoft} \Wsoft(F_\lambda + \nabla \varphi_\delta) \dd y  \leq \Wsoft^{\rm qc}(F_\lambda)+ \delta.
\end{align}
We set $\varphi_\delta$ equal to zero in the remainder of the unit cube and extend it $Y$-periodically to $\R^n$. 
For $\eps>0$ let $v_\eps^F$ be a Lipschitz function with gradients as in~\eqref{eq:affine} and vanishing mean value on $\Omega$. Then, $v_\eps^F\to u$ in $L^p(\Omega;\R^n)$ as $\eps\to 0$. 
With
\begin{align*}
u_{\delta,\eps}(x) = v_\eps^F(x) + \eps \varphi_\delta(\tfrac{x}{\eps}), \quad x \in \Omega,
\end{align*}
it follows that $u_{\delta, \eps}\to u$ in $L^p(\Omega;\R^n)$ as $\eps\to 0$. 
Regarding energies, we obtain that
\begin{align*}
E_\eps(u_{\delta,\eps}) = \int_{\eps\Ysoft\cap \Omega} \Wsoft \big(F_\lambda + \nabla \varphi_\delta (\tfrac{x}{\eps}) \big)\dd x = \int_{\Omega} \Wsoft (F_\lambda + \nabla \varphi_\delta \bigl(\tfrac{x}{\epsilon})\bigr)\mathbbm{1}_{\Ysoft}(\tfrac{x}{\eps})  \dd{x}. 
\end{align*}
Hence, as $\eps$ tends to zero,
\begin{align*}
\lim_{\eps\to 0} E_\eps(u_{\delta,\eps})  =|\Omega| \int_{Y}\Wsoft \bigl(F_\lambda + \nabla \varphi_\delta\bigr) \mathbbm{1}_{\Ysoft} \dd{x} = \lambda |\Omega| \dashint_{\Ysoft} \Wsoft (F_\lambda + \nabla \varphi_\delta)\dd{x},
\end{align*}
and we infer along with~\eqref{est_Whom} that
\begin{align*}
E_{\rm hom}(u_\delta) \leq \lim_{\eps \rightarrow 0} E_\eps(u_{\delta,\eps})\leq E_{\rm hom}(u_\delta) + \lambda|\Omega| \delta. 
\end{align*}
By Attouch's diagonalization lemma (see e.g.~\cite[Lemma~1.15, Corollary~1.16]{Att84}) there exist $\delta(\eps)$ such that $u_{\delta(\eps),\eps} \to u$ in $L^p(\Omega;\R^n)$ and $E_\eps(u_{\delta(\eps), \eps}) \to E_{\rm hom}(u)$ as $\eps\to 0$. Finally, defining $u_\eps:=u_{\delta(\eps), \eps}$ yields the desired recovery sequence for $u$. . 

\textit{Step~2: Lower bound.}
Let $(u_\epsilon)_\epsilon \subset L^p_0(\Omega;\R^n)$ be such that $u_\eps\rightarrow u$ in $L^p(\Omega;\R^n)$. We will show that 
\begin{align}\label{lowerbound}
\liminf_{\eps\to 0} E_\eps(u_\eps)\geq E_{\rm hom}(u). 
\end{align} 

Without loss of generality, let $\liminf_{\eps\to 0}E_{\eps}(u_\eps) = \lim_{\eps\to 0}E_\eps(u_\eps)<\infty$. In view of Part~I, 
one may further assume that 
\begin{align}\label{convergenceW1p}
u_\epsilon \rightharpoonup u\qquad \text{ in $W^{1,p}(\Omega;\R^n)$.}
\end{align}
We remark that
~\eqref{lowerbound} follows immediately, if one can prove that
\begin{align}\label{lowerbound2}
\liminf_{\eps\to 0} \int_{\eps\Ysoft\cap Q} \Wsoft (\nabla u_\eps)\dd{x} \geq \int_Q \Whom(\nabla u) \dd{x} = |Q| \Whom(F) 
\end{align} 
for any open cuboid $Q = O\times J \subset\subset \Omega$, where $O\subset \R^{n-1}$ and $J\subset \R$ and open interval. 
To deduce~\eqref{lowerbound}, we can then exhaust $\Omega$ with disjoint cuboids $Q_i\subset \Omega$ for $i\in \N$ such that $|\Omega\setminus \bigcup_{i=1}^\infty Q_i| = 0$ and apply~\eqref{lowerbound} on each $Q_i$. More precisely, for any $N\in\N$, 
\begin{align*}
\liminf_{\eps\to 0} \int_\Omega W_\eps^\alpha(\nabla u_\eps)\dd x \geq \sum_{i=1}^N \int_{\eps\Ysoft \cap Q_i} \Wsoft(\nabla u_\eps)\dd x\geq \big| \bigcup_{i=1}^N Q_i\big| \Whom(F), 
\end{align*}
so that taking the supremum over $N\in \N$ implies~\eqref{lowerbound}.

It remains to prove~\eqref{lowerbound2}, which relies substantially on hypothesis $(H1)$, or equivalently~\eqref{H1_equi}. 
Since $\Wsoft^{\rm pc}$ is polyconvex, we can find a convex function $g:\R^{\tau(n)}\to \R$ such that $W^\mathrm{pc}(F) = g(\mathcal{M}(F))$ for all $F \in \R^{n\times n}$. Moreover, let $P_\eps^i = (\R^{n-1}\times\eps [i, i+1))\cap Q$ for $i\in \N$ and $I_\eps\subset \N$ an index set such that $i\in I_\eps$ if and only if $|P_\eps^i| = \eps|O|$. As a consequence, $|\eps\Ysoft\cap P_\eps^i|= \lambda\eps|O|$ for all $i\in I_\eps$, and one finds with $Q_\eps=\bigcup_{i\in I_\eps} P_\eps^i\subset Q$ that
\begin{align}\label{conv_Qeps}
|Q_\eps| =\eps \# I_\eps \rightarrow |Q|\quad \text{and}\quad |\eps\Ysoft\cap Q_\eps| \to \lambda|Q|\qquad \text{ as $\eps\to 0$.}
\end{align}

Due to the convexity of $g$ we can invoke Jensen's inequality, applied twice, first in the version for Lebesgue-measurable functions and second in the discrete version, to obtain
\begin{align}\label{eq656}
\int_{\epsilon\Ysoft\cap Q} \Wsoft(\nabla u_\epsilon)\dd x 
&\geq \int_{\epsilon\Ysoft\cap Q} \Wsoft^\mathrm{pc}(\nabla u_\epsilon)\dd x 
\geq \sum_{i\in I_\epsilon} \int_{\eps\Ysoft\cap P^i_\eps} g\big(\mathcal{M}(\nabla u_\epsilon)\big) \dd x \nonumber \\ 
& \geq \lambda\eps \sum_{i\in I_\epsilon}\, g\Big(\dashint_{\eps\Ysoft\cap P_\eps^i} \mathcal{M}(\nabla u_\epsilon) \dd x\Big) \\
& \geq \lambda \eps \# I_\eps\,  g\Big(\frac{1}{\#I_\eps} \sum_{i\in I_\eps}\dashint_{\eps\Ysoft\cap P_\eps^i} \mathcal{M}(\nabla u_\epsilon) \dd x \Big) \nonumber \\
&=\lambda |Q_\eps|\, g\Big(\dashint_{\eps\Ysoft\cap Q_\eps} \Mcal(\nabla u_\eps)\dd x\Big). \nonumber
\end{align}

With the aim of eventually passing to the limit $\eps\to 0$ in~\eqref{eq656}, we will show first that
\begin{align}\label{weakconvergence1}
\mathcal{M}(\nabla u_\epsilon)\charf_{\epsilon\Ysoft\cap Q} 
\rightharpoonup \mathcal{M}(F)-(1-\lambda)\mathcal{M}(R_F) \quad \text{in $L^1(Q;\R^{\tau(n)})$.} 
\end{align}
For this the properties of $u_\eps$ due to the presence of the stiff layers need to be taken into account. Owing to $(H4)$ and~\eqref{convergenceW1p}, the sequence $(u_\eps)_\eps$ satisfies the requirements of Theorem~\ref{theo:rigidity_general}, and also Corollary~\ref{cor:rigidity}. Following the proofs we find the one-dimensional auxiliary sequence $(\Sigma_\eps)_\eps\subset L^p(J;SO(n))$ defined in~\eqref{Sigmaeps}. Recall from~\eqref{convergence_Sigmaeps} and~\eqref{def_R} that $\Sigma_\eps\to\Sigma_0$ in $L^p(J;\R^{n\times n})$ with $\Sigma_0(x_n)=R_F$ for $x\in Q$.
For each $\eps$, we extend $\Sigma_\eps$ constantly in $x'$ and call the resulting function $S_\eps\in L^\infty(Q;SO(n))$.  As a consequence of~\eqref{FJMestimate} (cf.~also~\eqref{Sigmaeps}) it holds that
\begin{align*}
\norm{\nabla u_\eps -S_\eps}_{L^p(\eps\Ystiff\cap Q;\R^{n\times n})} \leq C\eps^{\frac{\alpha}{p}-1}.
\end{align*}

Summing up, we have hence found a sequence $(S_\eps)_\eps \subset L^\infty(Q;SO(n))$ such that 
\begin{align} \label{eq:closednessXi}
S_\epsilon \rightarrow R_F\quad\text{ in $L^p(Q;\R^{n\times n})$} \quad \text{and}\quad
\| \nabla u_\epsilon - S_\epsilon\|_{L^p(\eps\Ystiff\cap Q;\R^{n\times n})}\rightarrow 0
\end{align} 
as $\eps\to 0$.

To see~\eqref{weakconvergence1}, let us rewrite the expression $\Mcal(\nabla u_\eps)\mathbbm{1}_{\eps\Ysoft\cap Q}$ as follows, 
\begin{align}\label{conv44}
\mathcal{M}(\nabla u_\epsilon)\charf_{\epsilon\Ysoft\cap Q} & = \mathcal{M}(\nabla u_\eps) - \mathcal{M}(\nabla u_\epsilon) \charf_{\epsilon\Ystiff\cap Q} \nonumber\\ 
&= \Mcal(\nabla u_\eps) - \big(\mathcal{M}(\nabla u_\epsilon) - \mathcal{M}(S_\epsilon)\big)\charf_{\epsilon\Ystiff\cap Q} - \mathcal{M}(S_\epsilon)\charf_{\epsilon\Ystiff\cap Q}. 
\end{align}
It is well-known that for $p>n$ weak continuity of minors holds, that is, $\mathcal{M}(\nabla u_\epsilon) \rightharpoonup \mathcal{M}(\nabla u)=\Mcal(F)$ in $L^1(\Omega;\R^{\tau(n)})$, see e.g.~\cite[Theorem 8.20, Part 4]{Dac08}.
By~\eqref{eq:closednessXi} and the Leibniz formula for determinants in combination with H\"older's inequality, 
$\Mcal(\nabla u_\eps)-\Mcal(S_\eps)\to 0$ and $\Mcal(S_\eps)\to \Mcal(R_F)$ both in $L^1(Q;\R^{\tau(n)})$. 
From the lemma on weak convergence of highly oscillating periodic functions \cite[Section~2.3]{CiD99} we infer that $\mathbbm{1}_{\eps\Ystiff\cap Q}\weaklystar (1-\lambda)$ in $L^\infty(Q)$. 
Finally, applying these results to the individual terms in~\eqref{conv44} along with a weak-strong convergence argument  implies~\eqref{weakconvergence1}.

Next, we observe that, as a Null-Lagrangian or polyaffine function, $G\mapsto \Mcal(G)$ for $G\in \R^{n\times n}$ is also rank-one affine, cf.~\cite[Theorem 5.20]{Dac08}. Since $F = \lambda F_\lambda +(1-\lambda)R_F$ and $F_\lambda-R_F=\frac{1}{\lambda}(F-R_F) = \frac{1}{\lambda}d_F\otimes e_n$, it follows that 
\begin{align*}
\Mcal(F) = \lambda \Mcal(F_\lambda) + (1-\lambda)\Mcal(R_F). 
\end{align*}

Then, together with~\eqref{weakconvergence1}, we obtain 
\begin{align*}
\mathcal{M}(\nabla u_\epsilon)\charf_{\epsilon\Ysoft\cap Q} 
\rightharpoonup \lambda \mathcal{M}(F_\lambda) \quad \text{in $L^1(Q;\R^{\tau(n)})$,} 
\end{align*}
which in view of~\eqref{conv_Qeps} and the uniform boundedness of $(\nabla u_\eps)_\eps$ in $L^p(Q;\R^n)$ results in
\begin{align}\label{conv7}
\lim_{\eps\to 0}\dashint_{\eps\Ysoft\cap Q_\eps} \Mcal(\nabla u_\eps)\dd{x} = \Mcal(F_\lambda).
\end{align}

Finally, we combine~\eqref{eq656} with~\eqref{conv_Qeps} and~\eqref{conv7} and exploit the continuity of $g$ as a convex function to arrive at~\eqref{lowerbound2}. This concludes the proof of the lower bound. 
\end{proof}

\begin{remark}\label{rem:boundaryvalues}
a) Step~1 can be performed as above for any open and bounded set $\Omega$, meaning that the restriction to a flat, cross-section Lipschitz domain is not necessary for the construction of a sequence satisfying~\eqref{recov_affine}.\\[-0.2cm] 

b) Note that the recovery sequence constructed in Step~1 can be assumed to have the same boundary values as  $v^F_\eps$, i.e.~$u_{\epsilon} - v^F_\epsilon \in W^{1,p}_0(\Omega;\R^n)$. Indeed, the small-scale oscillations glued onto the laminate $v^F_\epsilon$ for sufficiently small $\eps$ can be adapted outside of $\{x\in \Omega: \dist(x, \partial\Omega)>2\eps\}$ to vanish on $\{x\in \Omega: \dist(x, \partial \Omega)<\eps\}$. 
This modification affects neither the convergence of $(u_\eps)_\eps$ nor of $(E_\eps(u_\eps))_\eps$.
\end{remark}

Based on the findings of Part II for the affine case, we will now prove the homogenization $\Gamma$-convergence result for general limit functions. 

\begin{proof}[Proof of Theorem~\ref{theo:hom} (Part III): General case.]
Let $u\in W^{1,p}(\Omega;\R^n)\cap L_0^p(\Omega;\R^n)$ be  such that $u(x) = R(x)x + b(x)$ for $x\in \Omega$, where $R \in W^{1,p}(\Omega;SO(n))$ and $b \in W^{1,p}(\Omega;\R^{n})$ satisfy $\nabla' R =0$ and $\nabla' b =0$.
 As in the previous parts, we have arranged the arguments in several steps, numbered consecutively. 

\textit{Step~3: Upper bound.} We aim to find a sequence $(u_\eps)_\eps\subset W^{1,p}(\Omega;\R^n)$ such that $u_\eps\weakly u$ in $W^{1,p}(\Omega;\R^n)$ and $\limsup_{\eps\to 0} E_\eps(u_\eps) \leq E_{\rm hom}(u)$. 
The idea behind the construction of a recovery sequence for $u$ is to use the approximating  sequence from Proposition~\ref{prop:restricted_approximation} and to perturb it in the softer layers by suitably relaxing microstructures that guarantee the optimal energy. To obtain these perturbations, the results from Step~1 (Part II) are applied to piecewise affine approximations of $u$. 
\textit{Step 3a: Piecewise constant approximation of $\nabla u$.}
Recall that the gradient of $u$ is 
\begin{align}\label{representation4}
\nabla u = R+ (\partial_nR)x \otimes e_n + d \otimes e_n. 
\end{align}

 First we approximate the functions in~\eqref{representation4}, that is $d$, $\partial_n R$, $R$, and the identity map~${\rm id}_{\R^n}:x\mapsto x$, by simple functions. Indeed, by following standard constructions (e.g.~\cite[Theorem~1.2]{AmE08}), it is not hard to see that uniform approximation of the continuous function $R$ is possible while preserving the values in $SO(n)$. Without loss of generality, we may assume that all four approximations above have a common partition of $\Omega$. Due to the globally one-dimensional character of $d, \partial_n R$ and $R$, the elements of the partition that do not intersect with $\partial \Omega$ can be assumed to be cubes aligned with the coordinate axes. 
To be precise, for every $\delta>0$ there are finitely many cubes $Q_\delta^i\subset\R^n$, which we index by $I_\delta$,  with maximal side length $\delta$ such that $|\Omega\setminus\bigcup_{i\in I_\delta}Q_\delta^i|= 0$ and $Q_\delta^i\cap \Omega\neq\emptyset$ for $i\in I_\delta$, 
and $d^i_\delta, \xi^i_\delta \in \R^n$, $S_\delta^i\in \R^{n\times n}$, and $R^i_\delta \in SO(n)$ such that the simple functions defined by
\begin{align*}
R_\delta=\sum_{i \in I_\delta} R^i_\delta \mathbbm{1}_{Q^i_\delta\cap\Omega},\quad  
d_\delta=\sum_{i \in I_\delta} d_\delta^i \mathbbm{1}_{Q^i_\delta\cap\Omega},\quad  
S_\delta=\sum_{i \in I_\delta} S_\delta^i \mathbbm{1}_{Q^i_\delta\cap\Omega} \quad \text{and}\quad 
\xi_\delta=\sum_{i \in I_\delta} \xi^i_\delta \mathbbm{1}_{Q^i_\delta\cap\Omega},
\end{align*} 
satisfy
\begin{equation}
\begin{aligned}\label{pc_approximation1}
\norm{R_\delta-R}_{L^\infty(\Omega; \R^{n\times n})} + \norm{{d}_\delta-d}_{L^p(\Omega;\R^{n})} + \norm{S_\delta-\partial_n R}_{L^p(\Omega;\R^n)} + \norm{\xi_\delta-\id_{\R^n}}_{L^\infty(\Omega;\R^n)}< \delta. 
\end{aligned}
\end{equation}
Consider  the piecewise constant function $U_\delta\in L^\infty(\Omega;\R^{n\times n})$ defined by
\begin{align}\label{Udelta}
U_\delta = R_\delta +S_\delta  \xi_{\delta} \otimes e_n  + d_\delta \otimes e_n = \sum_{j \in I_\delta} U_\delta^{i}\mathbbm{1}_{Q_\delta^i\cap \Omega}, 
\end{align}
where $U_\delta^i = R_\delta^i + S_\delta^i \xi_\delta^i \otimes e_n+ d_\delta^i \otimes e_n\in \Acal$ for $i \in I_\delta$. 
Then,
\begin{align}\label{Fmu}
\norm{U_\delta - \nabla u}_{L^p(\Omega;\R^{n\times n})} \leq C\delta,
\end{align}
with a constant $C>0$ independent of $\delta$. Indeed, in view of~\eqref{pc_approximation1} and~\eqref{representation4} this is an immediate consequence of the estimate
\begin{align*}
\norm{U_\delta-\nabla u}_{L^p(\Omega;\R^{n\times n})} & \leq \norm{R_\delta-R}_{L^\infty(\Omega;\R^{n\times n})}  + \diam(\Omega)
\norm{S_\delta - \partial_n R}_{L^p(\Omega;\R^{n\times n})} \\ &\qquad + \norm{\partial_n R}_{L^p(\Omega;\R^{n\times n})}\norm{\xi_\delta-\id_{\R^n}}_{L^p(\Omega;\R^n)} + \norm{{d}_\delta-d}_{L^p(\Omega;\R^n)}. 
\end{align*}

\textit{Step~3b: Locally optimal microstructure.}
By Step~1 (Part II), where recovery sequences in the affine case were established, we can find under consideration of Remark~\ref{rem:boundaryvalues}\,a) on each $Q^i_\delta\cap\Omega$ with $\delta>0$ and $i\in I_\delta$ a sequence $(u^i_{\delta, \epsilon})_\epsilon \subset  W^{1,p}(Q^i_\delta\cap \Omega;\R^n)$ such that $\nabla u^i_{\delta,\epsilon} \rightharpoonup U_\delta$ in $L^p(Q^i_\delta\cap\Omega;\R^{n\times n})$ as $\eps\to 0$ and 
\begin{align}\label{eq42}
\lim_{\epsilon\rightarrow 0} \int_{Q^i_\delta\cap\Omega} W^\alpha_\epsilon (x, \nabla u^i_{\delta,\epsilon})\dd x =  \lim_{\eps\to 0} \int_{\eps\Ysoft \cap Q^i_\delta \cap\Omega} \Wsoft(\nabla u^i_{\delta, \eps}) \dd x =  \int_{Q^i_\delta\cap\Omega} \Whom(U_\delta) \dd x. 
\end{align}
Now, with $w_{\delta, \eps}^i:=v^{U^i_{\delta}}_\epsilon \in W^{1,\infty}(Q^i_\delta\cap\Omega;\R^n)$ a laminate as introduced in~\eqref{eq:affine}, let
\begin{align*}
\varphi_{\delta,\epsilon}^i = u^i_{\delta,\epsilon} - w_{\delta, \eps}^i\quad \text{on $Q_\delta^i\cap\Omega$.} 
\end{align*}
According to Remark~\ref{rem:boundaryvalues}\,b), we may assume that the boundary values of $u^i_{\delta,\epsilon}$  and $w_{\delta, \eps}^i$ coincide, which entails that $\varphi_{\delta,\epsilon}^i \in W^{1,p}_0(Q^i_\delta\cap \Omega;\R^n)$.
Let us join these local components together in one function $\varphi_{\delta, \eps}\in W^{1,p}_0(\Omega;\R^n)$ given by
\begin{align}\label{varphi_deltaeps}
\varphi_{\delta,\epsilon} = \sum_{i\in I_\delta} \varphi^i_{\delta,\epsilon}\mathbbm{1}_{Q^i_\delta\cap\Omega}.
\end{align}
Note that by construction $\varphi_{\delta, \eps}=0$ in $\eps\Ystiff\cap\Omega$. Moreover, 
\begin{align}\label{conv79}
\nabla\varphi_{\delta, \eps}\weakly 0\quad \text{ in $L^p(\Omega;\R^{n\times n})$ as $\eps \to 0$, }
\end{align}
and $\norm{\nabla \varphi_{\delta, \eps}}_{L^p(\Omega;\R^{n\times n})}$ is uniformly with respect to $\eps$ and $\delta$.
In analogy to~\eqref{varphi_deltaeps} we define for later reference the map of local laminates
\begin{align}\label{wepsdelta}
w_{\eps, \delta} =  \sum_{i\in I_\delta} w^i_{\delta,\epsilon}\mathbbm{1}_{Q^i_\delta\cap\Omega} \in L^\infty(\Omega;\R^n). 
\end{align}
Since the homogenized energy density $\Whom$ satisfies the local Lipschitz condition~\eqref{locLip_Whom} according to Remark~\ref{rem:homogenization}\,c), we infer along with~\eqref{Fmu} and H\"older's inequality that
\begin{align*}
\int_{\Omega} \Whom(U_\delta) \dd x& \leq \int_{\Omega} \Whom(\nabla u) \dd x + C\norm{U_\delta - \nabla u}_{L^p(\Omega;\R^{n\times n})} 
\leq \int_{\Omega} \Whom (\nabla u) \dd x   + C\delta.
\end{align*}
Summing over all $i \in I_\delta$ in~\eqref{eq42} and taking the limit $\eps\to 0$ gives that
\begin{align}\label{prelim}
\limsup_{\eps \to 0} \int_{\Omega} W^\alpha_\eps(x, U_{\delta,\epsilon})\dd{x} =  \limsup_{\eps \to 0} \int_{\eps\Ysoft\cap\Omega} \Wsoft(U_{\delta,\epsilon})\dd{x}  \leq  \int_{\Omega} W_{\rm hom} (\nabla u)\dd x + C\delta,
\end{align}
where $U_{\eps, \delta} = \sum_{i\in I_\delta} \nabla u^i_{\delta, \eps} \mathbbm{1}_{Q^i_{\delta}\cap\Omega}$.

\textit{Step~3c: Optimal construction with admissible gradient structure.} 
After diagonalization, the functions $U_{\eps, \delta(\eps)}$ would define a recovery sequence as desired, provided they have gradient structure, i.e., there is a potential $u_\eps\in W^{1, p}(\Omega;\R^n)$ with $\nabla u_\eps = U_{\eps, \delta(\eps)}$. Due to incompatibilities at the interfaces between neighboring cubes, however, this can in general not be expected.
 To overcome this issue and to obtain an admissible recovery sequence, we discard the local laminates $w_{\eps, \delta}$ from~\eqref{wepsdelta}, and instead add the locally optimal microstructures $\varphi_{\delta,\epsilon}$ onto the functions $v_\epsilon$, which result from Proposition~\ref{prop:restricted_approximation} applied to $u$. 

More precisely, applying Proposition~\ref{prop:restricted_approximation} to the given $u$ provides us with an approximating sequence in $W^{1,p}(\Omega;\R^n)$ with useful properties, which we call $(v_\eps)_\eps$. In particular, 
\begin{align}\label{conv_veps}
\nabla v_\eps\weakly \nabla u\qquad \text{in $L^{p}(\Omega;\R^n)$,}
\end{align}
$\nabla v_\eps\in SO(n)$ a.e.~in $\eps\Ystiff\cap \Omega$, 
\begin{align}\label{conv93}
\norm{\nabla v_\eps - (\nabla u)_\lambda}_{L^p(\eps\Ysoft\cap \Omega;\R^{n\times n})} \to 0 
\end{align}
with $(\nabla u)_\lambda$ as in~\eqref{Ulambda}. 

Let $u_{\delta, \epsilon}\in W^{1,p}(\Omega;\R^n)\cap L_0^{1,p}(\Omega;\R^n)$ be given by
\begin{align*}
u_{\delta,\eps}= v_\eps + \varphi_{\delta,\epsilon} - \dashint_{\Omega} v_\eps + \varphi_{\delta, \eps}\dd{x}. 
\end{align*}
Next, we estimate the energetic error brought about by replacing $w_{\eps, \delta}$ in Step~3b with $v_\eps$.
By $(H3)$, H\"older's inequality and the above definitions, 
\begin{align}\label{estimate-energy}
& \int_{\eps\Ysoft\cap \Omega} |\Wsoft(U_{\delta, \eps}) - \Wsoft(\nabla u_{\delta, \eps})| \dd{x}\nonumber \\
& \ \leq  L \normb{1+|U_{\delta, \eps}|^{p-1}+|\nabla u_{\delta, \eps}|^{p-1}}_{L^{\frac{p}{p-1}}(\eps\Ysoft\cap \Omega)} \norm{U_{\delta, \eps}-\nabla u_{\delta, \eps}}_{L^p(\eps\Ysoft\cap \Omega;\R^{n\times n})}\\
&\  \leq C\big(1+\norm{\nabla v_\eps}_{L^p(\Omega;\R^n)} + \norm{(U_\delta)_\lambda}_{L^p(\Omega;\R^{n\times n})} + \norm{\nabla \varphi_{\eps, \delta}}_{L^p(\Omega;\R^n)}\big) \norm{(U_\delta)_\lambda-\nabla v_\eps}_{L^p(\eps\Ysoft\cap \Omega;\R^{n\times n})} \nonumber
\end{align}
with $C>0$ independent of $\eps$ and $\delta$. The first factor in the last line of~\eqref{estimate-energy} is uniformly bounded (with respect to $\delta$ and $\eps$) as a consequence of~\eqref{conv_veps},~\eqref{Fmu} and the remark below~\eqref{conv79}.
The second factor can be controlled with the help of~\eqref{conv93} and the following estimate, which exploits~\eqref{pc_approximation1} and~\eqref{Fmu}, 
\begin{align}\label{est09}
\| (\nabla u)_\lambda - (U_\delta)_\lambda\|_{L^p(\eps\Ysoft\cap \Omega;\R^{n\times n})}
 \leq (1-\tfrac{1}{\lambda}) \|R -R_\delta\|_{L^\infty(\Omega;\R^{n\times n})} + \tfrac{1}{\lambda} \norm{\nabla u-U_\delta}_{L^p(\Omega;\R^{n\times n})} \leq C\delta.
\end{align}

Thus, 
\begin{align}\label{est51}
\int_{\eps\Ysoft\cap \Omega} |\Wsoft(U_{\delta, \eps}) - \Wsoft(\nabla u_{\delta, \eps})| \dd{x} \leq 
C \big(\norm{(\nabla u)_\lambda-\nabla v_\eps}_{L^p(\eps\Ysoft\cap \Omega;\R^{n\times n})} +\delta\big).
\end{align}

\textit{Step 3d: Diagonalization.}
As both $U_{\eps, \delta}$ and $\nabla u_{\delta, \eps}$ lie in $SO(n)$ almost everywhere on the stiff layers,~\eqref{prelim} in combination with~\eqref{est51},~\eqref{conv93} and $(H2)$  
 yields that
\begin{align*}
\limsup_{\eps\to 0} E_\eps(u_{\eps, \delta}) & \leq \int_{\Omega}W_{\rm hom}(\nabla u)\dd x + C\delta. 
\end{align*}
Besides, we derive from~\eqref{conv_veps} and~\eqref{conv79} that $\nabla u_{\eps, \delta}\weakly \nabla u$ in $L^{p}(\Omega;\R^n)$ as $\eps\to 0$ for every $\delta$. After exploiting Poincar\'e's inequality, the compact embedding of $W^{1,p}$ into $L^p$, and the Urysohn subsequence principle it follows then that $u_{\eps, \delta} \to u$ in $L^p(\Omega;\R^n)$ as $\eps\to 0$.

Finally, the diagonalization lemma by Attouch (see e.g.~\cite[Lemma~1.15, Corollary~1.16]{Att84}) guarantees the existence of a sequence $\delta(\eps)$ such that $u_\eps:=u_{\eps, \delta(\eps)}\in W^{1,p}(\Omega;\R^n)\cap L_0^p(\Omega;\R^n)$ satisfies 
\begin{align*}
\limsup_{\eps\to 0} E_\eps(u_{\eps}) \leq \int_{\Omega}W_{\rm hom}(\nabla u)\dd x. 
\end{align*}
and $u_\eps\to u$ in $L^p(\Omega;\R^n)$.
This shows that $(u_\eps)_\eps$ is a recovery sequence for $u$ as stated.


\textit{Step~4: Lower bound.}
Let $(u_\eps)_\eps\subset W^{1,p}(\Omega;\R^n)\cap L_0^p(\Omega;\R^n)$ be a sequence of uniformly bounded energy, i.e.~$E_\epsilon(u_\epsilon) < C$ for all $\eps>0$, such that $u_\eps \rightharpoonup u$ in $W^{1,p}(\Omega; \R^n)$ for some $u\in W^{1,p}(\Omega; \R^n)$. By Theorem~\ref{theo:rigidity_general}, $\nabla u$ has the form~\eqref{eq:form_limitgradient}.
We will show that
\begin{align}\label{liminf_aux}
\liminf_{\eps\rightarrow 0} \int_{\eps\Ysoft\cap \Omega} \Wsoft(\nabla u_\eps) \dd x \geq \int_\Omega \lambda \Wsoft^{\rm qc}((\nabla u)_\lambda) \dd x = \int_{\Omega} W_{\rm hom}(\nabla u)\dd{x},
\end{align}
which implies the desired liminf-inequality $\liminf_{\eps\to 0}E_\eps(u_\eps)\geq E_{\rm hom}(u)$. 

To tie this general case to the affine one in Step~2, we adjust to our specific situation a common approximation strategy (see e.g.~\cite[Theorem~1.3]{Mul87}) based on comparison sequences that involve elements of the constructed recovery sequences. Note that there is no need for the comparison sequence to have full gradient structure, which allows us to argue separately on each piece of the piecewise constant approximation of $\nabla u$.

\textit{Step 4a: Construction of a comparison sequence.} 
First, we approximate $\nabla u$ by piecewise constant functions $U_\delta$ as in Step~3a, 
see~\eqref{Udelta} and \eqref{Fmu}. 
For $\eps, \delta>0$ let $w_{\eps, \delta}$ and $v_\epsilon$ be as in Step~3c. Recall that for any $\delta>0$ and $i\in I_\delta$,
\begin{align}\label{conv64}
\nabla w_{\delta, \eps}^i \rightharpoonup U^i_\delta  \quad \text{in $L^p(Q^i_\delta;\R^{n\times n})$ as $\eps\to 0$,}
\end{align}
and that the sequence $(v_\eps)_\eps\subset W^{1,p}(\Omega;\R^n)$ satisfies~\eqref{conv_veps} and~\eqref{conv93}.  Moreover, \begin{align}\label{conv55}
\norm{\nabla v_\eps - R}_{L^p(\eps\Ystiff\cap\Omega;\R^{n\times n})} \to 0 \quad \text{as $\eps\to 0$,}
\end{align}
in view of Proposition~\ref{prop:restricted_approximation}. 

Now let us introduce 
\begin{align*}
z_{\delta, \epsilon} = u_\epsilon-v_\epsilon+w_{\delta,\epsilon} + \dashint_{\Omega} v_\eps - w_{\delta, \eps} \dd{x}.
\end{align*} These functions have vanishing mean value on $\Omega$ and satisfy 
 $z_{\delta, \eps}^i=z_{\delta, \eps}|_{Q_\delta^i} \in W^{1,p}(Q_\delta^i;\R^n)$ for any $i\in I_\delta$. Due to~\eqref{conv64},~\eqref{conv_veps} and the assumption on the weak convergence of $(u_\eps)_\eps$, it follows for every $\delta>0$ that
\begin{align*}
\nabla z_{\delta, \epsilon}^i = \nabla u_\epsilon- \nabla v_\epsilon + \nabla w_{\delta,\epsilon}^i 
\rightharpoonup U^i_\delta \quad \text{in~} L^p(Q^i_\delta;\R^n) \quad\text{as $\eps\to 0$.} 
\end{align*}

Hence, as a consequence of the result in the affine case (see Step~2, Part~II), applied to the restriction of $z_{\delta, \eps}$ to any cuboid $Q^i_\delta$ with $i\in \tilde I_\delta:=\{i\in I_\delta: Q_\delta^i\subset\subset Q\}$, we deduce that 
 \begin{align*}
\liminf_{\epsilon \rightarrow 0}
 \int_{\epsilon\Ysoft\cap Q_\delta^i} W^\alpha_\eps(x,\nabla z_{\delta,\epsilon}^i) \dd x \geq  \int_{Q_\delta^i} W_{\rm hom}(U_\delta^i) \dd x.
 \end{align*}  
In fact, if
\begin{align}\label{conv60}
\norm{\dist(\nabla z_{\delta, \eps}^i, SO(n))}_{L^p(\eps\Ystiff \cap Q_\delta^i)}\to 0
\end{align} 
as $\eps\to 0$, one can follow the reasoning of Step~2 in Part~II to see that even 
 \begin{align}\label{est32}
\liminf_{\epsilon \rightarrow 0} 
 \int_{\epsilon\Ysoft\cap Q_\delta^i} \Wsoft(\nabla z_{\delta,\epsilon}^i) \dd x \geq  \int_{Q_\delta^i} W_{\rm hom}(U_\delta^i) \dd x.
 \end{align}  
  
To verify~\eqref{conv60} for $i\in \tilde I_\delta$, we mimic the arguments leading to~\eqref{eq:closednessXi} on the cuboid $Q_\delta^i\subset\subset \Omega$. This implies in particular that
\begin{align}\label{conv61}
 \|\nabla u_\epsilon - R\|_{L^p(\epsilon\Ystiff\cap Q_\delta^i;\R^{n\times n}) }\to 0.
\end{align}
Then, 
\begin{align*}
\|\nabla z_{\delta, \epsilon}^i-R^i_\delta\|_{L^p(\epsilon\Ystiff\cap Q_\delta^i;\R^{n\times n})}  &
 =\|\nabla z_{\delta, \epsilon}^i-\nabla w_{\delta, \eps}^i\|_{L^p(\epsilon\Ystiff\cap Q_\delta^i;\R^{n\times n})} \\
 &
= \|\nabla u_\epsilon-\nabla v_\epsilon\|_{L^p(\epsilon\Ystiff\cap Q_\delta^i;\R^{n\times n})} \\
& \leq \|\nabla u_\epsilon- R\|_{L^p(\epsilon\Ystiff\cap Q_\delta^i;\R^{n \times n})} 
+\|R- \nabla v_\eps\|_{L^p(\eps\Ystiff \cap Q_\delta^i;\R^{n\times n})}, 
\end{align*}
which in light of~\eqref{conv61} and~\eqref{conv55} gives~\eqref{conv60}.

 \textit{Step 4b: Energy estimates.} 
For the homogenized energy, we derive from the local Lipschitz continuity of $W_{\rm hom}$ (cf.~Remark~\ref{rem:homogenization}\,c)), along with~\eqref{Fmu} and H\"older's inequality, that 
\begin{align*}
 \int_\Omega |W_\mathrm{hom}(U_\delta) - W_\mathrm{hom}(\nabla u) | \dd{x} \leq C \|U_\delta-\nabla u\|_{L^p(\Omega;\R^{n\times n})} < C\delta. 
 \end{align*}
Furthermore, with $(H2)$ and the uniform $L^p$-bounds on $\nabla u_\eps$ and $\nabla w_{\delta, \eps}^i$, we have for any $i\in I_\delta$,
 \begin{align*}
&  \int_{\epsilon\Ysoft\cap Q_\delta^i} |\Wsoft (\nabla z_{\delta,\epsilon}^i) - \Wsoft (\nabla u_\eps)|
 = \int_{\eps\Ysoft\cap Q_\delta^i} |\Wsoft(\nabla u_\epsilon- \nabla v_\epsilon+\nabla w_{\delta,\epsilon}) - \Wsoft(\nabla u_\eps)|\dd x \\
  &\qquad\qquad \leq  C \|\nabla v_\epsilon-\nabla w_{\delta,\epsilon}\|_{L^p(\eps\Ysoft \cap Q_\delta^i;\R^{n\times n})}  \\ &\qquad\qquad \leq C \bigl( \|\nabla v_\epsilon-(\nabla u)_\lambda\|_{L^p(\eps\Ysoft \cap Q_\delta^i;\R^{n\times n})} + \|(\nabla u)_\lambda - (U_\delta)_\lambda \|_{L^p(\eps\Ysoft \cap Q_\delta^i;\R^{n\times n})} \bigr).
  \end{align*}
 
Due to~\eqref{conv93}, the first expression on the right hand side converges to zero as $\eps\to0$, while the second can be estimated from above by $\delta$ by~\eqref{est09}. Considering~\eqref{est32}, we conclude after summing over $i\in \tilde I_\delta$ that
 \begin{align*}
 \liminf_{\eps\to 0} \int_{\eps\Ysoft\cap \Omega_\delta} \Wsoft(\nabla u_\eps)\dd{x}  \geq \int_{\Omega_\delta} W_{\rm hom}(\nabla u) \dd{x} - C\delta,
 \end{align*}
 where $\Omega_\delta= \bigcup_{i\in \tilde I_\delta} Q_\delta^i$. Since $|\Omega\setminus \Omega_\delta|\to 0$ by construction, passing to the limit $\delta \rightarrow 0$ establishes~\eqref{liminf_aux}, which concludes the proof.
  \end{proof}

As the next remark shows, the homogenized energy density $W_{\rm hom}$ from~\eqref{Whom} coincides with the single-cell formula arising from a related model without elasticity (''$\alpha=\infty$'') on the stiff layers. This observation indicates that microstructures developing over multiple cells, as they are to be expected in general homogenization problems with non-convex energy densities (cf.~\cite{Mul87} and more recently~\cite{BaG10}), do not occur. They are indeed inhibited by the presence of the stiff horizontal layers.

\begin{remark}\label{rem:cell}
With $\Wsoft$ satisfying $(H1)$-$(H3)$ and $W_{\rm rig}(F)=\chi_{SO(n)}(F)$ for $F\in \R^{n\times n}$, let $\overline{W}: \Omega\times \R^{n\times n}\rightarrow [0,\infty]$ be given by
\begin{align*}
\overline{W}(x, F) =\begin{cases}
W_{\rm rig}(F) & \text{for $x\in \Ystiff \cap\Omega$,} \\
\Wsoft(F) & \text{for $x\in \Ysoft\cap \Omega$, }
\end{cases}
\end{align*}
and denote by $\overline{W}_{\rm cell}$ the cell formula associated with $\overline{W}$, i.e.,
\begin{align*}
\overline{W}_{\rm cell}(F) = \inf_{\psi \in W^{1,p}_{\#}(Y;\R^n)} \dashint_Y \overline{W}(y, F+\nabla \psi) \dd y, \quad F\in \R^{n\times n}.
\end{align*} 
We will show that for $F\in \R^{n\times n}$,
\begin{align}\label{equality_Wcell}
\overline{W}_{\rm cell}(F) = \begin{cases}
\Whom (F) & \text{for $F \in \Acal$, }\\  
\infty & \text{otherwise. }
\end{cases}
\end{align}

Indeed, if $\overline{W}_{\rm cell}(F)<\infty$, there exists $\psi\in W^{1,p}_{\#}(Y;\R^n)$ such that the expression $\dashint_{Y} W_{\rm rig}(y, F+\nabla \psi)\dd y$ is finite. This implies $F+ \nabla \psi\in SO(n) $ a.e.~in $\Ystiff$, and we infer from Reshetnyak's theorem~\cite{Res67} (cf.~also Theorem~\ref{prop:FJM}) that for some $R \in SO(n)$,
\begin{align}\label{eq88}
F + \nabla \psi = R\quad \text{ on $\Ystiff$.}  
\end{align}
Therefore, since $\psi$ is periodic, one obtains for $i=1, \ldots, n-1$ that
\begin{align*}
Fe_i = Fe_i + \int_Y \partial_i \psi \dd y = \int_Y R e_i  \dd y = R e_i,
\end{align*}
and hence, $F\in \Acal$ and in particular, $F=R+d\otimes e_n$ with $d\in \R^n$. By~\eqref{eq88},  $\nabla \psi = - d\otimes e_n$ on $\Ystiff$. 

Considering the piecewise affine function $v \in W^{1, \infty}_{\#}(Y;\R^n)$ with zero mean value and gradient
\begin{align*}
\nabla v = (-\mathbbm{1}_{\Ystiff} + \tfrac{1-\lambda}{\lambda}\mathbbm{1}_{\Ysoft}) d\otimes e_n, 
\end{align*}
we can find $\varphi\in W^{1,p}_{\#}(Y;\R^n)$ such that $\nabla \varphi=0$ in $\Ystiff$ and $\psi$ is represented as $\psi= v+\varphi$. 

Thus, 
 \begin{align*}
& \inf_{\psi\in W^{1,p}_{\#}(Y;\R^n)}\dashint_{Y} \overline{W}(y, F+\nabla \psi)\dd y \\ & \qquad \qquad = \inf\Bigl\{ \int_{\Ysoft} \Wsoft(F+\nabla \psi) \dd{y}: \psi\in W^{1,p}_\#(Y;\R^n), \nabla \psi = -d\otimes e_n \text{ on $\Ystiff$}\Bigr\}\\
& \qquad \qquad =  \inf\Bigl\{ \int_{\Ysoft}  \Wsoft (F + \tfrac{1-\lambda}{\lambda}d\otimes e_n+ \nabla \varphi) \dd{y}: \varphi\in W^{1,p}_\#(Y;\R^n),\ \nabla \varphi= 0 \text{ on $\Ystiff$}\Bigr\}\\
& \qquad \qquad =  \inf\Bigl\{ \dashint_{\Ysoft} \lambda \Wsoft(F_\lambda + \nabla \varphi) \dd{y}: \varphi\in W^{1,p}_\#(Y;\R^n),\   \varphi= 0 \text{ on $\Ystiff$}\Bigr\}\\
& \qquad \qquad = \lambda \inf_{\phi\in W^{1,p}_0(\Ysoft;\R^n)} \dashint_{\Ysoft} \Wsoft (F_\lambda +\nabla \phi)\dd{y}.
\end{align*} 

where the last equality makes use of Remark~\ref{rem:homogenization}\,b). This verifies~\eqref{equality_Wcell}.  
\end{remark}


\begin{appendix}
\section{Collected auxiliary results}

In the following lemma, we provide a type of reverse Poincar\'{e} inequality for special affine maps given as the difference of two rotations on a domain that is thin in one dimension. The special feature of this result (e.g.~in comparison with classical Caccioppoli estimates for harmonic maps \cite{IwS03}) is that the constant can be chosen independently of the thickness of the domain in $e_n$-direction. 

\begin{lemma}\label{lem:estimateR1R2}
For an integer $n\geq 2$ let $P=O \times I$ with $O\subset \R^{n-1}$ an open cube of side length $l>0$ and $I\subset \R$ an interval of length $h>0$, and let $1\leq p<\infty$. 
Then there exists a constant $C>0$ depending only on $n$ and $p$ such that for all rotations $R_1, R_2\in SO(n)$ and translation vectors $d\in \R^n$,
\begin{align*}
 \int_P \big|(R_2-R_1)x + d\big|^p \dd x \geq C l^p\, |P|\, |R_2-R_1|^p.
 \end{align*}
\end{lemma}

\begin{proof}
We will prove the result for $p=1$, for general $p$ the statement then follows immediately from H\"older's inequality.

Moreover, without loss of generality let $R_2$ be the identity matrix $\Ibb=\Ibb_n\in \R^{n\times n}$. We set $R=R_1\in SO(n)$ and write $A:= \Ibb-R\in \R^{n\times n}$.
Let $\overline{P}$ denote the translation of the open cuboid $P$ centered in the origin. The following arguments make use of the nested sets $\widehat{P}\subset Z\subset \overline{P}$, where $Z$ is the cylinder with circular cross section inscribed in $\overline{P}$ and $\widehat{P}$ is the largest centered, open cuboid contained in $Z$. Precisely, 
\begin{align*}
Z=B^{n-1}_{l/2}\times \bigl(-\tfrac{h}{2}, \tfrac{h}{2}\bigr)\quad \text{and} \quad  \widehat{P} = \bigl(-\tfrac{l}{2\sqrt{n}}, \tfrac{l}{2\sqrt{n}}\bigr)^{n-1} \times \bigl(-\tfrac{h}{2}, \tfrac{h}{2}\bigr),
\end{align*}
where $B^{n-1}_r$ the $(n-1)$-dimensional ball around the origin with radius $r$.
  
With this notation in place, we observe that 
\begin{align}\label{estcentering}
\int_P |Ax + d| \dd x \geq \int_{\overline{P}} |Ax| \dd x \geq \int_{Z} |Ax| \dd x.
\end{align}
To derive the desired estimate, we determine the singular values of $A$. It follows from the specific structure of $A$ that
\begin{align*}
A^T A = 2\mathbb{I} - (R + R^T).
\end{align*}
Considering that every $R\in SO(n)$ can be transformed into block diagonal form with the help of another rotation $U\in SO(n)$, i.e.~there is an integer $k\leq \frac{n}{2}$ and two-dimensional rotations $\Theta_1, \dots, \Theta_{k} \in SO(2)$ such that 
\begin{align*} 
R = U^T\, {\rm diag}(\Theta_1, \dots, \Theta_{k}, \Ibb_{n-2k}) \,U,
\end{align*}
see e.g.~\cite[Satz 8.3.10]{KoM03}, we conclude from the fact that the symmetric part of a two-dimensional rotation matrix is diagonal that $A^TA  = U^T D U$, where
\begin{align*}
D= 2\, {\rm diag}(1-\theta_{1}, 1-\theta_{1}, \dots, 1-\theta_{k}, 1-\theta_{k}, 0, \ldots, 0) \in \R^{n\times n}
\end{align*}
with $\theta_i=(\Theta_i)_{11}\in [-1,1)$. One may assume without loss of generality that $\theta_1\leq \theta_2\leq \ldots \leq \theta_k$, which implies that $2(1-\theta_1)$ is the largest eigenvalue of $A^TA$, and hence corresponds to the squared spectral norm of $A$. Since all norms on $\R^{n\times n}$ are equivalent, there is a constant $C=C(n)>0$ such that $\sqrt{2(1-\theta_1)}\geq c |A|$, where $|\cdot|$ denotes the Frobenius norm.  
Hence,
\begin{align}\label{est_rotation2}
\int_{Z} |Ax|\dd x &= \int_Z \sqrt{A^TA x\cdot x}\dd x  = \int_{Z} \sqrt{D(Ux)\cdot Ux} \dd x \nonumber \\ & \geq \sqrt{2(1-\theta_1)} \int_{UZ} \sqrt{x_1^2+x_2^2} \dd x \geq C |A| \int_{UZ} \sqrt{x_1^2+x_2^2}\dd x.
\end{align}

In view of~\eqref{estcentering} and~\eqref{est_rotation2} it remains to show that
\begin{align}\label{int}
\int_{UZ} \sqrt{ x_1^2+x_2^2}\dd x \geq Cl\,|P|
\end{align}
with $C>0$ depending only on $n$.
If $U=\Ibb$, we simply neglect one of the two additive terms in the integrand, say $x_2^2$, and estimate that
\begin{align}\label{est98}
\int_Z |x_1|\dd{x} \geq \int_{\widehat{P}} |x_1| \dd x = 2 h \Bigl(\frac{l}{\sqrt{n}}\Bigr)^{n-2} \int_0^{\tfrac{l}{2\sqrt{n}}} x_1 \dd x_1= h\Bigl(\frac{l}{\sqrt{n}}\Bigr)^n = n^{-\frac{n}{2}}l\,|P|.
\end{align}
For general $U$, our argument requires to select a suitable rotation of the plane spanned by the unit vectors $e_1$ and $e_2$ to guarantee that the axes of the rotated cylinder $UZ$ is orthogonal to $e_1$. 
More precisely, one observes that any planar rotation $S={\rm diag}(\Sigma, \Ibb_{n-2})$ with $\Sigma\in SO(2)$ leaves the integral in~\eqref{int} unchanged, and therefore 
\begin{align}\label{est_UZ}
\int_{UZ} \sqrt{x_1^2+x_2^2} \dd x = \int_{SUZ} \sqrt{x_1^2+x_2^2} \dd x  \geq \int_{SUZ} |x_1|\dd{x}. 
\end{align}
Since the intersection of ${\rm span}\{e_1, e_2\}$ with the $(n-1)$-dimensional orthogonal complement of ${\rm span}\{Ue_n\}$ is at least a one-dimensional subspace, we can choose a planar rotation $S$ such that $Ue_n\cdot S^Te_1=0$, and thus $(SU)^Te_1\cdot e_n=0$.
Then there exists $Q = {\rm diag}(\Xi, \Ibb_1)\in SO(n)$ with $\Xi\in \SO(n-1)$ such that $Q^Te_1=(SU)^Te_1$, and
\begin{align*}
\int_{SUZ} |x_1|\dd{x} & =\int_Z |SU x\cdot e_1|\dd x = \int_{Z} |Qx\cdot e_1| \dd x  = \int_{QZ} |x_1| \dd x  = \int_{Z} |x_1|\dd x,
\end{align*}
where we have used the invariance of the cylinder $Z$ is invariant under rotations that leave the $x_n$-component unaffected.
In view of~\eqref{est_UZ} and~\eqref{est98} this shows~\eqref{int}, and hence, finishes the proof. 
\end{proof}

Next we give details on the extension result for locally one-dimensional functions in $e_n$-direction used in Sections~\ref{sec:proof} and~\ref{sec:suff}.
Recall that for a bounded domain $\Omega\subset \R^n$, the smallest open cuboid containing $\Omega$ is denoted by $Q_\Omega$ and  $Q_\Omega = O_\Omega \times J_\Omega$ with $O_\Omega\subset \R^{n-1}$ an open cuboid and an open interval $J_\Omega\subset\R$.

\begin{lemma}\label{lem:extension}
Let $\Omega\subset\R^n$ be a bounded, flat and cross-section connected Lipschitz domain. If $v\in W^{1, p}(\Omega;\R^m)$ satisfies $\nabla' v=0$, then $v$ can be extended to $Q_\Omega$ by a globally one-dimensional function in $e_n$-direction $\tilde v\in W^{1,p}(Q_\Omega;\R^m)\cap C^0(Q_\Omega;\R^m)$. 

In particular, one can identify $v$ with the one-dimensional function $\nu\in W^{1,p}(J_\Omega;\R^m)$ defined by the identity $\tilde v(x)=\nu (x_n)$ for $x\in Q_\Omega$.
\end{lemma}

\begin{proof} 
As pointed out at the beginning of Section~\ref{sec:proof}, $v$ is locally one-dimensional in $e_n$-direction, and hence, locally constant on any non-empty cross section $\Omega_t=H_t\cap \Omega = \{x\in \R^n: x_n=t\}\cap\Omega$. Since the latter are connected by assumption, it follows that $v$ is also globally one-dimensional in $e_n$-direction. 

We can now define an extension $\tilde{v}$ of $v$ to $Q_\Omega$ by setting
\begin{align}\label{def_tildev}
\tilde{v}(x)=v(y) \quad \text{with $y\in \Omega_{x_n}$}
\end{align}
for $x\in Q_\Omega$.
Observe that with $Q_\Omega$ the smallest open cuboid such that $\Omega\subset Q_{\Omega}$, the intersection $H_{x_n}\cap\Omega =\Omega_{x_n}$ is non-empty for all $x\in \Omega$.
Clearly, $\tilde{v}$ is globally one-dimensional in $e_n$-direction by definition. It therefore remains to prove that $\tilde{v}\in W^{1,p}(Q_\Omega;\R^m)$ (for continuity one can then argue as in the first paragraph of Section~\ref{sec:proof}). 

To see this we will construct a sequence $w_j\in C^\infty(\overline{Q_\Omega};\R^m)$ that approximates $\tilde{v}$ in $W^{1,p}(Q_{\Omega};\R^m)$. Let 
$J_\Omega=(a, b)$ with $a, b\in \R$  $a<b$. 
Since $\Omega$ is a flat Lipschitz domain there exist $x_a\in \Omega_a$ and $x_b\in \Omega_b$ and balls $B_r(x_a)$ and $B_r(x_b)$ with radius $r>0$ such that $B_r(x_a)\cap Q_\Omega \subset \Omega$ and $B_r(x_b)\cap Q_{\Omega}\subset \Omega$. 
Exploiting further that $\Omega$ is open and connected, hence also path-connected, we can connect the edge points $x_a$ with $x_b$ by a $C^1$-curve $\gamma$ (after smoothing of a continuous curve).
Moreover, one can be chosen $\gamma$ to be monotone in $x_n$ due to the cross-section connectedness of $\Omega$ and even strictly monotone, which implies that $\gamma$ is a regular curve, considering that $\Omega$ is open. 
After reparametrization we obtain 
\begin{align}\label{curve_gamma}
\gamma \in C^1([a, b];\R^n) \quad \text{with $\gamma(t) \in \Omega_t$ for all $t\in [a, b]$.} 
\end{align}

For the composition $w=v\circ \gamma\in W^{1,p}(J_\Omega;\R^m)$ there exist approximating functions $w_j\in C^\infty(\overline{J_\Omega};\R^m)$ such that $w_j\to w$ in $W^{1,p}(a,b)$. Without changing notation, let us identify $w_j$ and $w$ with their constant expansion in $x'$, that is with elements in $W^{1,p}(Q_\Omega;\R^m)$ and $C^\infty(\overline{Q_\Omega};\R^m)$, respectively. Finally, in view of~\eqref{def_tildev} and~\eqref{curve_gamma},
\begin{align*}
w_j \to  w= v\circ \gamma = \tilde{v}  \quad \text{in $W^{1,p}(Q_\Omega;\R^m)$},
\end{align*}
which shows that $\tilde{v}\in W^{1,p}(Q_\Omega;\R^m)$ and concludes the proof. 
\end{proof}

\begin{remark}\label{rem:extension}
a) Since only local arguments have been used in the proof above, Lemma~\ref{lem:extension} still holds if $W^{1,p}(\Omega;\R^m)$ is replaced with $W^{1,p}_{\rm loc}(\Omega;\R^m)$. In this case, it is even enough to require that $\Omega\subset\R^n$ is a bounded, cross-section connected domain.\\[-0.2cm]

b) As Lemma~\ref{lem:extension} relies on constant extensions only, changing the codomain of $v$ from $\R^m$ to $SO(n)$ does not change the statement. 
\end{remark}
\end{appendix}

\section*{Acknowledgements}
The authors would like to thank Georg Dolzmann for his valuable comments on FC's PhD thesis, which helped to improve also the presentation of this manuscript. FC gratefully acknowledges a traveling grant by the DFG Graduiertenkolleg 1692 ``Curvature, Cyles, and Cohomology''. CK was partially supported by a Westerdijk Fellowship from Utrecht University. 

\bibliographystyle{abbrv}
\bibliography{Homogenization}

\end{document}